\documentclass{amsart}

\usepackage{Preset}
\usepackage[makeroom]{cancel}
\usepackage{MnSymbol}

\usepackage[latin1]{inputenc}
\usepackage[T1]{fontenc}
\usepackage{tikz}
\usepackage[normalem]{ulem}
\usepackage[colorlinks=true,linkcolor=blue!50!black,anchorcolor=red,citecolor=blue!75!black,filecolor=black,menucolor=black,runcolor=black,urlcolor=black]{hyperref}
\allowdisplaybreaks[2]

\usetikzlibrary{shapes,arrows,graphs}

\tikzstyle{special} = [rectangle, draw, 
    text width=10em, text centered, fill=green!10, minimum height=4em]
\tikzstyle{process} = [rectangle, draw, 
    text width=4em, text centered, fill=gray!05, rounded corners, minimum height=2em]
\tikzstyle{block} = [rectangle, draw,
    text width=10em, text centered, fill=orange!10, rounded corners, minimum height=4em]
\tikzstyle{line} = [draw, -latex']
\tikzstyle{cloud} = [draw, ellipse, node distance=3cm,
    minimum height=2em]

\usepackage{letltxmacro}
\newcommand\textify[1]{%
  \expandafter\GlobalLetLtxMacro\csname textify@\string#1\endcsname#1%
  \protected\gdef#1{\ensuremath{\csname textify@\string#1\endcsname}}%
}

\newcommand{\drawHQ}[1]{\operatorname{
\begin{tikzpicture}[#1]
\draw[very thick] (0,0) circle (0.8ex);
\draw[white, thick] (0,0) circle (0.1ex);
\end{tikzpicture}
}}

\newcommand{\drawHR}[1]{\operatorname{
\begin{tikzpicture}[#1]
\draw[very thick] (0,0) circle (0.85ex);
\draw[very thick] (0,0) circle (0.7ex);
\draw[very thick] (0,0) circle (0.6ex);
\draw[very thick] (0,0) circle (0.5ex);
\draw[very thick] (0,0) circle (0.4ex);
\end{tikzpicture}
}}

\newcommand{\drawBC}[1]{\operatorname{
\begin{tikzpicture}[#1]
\draw[thick] (0,0) circle (0.85ex);
\draw[very thin] (0,0) circle (0.3ex);
\end{tikzpicture}
}}

\newcommand{\newHQ}{\ensuremath \raisebox{-1pt}{$\drawHQ{}$}}
\newcommand{\newHR}{\ensuremath \raisebox{-1pt}{$\drawHR{}$}}
\newcommand{\newBC}{\ensuremath \raisebox{-1pt}{$\drawBC{}$}}
\textify\newHQ
\textify\newHR
\textify\newBC

\makeatletter
\def\namedlabel#1#2{\begingroup
   \def\@currentlabel{#2}%
   \label{#1}\endgroup
}
\makeatother

\makeatletter
\newcommand*{\wackyenum}[1]{%
  \expandafter\@wackyenum\csname c@#1\endcsname%
}
\newcommand*{\@wackyenum}[1]{%
  $\ifcase#1\or \newHQ \or \newHR \or \newBC \or idk%
    \else\@ctrerr\fi$%
}
\AddEnumerateCounter{\wackyenum}{\@wackyenum}{(\textbf{BC})}
\makeatother

\setlist[description]{leftmargin=\parindent,labelindent=\parindent}

\title[A priori bounds and degeneration of Herman rings]{A priori bounds and degeneration of Herman rings with bounded type rotation number}
\author{Willie Rush Lim}
\address{Dept. of Mathematics, Brown University, Providence, RI 02912}
\email{willie\_rush\_lim@brown.edu}
\date{}

\begin{document}

\begin{abstract}
By adapting the near-degenerate regime designed by Kahn, Lyubich, and D. Dudko, we prove that the boundaries of Herman rings with bounded type rotation number and of the simplest configuration are quasicircles with dilatation depending only on the degree and the rotation number. As a consequence, we show that these Herman rings always degenerate to a \emph{Herman curve}, i.e. an invariant Jordan curve that is not contained in the closure of a rotation domain and on which the map is conjugate to a rigid rotation. This process enables us to construct the first general examples of rational maps having Herman curves of bounded type with arbitrary degree and combinatorics. In particular, they do not come from Blaschke products. We also demonstrate the existence of Renormalization Theory for Herman curves by constructing rescaled limits of the first return maps in the unicritical case.
\end{abstract}

\maketitle

\renewcommand{\ULdepth}{1.8pt}
\setcounter{tocdepth}{1}
\tableofcontents

\section{Introduction}
\label{sec:intro}

\subsection{On Herman rings} 
\label{ss:on-herman-rings}

By Fatou's classification of periodic Fatou components, every maximal domain $U \subset \RS$ in which a rational map $f: \RS \to \RS$ of degree $d\geq 2$ is conformally conjugate to a rigid rotation is either a topological disk, in which case $U$ is called a Siegel disk, or an annulus, in which case $U$ is called a Herman ring.

Siegel disks have been actively studied in the last few decades. In the second half of the last century, the study of local dynamics near a neutral fixed point has essentially received a complete treatment by the works of Brjuno, Herman, Yoccoz, and Perez-Marco. At the same time, the semi-local theory for Siegel disks of quadratic maps $e^{2\pi i \theta}z+z^2$ started with the introduction of Douady-Ghys surgery \cite{D87,G84}, which was based on the work of Herman and \'Swi\k{a}tek \cite{H86,S88}. The surgery procedure proves that Siegel disks of quadratic maps with bounded type rotation number are quasidisks containing a critical point on the boundary. The above result was generalized by Zakeri \cite{Z99} for cubic polynomials, by Shishikura for polynomials of arbitrary degree, and ultimately by Zhang \cite{Z11} for rational maps. Moreover, Zhang proved \emph{a priori bounds} for Siegel disks: the dilatation of the boundary of every invariant Siegel disk with bounded type rotation number $\theta$ of a rational map $f$ depends only on the degree of $f$ and the bound 
\[
\beta(\theta) := \max_i a_i < \infty
\]
of the continued fraction expansion $[0;a_1, a_2 ,\ldots ]$ of $\theta$.

The absence of periodic points associated to Herman rings makes them more difficult to study than Siegel disks. The construction of the first examples of Herman rings was based on the study of linearizability of analytic circle diffeomorphisms by Arnol'd and Herman. A more general construction was later established by Shishikura. In \cite{S87}, Shishikura developed surgery procedures to construct Herman rings out of two Siegel disks, and to convert Herman rings into Siegel disks. Combining this surgery and Zhang's results, every boundary component of a bounded type invariant Herman ring $\He$ of a rational map $f$ is a quasicircle containing a critical point and having dilatation that depends only on the degree of $f$, the rotation number, and the modulus of $\He$. We will develop the machinery to remove the dependence of the modulus and obtain \emph{a priori bounds}.

\begin{definition}
\label{main-definition}
    We define $\HRspace_{d_0, d_\infty,\theta}$ to be the space of all degree $d_0 + d_\infty -1$ rational maps $f$ such that 
\begin{enumerate}[label=\text{(\Roman*)}]
    \item\label{def:1} $0$ and $\infty$ are superattracting fixed points of $f$ with local degrees $d_0\geq 2$ and $d_\infty\geq 2$ respectively;
    \item\label{def:2} the map $f$ admits an invariant Herman ring $\He$ with a bounded type rotation number $\theta$;
    \item\label{def:3} $\He$ separates $0$ and $\infty$;
    \item\label{def:4} every critical point of $f$ other than $0$ and $\infty$ lies on the boundary of $\He$.
\end{enumerate}
\end{definition}

The space $\HRspace_{d_0, d_\infty,\theta}$ encapsulates general Herman rings of the simplest configuration that can be obtained from Shishikura's surgery: they can be constructed out of two polynomials having unique invariant Siegel disks satisfying conditions similar to \ref{def:4}. The existence and rigidity of maps in $\HRspace_{d_0, d_\infty,\theta}$ of any prescribed combinatorics are guaranteed by a Thurston-type result by Wang \cite{W12}. The following is our main theorem. 

\begin{thmx}[\emph{A priori bounds}]
\label{main-theorem-01}
    The boundary components of the Herman ring of every map in $\HRspace_{d_0, d_\infty,\theta}$ are quasicircles with dilatation depending only on $d_0$, $d_\infty$, and $\beta(\theta)$.
\end{thmx}

Most of this paper is dedicated to the proof of Theorem \ref{main-theorem-01}, which will be achieved in the {\bf Near-Degenerate Regime}. A more detailed summary of our proof will be provided in the outline \S\ref{ss:outline}. The idea that compactness results are amenable for near-degenerate surfaces goes back to the work of W. Thurston on the geometry of $3$-manifolds. (See, for instance, the Double Limit Theorem \cite{Th86}.) In complex dynamics, the near-degenerate regime was successfully implemented in the proof of W. Thurston's characterization of postcritically finite rational maps \cite{DH93}. In mid 2000's, Kahn \cite{K06} introduced the near-degenerate regime to the Renormalization Theory of quadratic-like maps. Together with Lyubich, they set up fundamental tools, such as the Quasi-Additivity Law and the Covering Lemma \cite{KL05}, and attained substantial progress in the primitive case of the MLC conjecture \cite{KL08,KL09a}. Other applications of the Covering Lemma include the extension of Yoccoz's results and puzzle-parapuzzle machinery to higher degrees \cite{KL09b,AKLS,KvS,ALS}. (See also \cite{CDKvS} for a detailed exposition.) Recently, D. Dudko and Lyubich \cite{DL22} transferred the near-degenerate regime to neutral dynamics of quadratic polynomials $e^{2\pi i \theta} z+z^2$: they constructed \emph{pseudo-Siegel disks} out of bounded type Siegel disks by filling in fjords at all scales, and showed that these are quasidisks with uniform dilatation. Even though the above instances of the near-degenerate regime are unified by the same general principle, they have little in common on the technical level.

\subsection{On Herman curves} 
\label{ss:on-herman-curves}
The interest in studying invariant curves of a holomorphic dynamical system can be traced back to Fatou's memoirs \cite{F20} in 1920. One hundred years later, Eremenko posed a question at the online conference ``On Geometric Complexity of Julia Sets II`` \cite{E20} on the existence of non-trivial Herman curves. This object is defined as follows.

\begin{definition}
\label{herman-curve-defn}
    Given a rational map $f$, a forward invariant Jordan curve $\Hq$ is called a \emph{Herman curve} of $f$ of rotation number $\theta$ if $\Hq$ is not contained in the closure of any rotation domain and $f|_\Hq$ is conjugate to the rigid rotation $R_\theta$ by angle $\theta$ on the unit circle. Additionally, $\Hq$ is called a \emph{Herman quasicircle} if it is a quasicircle.
\end{definition}

The first known example of a Herman curve is the unit circle $\T$. If a rational map leaves a round circle invariant, it must be conformally conjugate to a Blaschke product. One explicit example is as follows. For any irrational $\theta \in (0,1)$, there is a unique $t_{\theta} \in [0,1)$ such that the Blaschke product
\begin{equation}
\label{cubic-blaschke-product}
    B_{\theta}(z) = e^{2\pi i t_{\theta}} z^2 \frac{z-3}{1-3z}
\end{equation}
restricts to a self homeomorphism of $\T$ with rotation number $\theta$. This map was originally studied by Douady \cite{D87} in his contribution towards showing that the boundary of an invariant quadratic Siegel disk with bounded type rotation number is a quasicircle containing the critical point.

Eremenko asked whether or not there exist Herman curves that are not round circles. One can perform quasiconformal surgery on $B_\theta$ to replace the multiplier of one of the superattracting fixed points ($0$ or $\infty$) to any small non-zero complex number. The resulting rational map has a Herman quasicircle that is not a Euclidean circle due to asymmetry, but it is still quasiconformally conjugate to a Blaschke product near its Julia set. We say that a Herman curve is \emph{trivial} if it is a round circle or it can be obtained from a round circle via quasiconformal surgery.

Let us view $\HRspace_{d_0, d_\infty,\theta}$ as a subspace of the space $\rat_{d_0+d_\infty-1}$ of degree $d_0+d_\infty-1$ rational maps endowed with the topology of uniform convergence on compact subsets. Denote the corresponding limit space by 
\[
\HRlimitspace_{d_0, d_\infty,\theta} := \overline{\HRspace_{d_0, d_\infty,\theta}} \backslash \HRspace_{d_0, d_\infty,\theta} \subset \rat_{d_0 + d_\infty -1}.
\]
One consequence of Theorem \ref{main-theorem-01} is that as maps in $\HRspace_{d_0, d_\infty,\theta}$ approach $\HRlimitspace_{d_0, d_\infty,\theta}$, the corresponding Herman rings must degenerate to a Herman quasicircle.

\begin{corx}
\label{main-corollary}
    For every $f \in \HRlimitspace_{d_0, d_\infty,\theta}$, the following properties hold:
    \begin{enumerate}[label=\textnormal{(\roman*)}]
    \item\label{cor:1} $0$ and $\infty$ are superattracting fixed points of $f$ with local degrees $d_0\geq 2$ and $d_\infty\geq 2$ respectively;
    \item\label{cor:2} the function $f$ admits a Herman quasicircle $\Hq$ of rotation number $\theta$;
    \item\label{cor:3} $\Hq$ separates $0$ and $\infty$;
    \item\label{cor:4} every critical point of $f$ other than $0$ and $\infty$ lies in $\Hq$;
    \item\label{cor:5} the conjugacy between $f|_\Hq$ and $R_\theta|_\T$ is quasisymmetric with dilatation depending only on $d_0$, $d_\infty$ and $\beta(\theta)$.
\end{enumerate}
\end{corx}

The combinatorics of a rational map $f \in \HRlimitspace_{d_0, d_\infty,\theta}$ is encoded by the relative position and criticality of the critical points along the Herman quasicircle $\Hq$ of $f$. (See Definition \ref{def:combinatorics} for details.) Using Wang's result as well as the precompactness arising from \emph{a priori bounds}, we construct rational maps in $\HRlimitspace_{d_0, d_\infty,\theta}$ of any prescribed combinatorics by taking limits of degenerating Herman rings in $\HRspace_{d_0,d_\infty,\theta}$.

\begin{thmx}[Realization]
\label{main-theorem-02}
    Given any prescribed combinatorics, there exists a rational map in $\HRlimitspace_{d_0,d_\infty,\theta}$ having a Herman quasicircle that realizes such combinatorics.
\end{thmx}

In particular, by picking combinatorial data with critical asymmetry, we obtain a wealth of examples of non-trivial Herman curves. See Figures \ref{fig:asymmetric-herman-quasicircle} and \ref{fig:asymmetric-herman-quasicircle-02}.

\begin{figure}
    \centering
    \includegraphics[width=0.97\columnwidth]{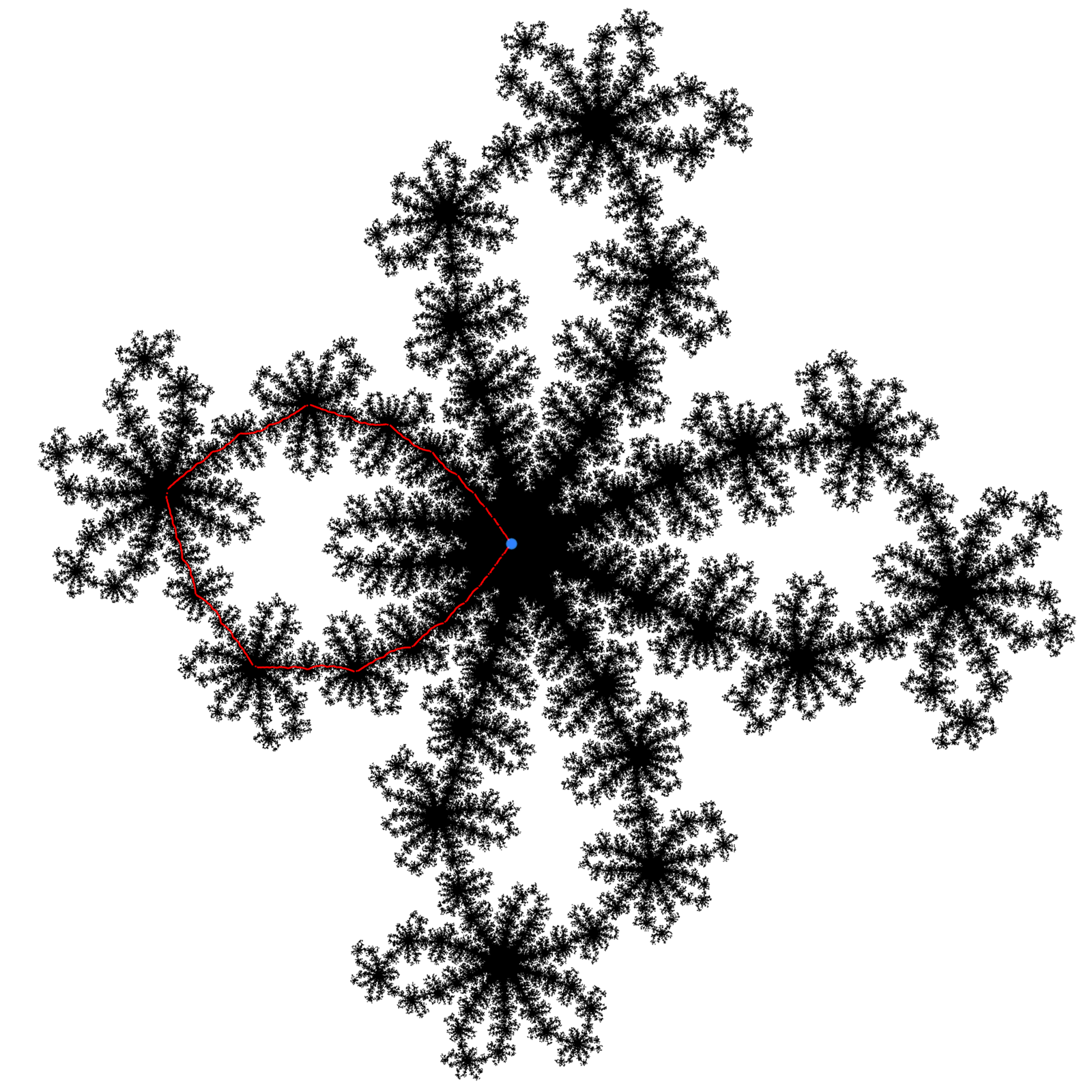}
    \caption{The Julia set of $f(z) =  c_* z^2 \: \dfrac{z^3-5z^2+10z-10}{1-5z}$. It has a unique free critical point at $1$, marked in blue, and its forward orbit is colored red. The critical value $c_* = f(1) \approx -0.386631-0.320505i$ is found numerically such that $f$ has a Herman quasicircle passing through $1$ of golden mean rotation number $\theta_* = \frac{\sqrt{5}-1}{2}$ and $f \in \HRlimitspace_{2,4,\theta_*}$. Refer to \S\ref{sec:unicritical-herman-curves} for further details.}
\label{fig:asymmetric-herman-quasicircle}
\end{figure}

\begin{figure}
    \centering
    \includegraphics[width=0.87\columnwidth]{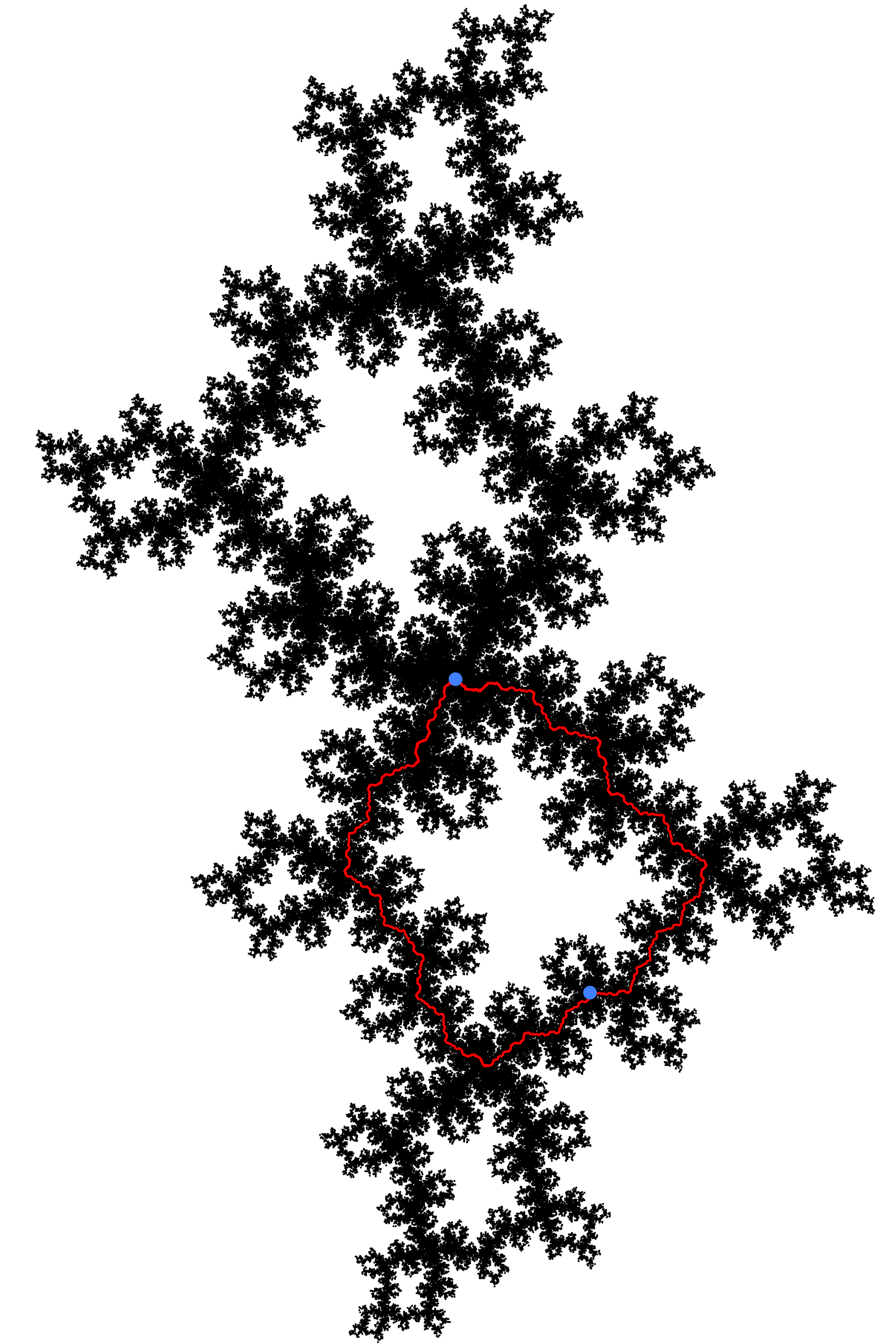}
    \caption{The Julia set of $f(z) =  z^2 \dfrac{q_*-z}{1+\overline{q_*}z}$. The map $f(z)$ is antipode preserving (commutes with $-1/\bar{z}$) and was first studied by Bonifant, Buff and Milnor in \cite{BBM18,BBM}. The two free antipodal critical points are marked in blue and their forward orbit is colored red. The constant $q_* \approx -1.26+2.94i$ is determined numerically such that $f$ admits a Herman quasicircle of golden mean rotation number $\theta_*$ and $f \in \HRlimitspace_{2,2,\theta_*}$.}
\label{fig:asymmetric-herman-quasicircle-02}
\end{figure}

During the time this paper was written, Yang Fei \cite{Y22} also gave an affirmative answer to Eremenko's question by proving the existence of a cubic rational map whose Julia set has positive area and contains a smooth Herman curve. In general, the bounded type Herman curves in $\HRlimitspace_{d_0,d_\infty,\theta}$ are not smooth due to the presence of critical points. Yang's construction assumes high type rotation number and adapts a perturbative method by Avila-Buff-Ch\'eritat \cite{ABC04} that was originally applied to show the existence of quadratic Siegel disks with smooth boundaries.

It is natural to ask the following questions:
\begin{enumerate}
    \item When is a limit of degenerating Herman rings a Herman curve?
    \item When is a Herman curve a limit of degenerating Herman rings?
\end{enumerate} 
We expect that most of the methods in this paper can be extended to larger classes of Herman rings. In particular, we conjecture that the limit of degenerating Herman rings of bounded type is always a Herman curve. A partial answer to (2) will be explored in a forthcoming paper \cite{Lim23}. 

At the end of this paper, we initiate the development of Renormalization Theory for unicritical Herman curves. We show the existence of rescaled limits of first return maps (Corollary \ref{rescaled-limits}) and state a conjecture on self-similarity in the associated parameter space (Conjecture \ref{conj:self-similarity}). See Figure \ref{fig:parspace2}. This serves as a motivation to further study the Renormalization Theory for Herman curves in the near future.

\subsection{Outline}
\label{ss:outline}

The proofs of our results hinge upon the near-degenerate machinery, including the Quasi-Additivity Law and the Covering Lemma \cite{KL05}. For the reader's convenience, the details are supplied in the appendix. Many of the steps in our proof are inspired by \cite{K06,DL22}.

As we will be working with Herman rings of arbitrarily small modulus $\mu\ll 1$, we will also consider the space $\HQspace_{d_0, d_\infty,\theta}$ of rational maps having a Herman quasicircle $\Hq$; this is axiomatically defined in a similar fashion as $\HRspace_{d_0, d_\infty,\theta}$ except that properties \ref{def:2}--\ref{def:4} are replaced by \ref{cor:2}--\ref{cor:4} in Corollary \ref{main-corollary}. We will work with a fixed map $f$ in $\HQspace_{d_0, d_\infty,\theta} \cup \HRspace_{d_0, d_\infty,\theta}$ and denote by $\Hq$ either the Herman quasicircle (Case \ref{case:herman-curve}), the closure of the Herman ring $\He$ (Case \ref{case:herman-ring}, a fat Herman curve), or a boundary component of $\He$ (Case \ref{case:boundary-ring}). Basic properties of $\Hq$ are covered in Section \S2.

We adapt the vocabulary in \cite{DL22} to represent local degeneration as quantities $W_\alpha(I)$ defined as follows. (Refer to \S\ref{ss:irrational-rotations} and \S\ref{ss:encoding-degeneration}.) Let $\phi$ denote the linearization of $f|_\Hq$. A \emph{piece} $I$ in $\Hq$ is the preimage of a closed interval $I'$ on the unit circle $\T := \R / \Z$ under the mapping $z \mapsto \frac{1}{2\pi}\arg(\phi(z))$. The \emph{combinatorial length} $|I|$ of a piece $I$ is defined to be the Euclidean length of $I'$. For $\alpha \in (1, |I|^{-1})$, the \emph{$\alpha$-width} $W_\alpha(I)$ of $I$ is the extremal width of the family of curves that connect $I$ and $\overline{\Hq \backslash \alpha I}$, where the piece $\alpha I$ is an enlargement of $I$ by a factor of $\alpha$. We say that a piece $I$ is \emph{$[K,\alpha]$-wide} if $W_\alpha(I)\geq K$.

The key to \emph{a priori bounds} is the Amplification Theorem \ref{amplification} which states that the existence of a $[K,10]$-wide piece $I \subset \Hq$ implies the existence of a $[2 K,10]$-wide piece, where $K$ is sufficiently large depending only on $d_0$, $d_\infty$, and $\beta(\theta)$. Our analysis is split into two cases. (See \S\ref{sss:modulus-Siegel-scale} and \S\ref{sec:a-priori-bounds}.)
\begin{description}
    \item[Herman scale] $|I| > \mu$, (the \textbf{main case}, roughly \ref{case:herman-ring} and \ref{case:herman-curve})
    \item[Siegel scale] $|I| \leq \mu.$ (roughly Case \ref{case:boundary-ring})
\end{description}
In the Siegel scale, this theorem is similar to (and was inspired by) \cite[Theorem 8.1]{DL22} in the context of quadratic Siegel disks. In the Herman scale, the techniques in \cite{DL22}, especially \cite[Snake Lemma 2.12]{DL22}, are not applicable because, unlike in the Siegel scale, the geometry on both sides of $I$ is unknown. In Sections \S\ref{sec:bubble-wave-argument}--\ref{sec:loss-of-horizontal-width}, we develop the fundamental results necessary to prove the Amplification Theorem.

In Section \S\ref{sec:bubble-wave-argument}, we discuss the notion of waves. In Proposition \ref{prop:wave-degeneration}, we show that large waves generate large $10$-width at a \emph{shallow} level, i.e. when $|I|\asymp 1$. One of the main ideas is to use the fact that \emph{bubbles}, i.e. preimages of $\Hq$, up to a certain generation that are attached to $\Hq$ have substantial harmonic measure about either $0$ or $\infty$ (\ref{s3.claim2} in the proof of Lemma \ref{bubble-wave-argument}).

Section \S\ref{sec:spreading-degeneration} discusses how to spread degeneration by rotating it along $\Hq$. Let $\frac{p_n}{q_n}$ be the best rational approximations of $\theta$. A piece $I \subset \Hq$ is a \emph{combinatorial piece} of level $n$ if it has endpoints $x$ and $f^{q_n}(x)$ for some point $x \in \Hq$. A level $n$ \emph{almost tiling} $\mathcal{I}$ is a collection of pieces with disjoint interiors of the form $\{f^i(I)\}_{0\leq i \leq q_{n+1}-1}$ for some level $n$ combinatorial piece $I$. By applying the Covering Lemma, we show in Proposition \ref{spreading-lambda} that for any $\Xi>1$, $\lambda \gg 10$, and $K \gg_{\Xi,\lambda} 1$, the existence of a $[K,\lambda]$-wide combinatorial piece implies the existence of either a $[\Xi K,10]$-wide piece or an almost tiling $\mathcal{I}$ consisting of $[\xi K,\lambda]$-wide pieces for some $\xi = \xi(\Xi)>0$.

The main result in Section \S\ref{sec:trading} is Theorem \ref{demotion} which states that for $K \gg_\lambda 1$, the existence of an almost tiling consisting of $[K,\lambda]$-wide pieces implies the existence of a $[\Pi_\lambda K,10]$-wide piece where $\Pi_\lambda \to \infty$ as $\lambda \to \infty$. The proof is split into two cases: the deep case, where the level of the almost tiling is high, and the shallow case, where the level is low. The deep case is an application of the Quasi-Additivity Law, whereas the shallow case uses the wave mechanism from Section \S\ref{sec:bubble-wave-argument}.

In Section \S\ref{sec:amplifying-tau-degeneration}, we prove in Theorem \ref{promotion} that for $\lambda \gg 10$ and $K\gg_\lambda 1$, the existence of a $[K,10]$-wide combinatorial piece induces the existence of a piece that is either $[2 K,10]$-wide or $[\chi K,\lambda]$-wide, where $0<\chi<1$ is independent of $\lambda$. Since the shallow case can again be handled via waves, we are left with the deep case. Our main strategy is to adapt Kahn's push-forward argument in \cite[\S7]{K06} to our setting. A key ingredient in the original push-forward argument is the positivity of the core entropy corresponding to primitive renormalization, which stands in contrast to the lack of entropy of the rotational action of $f$ on $\Hq$. Section \S\ref{sec:loss-of-horizontal-width} is dedicated to developing a replacement for Kahn's entropy argument, namely Proposition \ref{pos-entropy}. Due to technical considerations, we supply a more detailed outline in \S\ref{sss:rough-outline}.

Finally, the proof of the Amplification Theorem \ref{amplification} is an application of Theorems \ref{demotion} and \ref{promotion}. In short, we will eventually pick $\lambda$ to be large enough such that the constant $\Pi_\lambda$ beats the constant $\chi$. This is captured in Figure \ref{fig:implication-diagram}.

In Section \S\ref{sec:construction-of-herman-curves}, we discuss some consequences of \emph{a priori bounds}. Firstly, we show that the space of Herman rings in $\HRspace_{d_0, d_\infty,\theta}$ with modulus bounded from above, up to conformal equivalence, is precompact in the topology of locally uniform convergence. Secondly, we prove Corollary \ref{main-corollary} by showing that $\HRlimitspace_{d_0, d_\infty,\theta} \subset \HQspace_{d_0, d_\infty,\theta}$. Both of these allow us to prove Theorem \ref{main-theorem-02} as follows. Starting with any rational map $f_1 \in \HRspace_{d_0, d_\infty,\theta}$ with a Herman ring $\He_1$, we can quasiconformally deform the complex structure of $\He_1$ and its grand orbit to obtain a normalized family of rational maps $\{f_t \}_{0<t \leq 1}$ in $\HRspace_{d_0, d_\infty,\theta}$ depending real analytically on $t$ such that each $f_t$ has the same combinatorics as $f_1$ and that the modulus of the Herman ring of $f_t$ converges to $0$ as $t\to 0$. By precompactness, there is a limit $f_0$ of $\{f_t \}_{0<t\leq 1}$ as $t \to 0$. By Corollary \ref{main-corollary}, $f_0$ must have a Herman quasicircle with the same combinatorics as $f_1$. Roughly speaking, $f_0$ is the 'quotient' of $f_1$ obtained by collapsing each leaf of the radial foliation of $\He_1$ and its preimages.

In the final section, we calculate an explicit formula for unicritical maps in $\HQspace_{d_0,d_\infty,\theta}$. In Corollary \ref{rescaled-limits}, we also show that the sequence of rescaled first return maps near the critical point is precompact in the Carath\'eodory topology.

\subsection{Acknowledgements}
\label{ss:acknowledgements}
I would like to thank my advisor Dzmitry Dudko for suggesting this problem and for his relentless support and invaluable advice. I would also like to thank Mikhail Lyubich for helpful advice, and Araceli Bonifant and John Milnor for illuminating discussion on their work in \cite{BBM} with Xavier Buff. This project has been partially supported by the NSF grant DMS 2055532 and by Simons Foundation International, LTD. The results of this paper also appear in the PhD thesis \cite{LimThesis}.


\section{Preliminaries}
\label{sec:prelim}
Fix a pair of positive integers $d_0, d_\infty \geq 2$ and an irrational number $0<\theta<1$ with continued fraction expansion $\theta = [0;a_1,a_2, \ldots]$. Unless otherwise stated, we will always assume that $\theta$ is of bounded type, that is, the maximum $\beta(\theta):= \max_{i} a_i < \infty$ is well-defined.

\subsection{Herman rings}
\label{ss:herman-rings}
The following procedure allows one to obtain Siegel disks out of invariant curves.

\begin{theorem}[Douady-Ghys surgery]
    Let $f: \RS \to \RS$ be a rational map, $Y \subset \RS$ be a quasidisk such that $\partial Y$ is forward invariant and $f|_{\partial Y}$ is quasisymmetrically conjugate to an irrational rotation $R_\theta$ of the circle $\T$. There exists a $K$-quasiconformal map $\phi: \RS \to \RS$ and a rational map $F$ such that
    \begin{enumerate}[label=\textnormal{(\arabic*)}]
        \item $F = \phi \circ f \circ \phi^{-1}$ on $\RS \backslash \phi(Y)$, and
        \item $F$ has a Siegel disk of the same rotation number $\theta$ containing $\phi(Y)$.
    \end{enumerate}
    Moreover, $K$ depends only on the dilatation of the conjugacy between $f|_{\partial Y}$ and $R_\theta$.
\end{theorem}

The original idea of the surgery procedure was by Ghys in \cite{G84}, but the formulation above follows from \cite[\S7.2]{BF14}. The essence of the surgery procedure is to replace the dynamics $f|_Y$ with a rotation. More precisely, we replace $f|_Y$ with $\psi^{-1} \circ R_\theta \circ \psi$, where $\psi: Y \to \D$ is a quasiconformal extension of the quasisymmetric conjugacy between $f|_{\partial Y}$ and $R_\theta|_{\partial \D}$, and straighten the new map via the measurable Riemann mapping theorem. 

As explained in the introduction, Douady-Ghys surgery plays an essential role in deducing the regularity of the boundary of Siegel disks with bounded type rotation number. The most general version of this result is the following theorem.

\begin{theorem}[\cite{Z11}]
    Let $f$ be a rational map of degree $d \geq 2$. If $f$ has an invariant Siegel disk $Z$ with bounded type rotation number $\theta$, then the boundary $\partial Z$ is a $K(d,\beta(\theta))$-quasicircle containing at least one critical point.
\end{theorem}

In \cite[\S6]{S87}, Shishikura originally discovered a way to convert Herman rings into Siegel disks (and vice versa) through quasiconformal surgery. We will formulate this procedure as a straightforward application of Douady-Ghys surgery and combine it with Zhang's theorem to obtain the following corollary.

\begin{corollary}
\label{old-bounds}
    Let $f$ be a rational map of degree $d \geq 3$ having an invariant Herman ring $\He$ with bounded type rotation number $\theta$ and modulus $\modu(\He) \geq \mu > 0$. Then,
    \begin{enumerate}[label=\textnormal{(\arabic*)}]
        \item every boundary component of $\He$ is a $K$-quasicircle containing at least one critical point;
        \item there is an $L$-quasiconformal map $\phi : \RS \to \RS$ that is conformal in $\He$ and conjugates $f|_{\overline{\He}}$ and the rigid rotation $R_\theta$ on the annulus $\{ 1\leq |z| \leq e^{2\pi\modu(\He)} \}$.
    \end{enumerate}
    Moreover, the dilatations $K$ and $L$ depend only on $d$, $\beta(\theta)$, and $\mu$.
\end{corollary}

\begin{proof}
    By Koebe distortion theorem, along the core curve $\gamma$ of $\He$, $f|_\gamma$ must be $K'$-quasisymmetrically conjugate to $R_\theta|_\T$ for some $K'=K'(\mu)$. Pick a boundary component $H$ of $\He$ and let $D$ be the component of $\RS \backslash \gamma$ containing $H$. Apply Douady-Ghys surgery along $\gamma$ to obtain a degree $\leq d-1$ rational map $F$ having an invariant Siegel disk $Z$ and an $L'$-quasiconformal map $\psi:\RS\to\RS$ that maps $H$ to $\partial Z$ and restricts to a conjugacy between $f|_{D}$ and $F|_{\psi(D)}$, where $L'$ depends on $\mu$. Then, the corollary follows from applying Zhang's theorem to $F$.
\end{proof}

In this paper, we would like to remove the dependency on the modulus $\mu$ for one of the simplest families of rational maps with Herman rings, namely $\HRspace_{d_0,d_\infty,\theta}$ defined in the introduction. Rational maps in $\HRspace_{d_0, d_\infty,\theta}$ can be constructed through Shishikura's quasiconformal surgery \cite{S87} (see also \cite[\S7.3]{BF14}) from two polynomials $P_0$ and $P_\infty$ of degree $d_0$ and $d_\infty$ respectively satisfying the following conditions:
\begin{itemize}
    \item[$\rhd$] $P_0$ and $P_\infty$ have invariant Siegel disks $Z_0$ and $Z_\infty$ of rotation numbers $1-\theta$ and $\theta$ respectively;
    \item[$\rhd$] the only non-repelling periodic points of $P_0$ and $P_\infty$ are the centers of $Z_0$ and $Z_\infty$;
    \item[$\rhd$] all finite critical points of $P_0$ and $P_\infty$ lie on $\partial Z_0$ and $ \partial Z_\infty$ respectively.
\end{itemize}

The surgery involves removing a proper invariant sub-disk of each $Z_0$ and $Z_\infty$, gluing the two remaining Riemann surfaces along the boundary of the sub-disks and applying the measurable Riemann mapping theorem to obtain some $f \in \HRspace_{d_0, d_\infty,\theta}$ that mimics the dynamics of both $P_0$ and $P_\infty$ outside of the removed disks. 

\begin{remark}
    Theorems by Wang (see Theorem \ref{wang} below) and Zhang in \cite{Z08} guarantee that every Herman ring in $\HRspace_{d_0,d_\infty,\theta}$ can be constructed from the surgery above. 
\end{remark}

Let $f \in \HRspace_{d_0,d_\infty,\theta}$. Denote by $\Hq$ the closure of the Herman ring of $f$, and by $Y^0$ and $Y^\infty$ the connected components of $\RS \backslash \Hq$ containing $0$ and $\infty$ respectively. The covering structure of $f$ is well understood.

\begin{proposition}
\label{bubble-str}
       The preimage $f^{-1}(\Hq)$ of $\Hq$ is of the form
        $$
        \Hq \cup \bigcup_{i=1}^{d_0-1} A_i^0 \cup \bigcup_{j=1}^{d_\infty-1} A_j^\infty
        $$
        where for each $\bullet \in \{0,\infty\}$ and $i \in \{1, \ldots, d_{\bullet}-1\}$,
        \begin{enumerate}[label=\textnormal{(\arabic*)}]
            \item $A_i^{\bullet}$ is a closed topological annulus in $\overline{Y^{\bullet}}$;
            \item $A_i^{\bullet} \cap \Hq = \{c\}$ for some critical point $c$;
            \item if $j \neq i$, $A_i^{\bullet} \cap A_j^{\bullet}$ is either empty or $\{c\}$ for some critical point $c$;
            \item $f$ is univalent in the interior of $A_i^\bullet$.
        \end{enumerate}
\end{proposition}

\begin{proof}
    For each $\bullet \in \{0,\infty\}$, the boundary $\partial Y^{\bullet}$ is a quasicircle along which $f$ is conjugate to the irrational rotation. We can perform Douady-Ghys surgery\footnote{A combinatorial proof avoiding the surgery procedure is possible, but we will leave it as an exercise to the keen reader.} to replace $f$ on the disk $D_{\bullet} := \RS \backslash (\Hq \cup Y^{\bullet})$ with a rotation and obtain a rational map $P_{\bullet}$ that satisfies the following properties:
    \begin{itemize}
        \item[$\rhd$] $P_{\bullet}$ admits an invariant Siegel disk $Z_{\bullet} \subset \C $, which is a quasidisk;
        \item[$\rhd$] there is a quasiconformal map $\phi_{\bullet}: \RS \to \RS$ that restricts to a conjugacy between $f|_{\RS \backslash D_{\bullet}}$ and $P_{\bullet}|_{\RS \backslash Z_{\bullet} }$;
        \item[$\rhd$] $\phi_{\bullet}(\bullet) = \infty$, and thus $P_{\bullet}$ has a superattracting fixed point at $\infty$ with local degree $d_{\bullet}$.
    \end{itemize} 
    
    Clearly, for each $\bullet$, $P_{\bullet}$ must have degree at least $d_{\bullet}$. The critical points of $P_{\bullet}$ aside from $\infty$ must lie on $\partial Z_{\bullet}$. Moreover, the sum of the numbers of critical points of $P_0$ and $P_\infty$ is equal to the number of critical points of $f$, which is $2 (d_0 + d_\infty - 2)$. As such, $P_0$ and $P_\infty$ must be polynomials of degrees $d_0$ and $d_\infty$ respectively. 
    
    For each $\bullet$, the maximum modulus principle implies that the preimage of $Z_{\bullet}$ under $P_{\bullet}$ must be of the form $Z_{\bullet} \cup E_1^{\bullet} \cup \ldots, E_{d_{\bullet}-1}^{\bullet}$ for some $d_{\bullet}-1$ pairwise disjoint open disks $E_i^{\bullet}$'s where for each $i$, $P_\bullet$ is univalent in $E_i^\bullet$ and the closure $\overline{E_i^\bullet}$ intersects $\overline{ Z_{\bullet}}$ precisely at one point, which is a critical point of $P_{\bullet}$. Therefore, the preimage of $\Hq$ under $f$ is of the form 
    \[
    \Hq \cup \bigcup_{\bullet \in \{0, \infty \}} \bigcup_{i=1}^{d_{\bullet}-1} \big( \phi_{\bullet}^{-1}(\overline{E_i^\bullet}) \cap f^{-1} (\Hq) \big).
    \]
    Then, the proposition follows immediately.
\end{proof}

Denote by $\mathcal{C}$ the set of all free critical points of $f$. For any $n\geq 1$, we refer to the closure of a component of $f^{-n}(\Hq) \backslash f^{-(n-1)}(\Hq)$ as a \emph{bubble} of \emph{generation} $n$. By Proposition \ref{bubble-str}, every bubble $B$ of generation $n$ is a closed annulus admitting a unique point on the outer boundary of $B$ that lies on the pre-critical set $ f^{-(n-1)}(\mathcal{C})$. This unique point will be called the \emph{root} of $B$. In particular, every bubble of generation $1$ is precisely one of the $A_i^\bullet$'s above and it is rooted at a unique critical point. (See Figure \ref{fig:bubbles}.) We say that a bubble attached to $\Hq$ is an \emph{inner} bubble if it lies in $\overline{Y^0}$ and an \emph{outer} bubble if it lies in $\overline{Y^\infty}$.

We shall formally define combinatorics of Herman rings as follows. For any $n \in \N$, the $n^{\text{th}}$ symmetric product $\SP^n(\T)$ of the unit circle $\T$ is the quotient of the $n$-dimensional torus $\T^n$ under the symmetric group $S_n$ acting by permutation. Elements of $\SP^n(\T)$ are precisely unordered $n$-tuples of elements of $\T$. 

\begin{definition}
    Define $\mathcal{C}_{m,n}$ to be the quotient space of $\SP^{m-1}(\T) \times \SP^{n-1}(\T)$ modulo the action of $\T$ by any rigid rotation, endowed with the quotient topology.
\end{definition}

Let $\phi : \Hq \to \{ 1 \leq |z| \leq R \}$, where $R = e^{2 \pi \modu(\Hq)}$, denote a linearization of $f|_\Hq$. Let $(c^0_1,\ldots,c^0_{d_0-1})$ and $(c^\infty_1,\ldots,c^\infty_{d_\infty-1})$ denote the tuples of inner and outer critical points of $f$ counting multiplicity.

\begin{definition}
\label{comb-definition}
    The \emph{combinatorics} of $f \in \HRspace_{d_0, d_\infty,\theta}$ is the element $\text{comb}(f)$ in $\mathcal{C}_{d_0,d_\infty}$ induced by the pairs of tuples $\left(\phi(c^0_1),\ldots, \phi(c^0_{d_0-1})\right)$ and $\left(\frac{\phi(c^\infty_1)}{R}, \ldots, \frac{\phi(c^\infty_{d_0-1})}{R}\right)$.
\end{definition}

Note that $\text{comb}(f)$ is well-defined because $\phi$ is unique up to post-composition with rigid rotation.

Zhang \cite{Z08} proved that bounded type Siegel disks of any prescribed combinatorics are realized by a unique rational map as long as outside the closure of the Siegel disk, the postcritical set is finite and there are no Thurston obstructions. Using methods similar to Shishikura's quasiconformal surgery, Wang \cite{W12} extended Zhang's result to rational maps with an invariant Herman ring where outside the closure of the Herman ring, the postcritical set is finite and there are no Thurston obstructions. In particular, such Herman rings are uniquely determined by their conformal moduli and the combinatorial data on their boundaries.

\begin{theorem}[\cite{W12}]
\label{wang}
    For any bounded type irrational number $\theta$, any positive number $\mu >0$, and any $\mathcal{C} \in \mathcal{C}_{d_0,d_\infty}$, there is a rational map $f \in \HRspace_{d_0, d_\infty,\theta}$ such that its Herman ring has modulus $\mu$ and combinatorics $\mathcal{C}$. Moreover, such $f$ is unique up to conformal conjugacy.
\end{theorem}

\subsection{Herman quasicircles}
\label{ss:herman-quasicircles}
To obtain a priori bounds for Herman rings in the space $\HRspace_{d_0, d_\infty,\theta}$, Corollary \ref{old-bounds} suggests that it is sufficient to consider those with small moduli. This motivates us to work with the following degenerate version.

\begin{definition}
    For any $d_0, d_\infty \geq 2$ and irrational number $\theta \in (0,1)$, define $\HQspace_{d_0, d_\infty,\theta}$ to be the space of all degree $d_0 + d_\infty -1$ rational maps $f$ such that
\begin{enumerate}[label=\text{(\roman*)}]
    \item\label{def:HQ1} $0$ and $\infty$ are superattracting fixed points of $f$ with local degrees $d_0\geq 2$ and $d_\infty\geq 2$ respectively;
    \item\label{def:HQ2} $f$ has a Herman quasicircle $\Hq$ of rotation number $\theta$ (see Definition \ref{herman-curve-defn});
    \item\label{def:HQ3} $\Hq$ separates $0$ and $\infty$;
    \item\label{def:HQ4} every critical point of $f$ other than $0$ and $\infty$ lies in $\Hq$.
\end{enumerate} 
\end{definition}

\begin{example}[Unicritical trivial Herman curves]
\label{eg:general-blaschke}
    For any $d \geq 2$ and irrational $\theta \in (0,1)$, there is a unique point $c=c(d,\theta)$ on $\T$ such that the map
    \[
    B_{d,\theta}(z) = c z^d \frac{ \displaystyle\sum_{j=0}^{d-1} \binom{2d-1}{j} (-1)^j z^{d-1-j} 
    }{ 
    \displaystyle\sum_{j=0}^{d-1} \binom{2d-1}{j} (-1)^j z^{j} 
    }
    \]
    is a Blaschke product lying in $\HQspace_{d,d,\theta}$ with a trivial Herman quasicircle $\T$ of rotation number $\theta$ and a unique free critical point $z=1$ and critical value $c$. Indeed, it is not difficult to check by straightforward computation that $B_{d,\theta}$ commutes with the reflection $z \mapsto 1/\overline{z}$, that \ref{def:HQ1} holds, and that $z=1$ is the only free critical point with image $B_{d,\theta}(1)=c$. By Rouch\'e's theorem, one can check that each of the $d-1$ poles lies within the unit disk. By the argument principle, $B_{d,\theta}$ restricts to an analytic circle homeomorphism. The uniqueness of the parameter $c$ comes from standard monotonicity considerations in the theory of circle maps. (See \cite[\S4]{dMvS93}.) Note that $B_{2,\theta}$ coincides with the Blaschke product $B_{\theta}$ in the introduction (\ref{cubic-blaschke-product}).
\end{example}

The example above will be further generalized in Proposition \ref{unicritical-formula}.

Let $f \in \HQspace_{d_0, d_\infty,\theta}$. Denote by $\Hq$ the Herman quasicircle of $f$, and by $Y^0$ and $Y^\infty$ the connected components of $\RS \backslash \Hq$ containing $0$ and $\infty$ respectively. A generalization of the Herman-\'Swi\k{a}tek theorem by Petersen guarantees that the bounded type assumption automatically gives us regularity of the conjugacy.

\begin{proposition}[\cite{Pe04}]
    The rotation number $\theta$ is of bounded type if and only if there exists a quasiconformal map $\phi: \RS \to \RS$ such that $\phi(\Hq) = \T$ and $f = \phi^{-1} \circ R_\theta\circ \phi$ in $\Hq$.
\end{proposition}

We can apply Douady-Ghys surgery along $\Hq$, repeat most of the proof of Proposition \ref{bubble-str}, and obtain the following proposition.

\begin{proposition}
\label{bubble-structure}
    Suppose $\theta$ is of bounded type.
    \begin{enumerate}[label=\textnormal{(\arabic*)}]
        \item For each $\bullet \in \{0,\infty\}$, there is a degree $d_\bullet$ polynomial $P_\bullet$ and a quasiconformal map $\psi_\bullet: \RS \to \RS$ such that
    \begin{itemize}
        \item[$\rhd$] $P_\bullet$ has a Siegel disk $Z_\bullet$ of rotation number $\theta$ if $\bullet = \infty$, $1-\theta$ if $\bullet = 0$;
        \item[$\rhd$] every finite critical point of $P_\bullet$ lies on $\partial Z_\bullet$;
        \item[$\rhd$] $\psi_\bullet$ conjugates $f|_{\overline{Y^\bullet}}$ and $P_\bullet|_{\RS \backslash Z_\bullet}$;
        \item[$\rhd$] $\psi_\bullet$ is conformal on the immediate basin of attraction of $\bullet$.
    \end{itemize}
    \item The preimage $f^{-1}(\Hq)$ of $\Hq$ is of the form
        $$
        \Hq \cup \bigcup_{i=1}^{d_0-1} A_i^0 \cup \bigcup_{j=1}^{d_\infty-1} A_j^\infty
        $$
        where for each $\bullet \in \{0,\infty\}$ and $i\in \{1, \ldots, d_{\bullet}-1\}$,
        \begin{itemize}
            \item[$\rhd$] $A_i^{\bullet}$ is a Jordan curve in $\overline{Y^{\bullet}}$;
            \item[$\rhd$] $A_i^{\bullet} \cap \Hq = \{c\}$ for some critical point $c$;
            \item[$\rhd$] if $j\neq i$, $A_i^{\bullet} \cap A_j^{\bullet}$ is either empty or $\{c\}$ for some critical point $c$, and the Jordan disks enclosed by $A_i^\bullet$ and $A_j^\bullet$ respectively are disjoint;
            \item[$\rhd$] $f: A_i^\bullet \to \Hq$ is a homeomorphism.
        \end{itemize}
    \end{enumerate}
\end{proposition}

The first part of the proposition hints that $f$ can be interpreted as a welding\footnote{This is similar to Bers' simultaneous uniformization. Compare with \cite{McM} and \cite[\S7.4]{BF14}.} of two polynomial Siegel disks of degrees $d_0$ and $d_\infty$ along their boundaries. The second part\footnote{Actually, the second part of Proposition \ref{bubble-structure} holds for any irrational $\theta$. This follows from an alternative combinatorial proof avoiding surgery.} tells us that the covering structure of $f$ is identical to that of maps in $\HRspace_{d_0, d_\infty,\theta}$. 

\begin{remark}
    Despite the striking similarity, a priori we do not know yet whether degenerating Herman rings in $\HRspace_{d_0,d_\infty,\theta}$ can converge to a limit in $\HQspace_{d_0,d_\infty,\theta}$. We also do not know yet if every rational map in $\HQspace_{d_0,d_\infty,\theta}$ arises as a genuine limit of Herman rings. Some of these will be resolved in Section \S\ref{sec:construction-of-herman-curves}.
\end{remark}

Similarly, we can define the notion of bubbles by taking preimages of the Herman quasicircle. See Figure \ref{fig:bubbles}. By standard properties of Julia sets (see \cite{M06}), bubbles generate the Julia set $J(f)$ of $f$.

\begin{figure}
    \centering
    
    \begin{tikzpicture}
    \node[anchor=south west,inner sep=0] (image) at (0,0) {\includegraphics[scale=0.17]{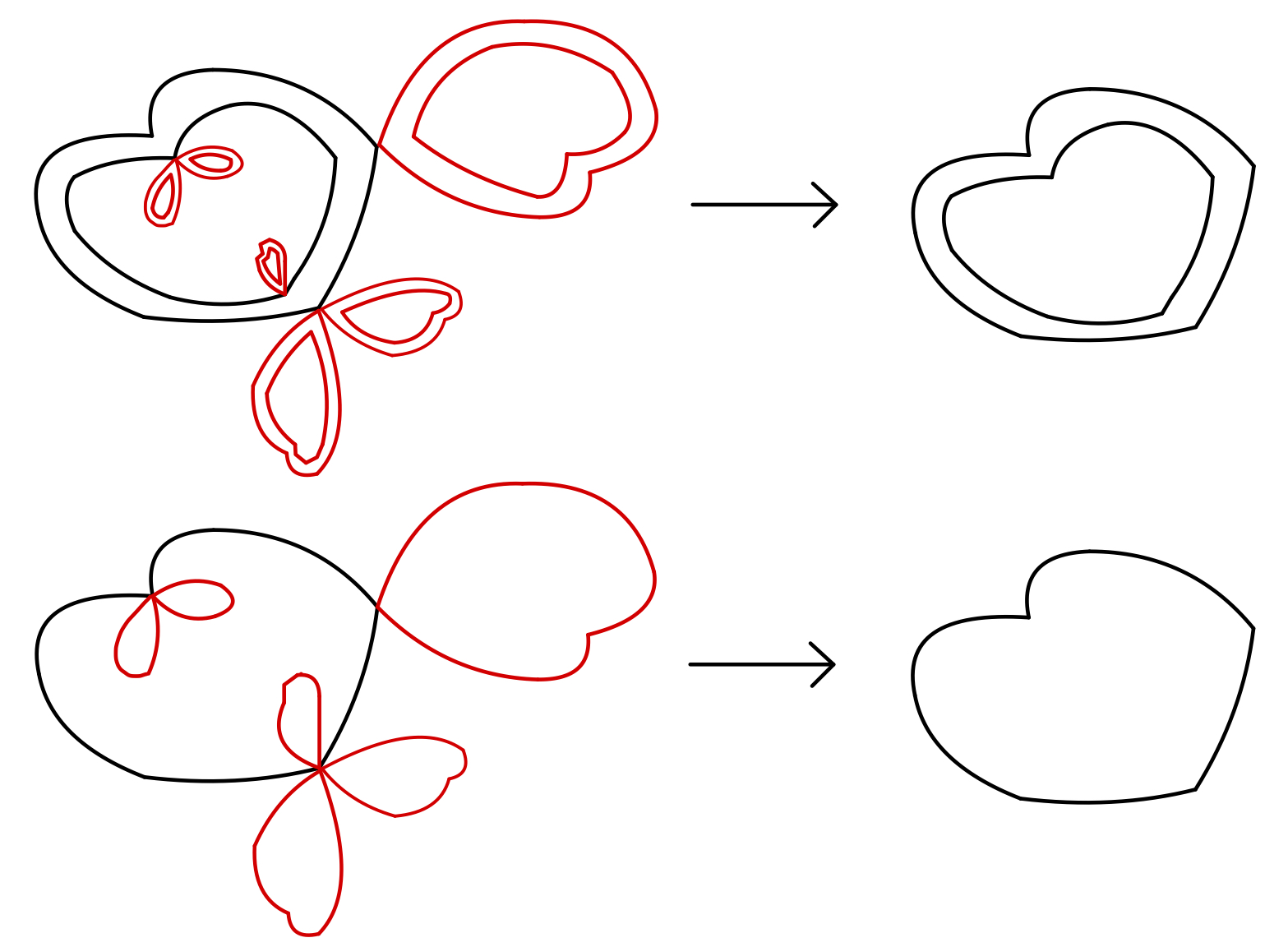}};
    \begin{scope}[
        x={(image.south east)},
        y={(image.north west)}
    ]
        \node [black, font=\bfseries] at (0.58,0.73) {$f$};
        \node [black, font=\bfseries] at (0.58,0.25) {$g$};
    \end{scope}
\end{tikzpicture}

    \caption{Bubbles of generation $1$ for $f \in \HRspace_{d_0, d_\infty,\theta}$ above and $g \in \HQspace_{d_0, d_\infty,\theta}$ below.}
    \label{fig:bubbles}
\end{figure}

\begin{proposition}
\label{bubbles-julia}
    $
    J(f) = \overline{\cup_{k=0}^\infty f^{-k}(\Hq)}.
    $
\end{proposition}

Let $\phi : \Hq \to \T$ be the quasisymmetric conjugacy between $f|_\Hq$ and $R_\theta$. By pushing forward inner and outer critical points of $f$ under $\phi$, we again obtain a well-defined element $\text{comb}(f)\in \mathcal{C}_{d_0,d_\infty}$.

\begin{definition}
\label{def:combinatorics}
    The \emph{combinatorics} of $f \in \HQspace_{d_0, d_\infty,\theta}$ is the element $\text{comb}(f)$ in $\mathcal{C}_{d_0,d_\infty}$.
\end{definition}

For bounded type rotation number $\theta$, showing that $\HQspace_{d_0, d_\infty,\theta}$ is non-empty and that any prescribed combinatorics is realizable is the heart of Theorem \ref{main-theorem-02}. 

\subsection{Irrational rotations}
\label{ss:irrational-rotations}
Consider the following general setup. Let $f : \Hq \to \Hq$ be a homeomorphism on a closed annulus $\Hq \subset \RS$. Suppose $f$ is topologically conjugate via $\phi: \Hq \to A$ to the rigid rotation $R_\theta(z) = e^{2\pi i\theta} z$ on a closed round annulus $A = \{1 \leq |z| \leq R\}$. Via the projection $\psi: A \to \T, z \mapsto \frac{1}{2\pi}\text{arg}(z)$, we can equip $\Hq$ with the pullback under $\psi \circ \phi$ of the Euclidean metric on $\T: = \R / \Z$, called the \emph{combinatorial pseudometric} of $\Hq$.

A closed set $I \subset \Hq$ is called a \emph{piece} in $\Hq$ if it is of the form $(\psi \circ \phi)^{-1}(I')$ for some closed interval $I' \subset \T$. Define the \emph{combinatorial length} $|I|$ of a piece $I$ to be the diameter of $I$ with respect to the combinatorial pseudometric.

For any two distinct points $x, y \in \Hq$, we denote by $[x,y]$ the unique combinatorially shortest piece that contains both $x$ and $y$. Note that if $\psi(\phi(x))=\psi(\phi(y))$, then $[x,y]$ is a radial segment in $\Hq$ with zero combinatorial length.

Let $\{ \frac{p_n}{q_n} \}$ be the sequence of best rational approximations of $\theta = [0;a_1,a_2,\ldots]$. These can be determined by the recurrence relation $p_{n} = a_{n} p_{n-1} + p_{n-2}$ and $q_{n} = a_{n} q_{n-1} + q_{n-2}$ where $p_0= q_{-1}=0$ and $q_0 = p_{-1} = 1$.

\begin{definition}
    A \emph{combinatorial piece} of \emph{level $n$} is a piece of the form $[x, f^{q_n}(x)]$ for some $x \in \Hq$.
\end{definition} 

Every combinatorial piece of level $n$ has the same combinatorial length equal to
\[
l_n: = |p_n - q_n\theta|.
\]
The bounded type assumption controls the rate of decrease of the $l_n$'s.

\begin{proposition}
\label{bounded-type}
    There exist a pair of constants $\tilde{C}, C>1$ depending only on $\beta(\theta)$ such that for every positive integer $n$, 
    \[
    \tilde{C} l_{n+1} \leq l_n \leq C l_{n+1}.
    \]
\end{proposition}

The sequence $q_n$'s are precisely the first return times for $R_\theta$ (and thus for $f|_\Hq$ too) in an alternating fashion with respect to the cyclic order:
\[
R^{q_1}_\theta(x) < R^{q_3}_\theta(x) < R^{q_5}_\theta(x) < \ldots < x < \ldots < R^{q_6}_\theta(x) < R^{q_4}_\theta(x) < R^{q_2}_\theta(x).
\]

\begin{proposition}
\label{renormalization-tiling}
    For any $x \in \Hq$ and $n \in \N$,
    \[
    \Hq = \bigcup_{i=0}^{q_{n+1}-1} f^i \left([x, f^{q_n} (x)]\right) \cup \bigcup_{j=0}^{q_n-1} f^j\left([f^{q_{n+1}}(x), x]\right).
    \]
    All the pieces in the expression above have pairwise disjoint interiors, and all the level $n+1$ combinatorial pieces above are pairwise disjoint.
\end{proposition}

The decomposition of $\Hq$ above is called the $n^{\text{th}}$ \emph{renormalization tiling} induced by $x \in \Hq$. Keeping only the level $n$ pieces from the renormalization tiling gives us an \emph{almost tiling} whose gaps have length $l_{n+1}$. We will also often apply the weaker fact that for any $n \geq 3$, the orbit $\{f^i(x)\}_{i=0,\ldots, q_n}$ partitions $\Hq$ into pieces of length between $l_{n}$ and $l_{n-2}$.

\subsection{Encoding degeneration}
\label{ss:encoding-degeneration}
For every $\alpha\geq 3$ and piece $I \subset \Hq$ of length $|I|< \frac{1}{\alpha}$, we will use the following notation:
\begin{itemize}
    \item[$\rhd$] $I^c =$ the closure of $\Hq\backslash I$;
    \item[$\rhd$] $\alpha I =$ the combinatorial rescaling of $I$ by the factor of $\alpha$, that is, the unique piece in $\Hq$ of length $\alpha |I|$ having the same combinatorial mid-segment as $I$;
    \item[$\rhd$] $\mathcal{F}_\alpha (I) =$ the set of proper curves in $\hat{\mathbb{C}} \backslash \left(I \cup (\alpha I)^c\right) $ connecting $I$ and $(\alpha I)^c$;
    \item[$\rhd$] $W_\alpha(I) = $ the \emph{$\alpha$-width} of $I$, that is, the extremal width of $\mathcal{F}_\alpha(I)$.
\end{itemize}

\begin{figure}
    \centering
    
    \begin{tikzpicture}
    \node[anchor=south west,inner sep=0] (image) at (0,0) {\includegraphics[width=0.8\linewidth]{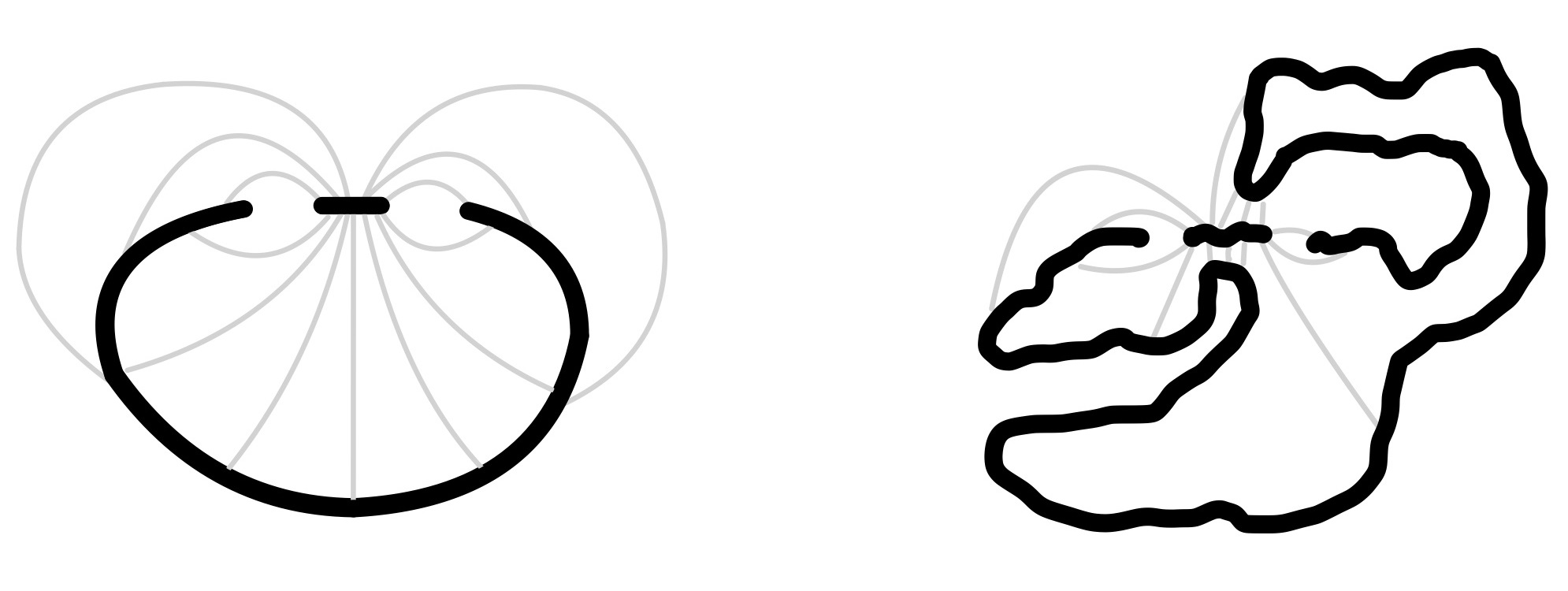}};
    \begin{scope}[
        x={(image.south east)},
        y={(image.north west)}
    ]
        \node [black, font=\bfseries] at (0.23,0.74) {$I$};
        \node [black, font=\bfseries] at (0.765,0.7) {$I$};
        \node [black, font=\bfseries] at (0.38,0.19) {$(\alpha I)^c$};
        \node [black, font=\bfseries] at (0.95,0.30) {$(\alpha I)^c$};
    \end{scope}
\end{tikzpicture}
    
    \caption{On the left, $I$ has small $\alpha$-width. On the right, $I$ has large $\alpha$-width.}
    \label{fig:comparison}
\end{figure}
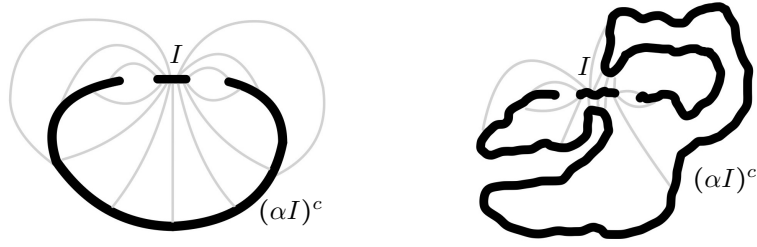

When $R=1$, $\Hq$ is a Jordan curve, the combinatorial pseudometric is a metric on $\Hq$, and every piece in $\Hq$ is a genuine interval. Additionally, when the conjugacy $\phi$ is quasiconformal, we have the following. (Compare with \cite[Lemma 11.3]{DL22}.)

\begin{proposition}
    \label{main-prop}
    Let $\phi: \RS \to \RS$ be a quasiconformal map that maps a quasicircle $\Hq$ onto the unit circle $\T$. Equip $\Hq$ with the combinatorial metric induced by $\phi$.
    \begin{enumerate}[label=\textnormal{(\arabic*)}]
        \item For every $\alpha \geq 3$, there is a constant $\mathbf{K}$ depending on $\alpha$ and the dilatation of $\phi$ such that every interval $I\subset \Hq$ of combinatorial length $|I| < (2\alpha)^{-1}$ satisfies $W_\alpha(I) \leq \mathbf{K}$.
        \item Conversely, if there are some constants $\alpha \geq 3$, $\varepsilon \in (0,1)$, and $\mathbf{K}>0$ such that $W_\alpha(I) \leq \mathbf{K}$ for every interval $I \subset \Hq$ of combinatorial length at most $\varepsilon$, then the dilatation of $\Hq$ depends only on $\alpha$, $\varepsilon$ and $\mathbf{K}$.
    \end{enumerate}
\end{proposition}

\begin{proof} 
Pick any $\alpha \geq 3$ and any interval $I \subset \Hq$ of length $|I|<(2\alpha)^{-1}$. On the circle, $\phi(I)$ has width $W_\alpha(\phi(I)) \leq M$ for some constant $M=M(\alpha)>1$. Therefore, on $\Hq$, the interval $I$ has width $W_\alpha(I) \leq k M$, where $k$ denotes the dilatation of $\phi$, and so (1) holds.

To show the converse, we first claim that every interval $I$ of length at most $\varepsilon$ must satisfy $W_3 (I) \leq \alpha \mathbf{K}$. If otherwise, then we could partition $I$ into $\lfloor \alpha \rfloor$ pieces $I_1, \ldots, I_{\lfloor \alpha \rfloor}$ of equal combinatorial length. Since each $(\alpha I_i)^c$ contains $(3I)^c$,
$$
\sum_{i=1}^{\lfloor \alpha \rfloor} W_\alpha(I_i) \geq \sum_{i=1}^{\lfloor \alpha \rfloor} W(I_i, (3I)^c) \geq W_3(I) > \alpha \mathbf{K}.
$$
Then, at least one of the pieces $I_j$ satisfies $W_\alpha (I_j) > \mathbf{K}$, which is a contradiction.

Assume without loss of generality that $\Hq$ separates $0$ and $\infty$. For any $\bullet \in \{0,\infty\}$, we denote by $Y^\bullet$ the component of $\RS \backslash \Hq$ containing $\bullet$. For any interval $J \subset \Hq$, let $\hm_{\bullet}(J)$ denote the harmonic measure of $J$ on $Y^{\bullet}$ about $\bullet$ and let $W_3^{\bullet} (J)$ denote the width of the family of proper curves in $Y^{\bullet}$ connecting $J$ and $(3J)^c$. Since $W_3 (I) \leq \alpha \mathbf{K}$, then $W_3^\bullet (I)\leq \alpha \mathbf{K}$ for $\bullet \in \{0,\infty\}$. 

Denote by $L$ and $R$ the two connected components of $\overline{3 I \backslash I}$. For $\bullet \in \{0,\infty\}$, by Proposition \ref{log-rule},
\[
\hm_{\bullet}(I) < M \cdot  \min\{ \hm_{\bullet}(L), \hm_{\bullet}(R) \}
\]
for some $M = M(\alpha \mathbf{K})\geq 1$. Thus, any two neighboring combinatorial intervals $I$ and $J$ of equal combinatorial length satisfy 
\[
M^{-1} \hm_{\bullet}(J) < \hm_{\bullet}(I) < M \hm_{\bullet}(J).
\]
As such, the inner and outer harmonic measures are quasisymmetrically equivalent to the combinatorial measure, and consequently to each other as well. By conformal welding, this implies (2).
\end{proof}

In practice, to bound the dilatation of a quasicircle $\Hq$, it suffices to establish a bound on the $\alpha$-width of sufficiently deep intervals in $\Hq$ and for some $\alpha\geq 3$. Degeneration is encoded by the presence of an interval with very large $\alpha$-width.

\subsection{Setup and notation}
\label{ss:setup}
This subsection provides a list of notations and conventions that will be used throughout the remainder of this paper. 

\subsubsection{Notation}
\label{sss:notation}

Unless otherwise stated, we always fix a bounded type irrational $\theta$ and integers $d_0, d_\infty \geq 2$. Any dependence on $d_0$, $d_\infty$ and $\beta(\theta)$ will always be implicit. In our analysis, we will often use the following notation:
\begin{itemize}
    \item[$\rhd$] $p_n/q_n =$ $n$\textsuperscript{th} best rational approximation of $\theta$ for $n\geq 1$;
    \item[$\rhd$] $l_n = |p_n - q_n\theta|$ for any $n\geq 1$;
    \item[$\rhd$] $x \oplus y := (x^{-1}+y^{-1})^{-1}$;
    \item[$\rhd$] $g = O(h)$ when $g,h >0$ and $g \leq c h$ for some implicit constant $c>0$;
    \item[$\rhd$] $g \succ h$ when $g,h >0$ and $g \geq c h$ for some implicit constant $c > 0$;
    \item[$\rhd$] $g \asymp h$ when $g \succ h$ and $h \succ g$;
\end{itemize}

Given a family of curves $\mathcal{F}$, denote by $W(\mathcal{F})$ the extremal width of $\mathcal{F}$. Given a compact subset $K$ of a Riemann surface $U$ with boundary, we denote:
\begin{itemize}
    \item[$\rhd$] $\mathcal{F}(U,K) =$ the family of proper curves in $U\backslash K$ connecting $\partial U$ and $K$;
    \item[$\rhd$] $\mathcal{F}^h_{\textnormal{can}}(U,K) =$ the set of leaves of the canonical lamination of $U \backslash K$ that are \emph{horizontal} (both endpoints are on $K$);
    \item[$\rhd$] $\mathcal{F}^v_{\textnormal{can}}(U,K) =$ the set of leaves of the canonical lamination of $U \backslash K$ that are \emph{vertical} (connects $\partial U$ and $K$);
    \item[$\rhd$] $W(U,K) =$ the extremal width of $\mathcal{F}(U,K)$.
\end{itemize}
Refer to Appendix \ref{sec:near-degenerate-regime} for details. 

We will be investigating degeneration for both Herman rings and Herman quasicircles at the same time. Throughout Sections \S\ref{sec:spreading-degeneration}--\ref{sec:a-priori-bounds}, we will always assume that we are in one of the following situations:

\begin{enumerate}[label=\wackyenum*]
    \item\label{case:herman-curve} $f \in \HQspace_{d_0, d_\infty, \theta}$ and $\Hq$ is the Herman quasicircle of $f$;
    \item\label{case:herman-ring} $f \in \HRspace_{d_0, d_\infty, \theta}$ and $\Hq$ is the closure of the Herman ring $\He$ of $f$;
    \item\label{case:boundary-ring} $f \in \HRspace_{d_0, d_\infty, \theta}$ and $\Hq$ is the outer boundary component of the Herman ring $\He$ of $f$. (The treatment for the inner boundary is analogous.)
\end{enumerate}
In Section \S\ref{sec:bubble-wave-argument}, only \ref{case:herman-curve} and \ref{case:herman-ring} are considered. For each of three cases above, we let 
\begin{itemize}
    \item[$\rhd$] $d := \text{deg}(f) = d_0 + d_\infty - 1$;
    \item[$\rhd$] $Y^{\bullet} :=$ the connected component of $\RS \backslash \Hq$ containing $\bullet$, for $\bullet \in \{0,\infty\}$.
\end{itemize}
For any piece $I \subset \Hq$,
\begin{itemize}
    \item[$\rhd$] $|I| :=$ the combinatorial length of $I \subset \Hq$;
    \item[$\rhd$] $I^c :=$ the closure of $\Hq\backslash I$.
\end{itemize}
Moreover, for any $\alpha \in (1, |I|^{-1})$,
\begin{itemize}
    \item[$\rhd$] $\alpha I :=$ the piece in $\Hq$ of length $\alpha |I|$ that shares the same mid-segment as $I$;
    \item[$\rhd$] $\mathcal{F}_\alpha(I) := \mathcal{F}( \RS \backslash (\alpha I)^c, I)$;
    \item[$\rhd$] $W_\alpha(I) := W( \RS \backslash (\alpha I)^c, I)$, a measure of (near-)degeneracy of $\Hq$ at $I$.
\end{itemize}
When $W_\alpha(I) \geq K$ for some $K>1$, we say that $I$ is \emph{$[K,\alpha]$-wide}.

Fix the constant\footnote{The reader may wish to assign a different value for $\tau$ as long as it is a sufficiently large integer.} $\tau := 10$. Local degeneration will be represented by two quantities, namely the \emph{$\tau$-degeneration} $W_\tau(I) \gg 1$ and the \emph{$\lambda$-degeneration} $W_\lambda(I) \gg 1$ at a piece $I \subset \Hq$ for some large parameter $\lambda \gg \tau$. We will take $\lambda$ to be sufficiently large for our analysis to work (it will eventually be fixed in Theorem \ref{amplification}), and we will emphasize whenever other constants depend on $\lambda$ throughout Sections \S\ref{sec:spreading-degeneration}--\ref{sec:loss-of-horizontal-width}. One particular parameter that will appear frequently is $\thres_\lambda$ defined below.

\begin{definition}
\label{threshold}
    For any $\lambda >1$, denote by $\thres_\lambda$ the smallest integer such that for any combinatorial piece $I \subset \Hq$ of level $\geq \thres_\lambda$, the pieces $2\lambda I$, $2\lambda f(I)$, and $2\lambda f^2(I)$ are pairwise disjoint.
\end{definition}

\subsubsection{The modulus and the Siegel scale}
\label{sss:modulus-Siegel-scale}

In Cases \ref{case:herman-ring} and \ref{case:boundary-ring}, we denote by $\mu$ the modulus of the Herman ring $\He$. In Case \ref{case:boundary-ring}, we impose the additional assumption that any interval $I \subset \Hq$ we consider is always at the \emph{Siegel scale}, i.e. $|I|\leq\mu$.

\begin{lemma}
    \label{negligibility}
    In Case \ref{case:boundary-ring}, for any interval $I \subset \Hq$, the width of curves in the Herman ring $\He$ connecting $I$ and the inner boundary component $H^0$ is at most $5$.
\end{lemma}

\begin{proof}
    For any interval $J$ in the outer boundary component $\Hq$, let $\tilde{J}$ denote the corresponding piece in $\overline{\He}$ such that $J=\tilde{J} \cap \Hq$. It comes with a canonical structure of a conformal rectangle with horizontal sides $\tilde{J} \cap \partial \He$.
    
    At the Siegel scale, it is sufficient to prove the lemma for any interval $I$ of length $|I| = \mu$. Let $L$ and $R$ denote the two intervals in $\Hq$ adjacent to $I$ that have the same length $\mu$. Then, $W(\widetilde{L}) = W(\widetilde{R}) = W(\tilde{I}) = 1$. The family $\mathcal{F}^0$ of curves in $\He$ connecting $I$ and $H^0$ is contained in the union $\mathcal{F}_1 \cup \mathcal{F}_2$, where $\mathcal{F}_1$ consists of vertical curves of $\widetilde{3 I}$ and $\mathcal{F}_2 := \mathcal{F}^0 \backslash \mathcal{F}_1$. Observe that $W(\mathcal{F}_1) = W(\widetilde{3I}) = 3$. Since every curve in $\mathcal{F}_2$ must cross either of the two rectangles $\widetilde{L}$ and $\widetilde{R}$, then by Proposition \ref{non-crossing-principle}, $W(\mathcal{F}_2) \leq 2$. Therefore, $W\left(\mathcal{F}^0\right) \leq 5$.
\end{proof}

Intervals at the Siegel scale are conformally far from the inner boundary component $H^0$ of the Herman ring $\He$. As such, this situation is comparable to that of an interval on the boundary of a Siegel disk, in which the width between $I$ and the inner component, which is the singleton consisting of the center, is $0$.

We always assume that the modulus $\mu$ of $\He$ is sufficiently small. (Otherwise, a priori bounds can be obtained from Corollary \ref{old-bounds}.) More precisely, we assume that 
\[
\mu \leq l_{\thres_\lambda + \step_\lambda}
\]
where $\thres_\lambda$ is from Definition \ref{threshold} and $\step_\lambda$ is some constant depending on $\lambda$ which will appear later in Theorem \ref{demotion} when the notion of shallow and deep scales are introduced. In particular, intervals at the Siegel scale are always on the deep scale.

The arguments we present in Sections \S\ref{sec:bubble-wave-argument}--\ref{sec:loss-of-horizontal-width} will mainly address Cases \ref{case:herman-curve} and \ref{case:herman-ring} using only the combinatorial and dynamical properties of $\Hq$. The modulus $\mu$ will not play any major role until Sections \S\ref{sec:a-priori-bounds}--\ref{sec:construction-of-herman-curves}. Most of the arguments in Sections \S\ref{sec:spreading-degeneration}--\ref{sec:loss-of-horizontal-width} apply to Case \ref{case:boundary-ring} with a few adjustments presented as separate remarks.


\section{Waves}
\label{sec:bubble-wave-argument}

In Sections \S\ref{sec:trading} and \S\ref{sec:amplifying-tau-degeneration}, we will encounter degeneration witnessed by a combinatorial piece $I$ that is either \emph{shallow}, i.e. has level bounded above by some constant, or \emph{deep}, i.e. not shallow. In the shallow case, we will need to rule out the presence of wide waves. Waves are defined as follows.

\begin{definition}
    For $\bullet \in \{0,\infty\}$ and a piece $A \subset \Hq$, we say that a curve $\gamma$ \emph{protects $A$ from $\bullet$} if it is a proper curve in $Y^\bullet$ such that every curve in $Y^\bullet$ joining $\bullet$ and $A$ must intersect $\gamma$.
    We say that a lamination $\Omega$ is a \emph{wave} if it is a proper lamination in $Y^\bullet$ for some $\bullet \in \{0,\infty\}$ such that there exists a piece $A$ that is protected from $\bullet$ by every leaf. 
    If $\bullet = 0$, it is called an \emph{inner wave}; if $\bullet = \infty$, it is called an \emph{outer wave}. The \emph{(combinatorial) length} $|\Omega|$ of a wave $\Omega$ is the maximum combinatorial length $|A|$ of pieces $A$ protected by $\Omega$.
\end{definition} 

In this section, our aim is to convert a wide wave into $\tau$-degeneration which increases with the length and width of the wave and is witnessed by a combinatorial piece of a controlled level.

\begin{proposition}[Wide waves $\longrightarrow$ $\tau$-degeneration]
\label{prop:wave-degeneration}
    There exists an absolute constant $m \in \mathbb{N}$ such that the following holds. For every $n \in \N$ and $\alpha \geq 1$, there exists some $\mathbf{K} = \mathbf{K}(n)>1$ such that if 
    \[
        \text{there exists a wave } \Omega \text{ of length } |\Omega| \geq \alpha l_n\text{ and width }  W(\Omega) \geq \mathbf{K} ,
    \]
    then 
    \[
    \text{there exists a level }n+m\text{ combinatorial piece }J\text{ with } W_\tau (J) \succ \alpha W(\Omega).
    \]
\end{proposition}

The constant $m$ above actually depends on the separation constant $\tau$, which we fix to be equal to $10$. Later in \S\ref{ss:trading-shallow} and \S\ref{ss:amplifying-shallow}, we will apply this proposition choosing $\alpha$ to be sufficiently large. 

Around the same time this paper is written, the notion of waves also appears in \cite{DL23} in which the authors prove \emph{a priori bounds} for infinitely renormalizable quadratic maps with bounded satellite combinatorics. In particular, \cite[Wave Lemma]{DL23} is an analog of Proposition \ref{prop:wave-degeneration}.

Here is the rough idea of the proof. As Figure \ref{fig:bubbles-and-waves} illustrates, most of the wave has to travel through roughly $5 \alpha$ disjoint bubble chains of generation up to $q_{n+a}$, where $a$ is some positive uniform constant. These bubble chains split the wave into multiple parts, which, after being pushed forward, induce an amplified wave with shorter length (Lemma \ref{bubble-wave-argument}). If the conclusion of Proposition \ref{prop:wave-degeneration} is not satisfied, then we can remove an inner buffer of the wave to increase the length back to the original size and the amplification argument can be repeated.

\subsection{Amplifying waves} 
\label{ss:amplifying-waves}
We first argue that a combinatorially long wide wave induces an even wider wave of smaller but controlled length. Refer to Figure \ref{fig:bubbles-and-waves}.

\begin{lemma}
\label{bubble-wave-argument}
    There exists an absolute constant $m' \in \mathbb{N}$ such that the following holds. For every $n \in \N$ and $\alpha \geq 1$, there exists some $\mathbf{K} = \mathbf{K}(n)>1$ such that if 
    \[
        \text{there exists a wave } \Omega \text{ of length } |\Omega| \geq \alpha l_n\text{ and width }  W(\Omega) \geq \mathbf{K} ,
    \]
    then 
    \[
    \text{there is another wave }\Omega'\text{ of length }|\Omega'| \geq l_{n+m'}\text{ and width  } W(\Omega') \geq 2\alpha \: W(\Omega).
    \]
\end{lemma}

\begin{proof} 
We shall first introduce two absolute constants $m'',m' \in \N$ satisfying
\begin{align}
    \label{ineq:m''} 10 \:l_{n+m''} &< l_n, \text{ and} \\
    \label{ineq:m'} 2d\: l_{n+m'} &\leq l_{n+m''+2}.
\end{align}
Set 
\[
t:= q_{n+m''+2}.
\]

Suppose $\Omega$ is an outer wave of length $\geq \alpha l_n$ and width $\geq \mathbf{K}$. Denote by $A$ the longest piece protected by $\Omega$. For every outer critical point $c \in \Hq$, denote by $\mathcal{O}_c:= \{ f|_\Hq^{-i}(c)\}_{i=0,\ldots, t-1} \cap A$ the set of preimages of $c$ up to time $t-1$ that lie on $A$. The union $\mathcal{O} := \cup_c \mathcal{O}_c$ partitions $A$ into $N$ pieces $P_1, \ldots, P_N$ of positive length.

\begin{claim1} 
    \namedlabel{s3.claim1}{Claim 1}
    There exist $k\geq 5\alpha$ distinct pieces $P_{n_1}, \ldots, P_{n_k}$ length $\geq l_{n+m'}$. 
\end{claim1}

\begin{proof}
    Since $|A| \geq \alpha l_n$, it is sufficient to show that the claim is true for $k \geq 5|A|/ l_n$. Suppose otherwise. Then, the number of pieces $P_i$ of length $< l_{n+m'}$ is more than $N-5|A|/l_n$ and the rest have length between $l_{n+m'}$ and $l_{n+m''}$. In particular,
    \[
    |A| < \left(N-\frac{5|A|}{l_n}\right) l_{n+m'} + \frac{5|A|}{l_n} l_{n+m''}.
    \]
    By (\ref{ineq:m''}), this simplifies to
    \[
    1< \left(\frac{2N}{|A|}-\frac{10}{l_n}\right) l_{n+m'}.
    \]
    For every critical point $c$, adjacent points in $\mathcal{O}_c$ have distance at least $l_{n+m''+2}$, so $\mathcal{O}_c$ has cardinality at most $|A|/l_{n+m''+2}$. Since $f$ has less than $d$ outer critical points, we deduce that $N < d|A|/l_{n+m''+2}$. As such,
    \[
    1 < \frac{2d \: l_{n+m'}}{l_{n+m''+2}} - \frac{10 \:l_{n+m'}}{l_n}.
    \]
    However, this implies that $2d \:l_{n+m'} > l_{n+m''+2}$, which contradicts (\ref{ineq:m'}).
\end{proof}

Denote by $Y^\infty_t$ the connected component of $f^{-t}(Y^\infty)$ that contains $\infty$. The map $f^t: Y^\infty_t \to Y^\infty$ is a degree $d_\infty^t$ covering map branched only at $\infty$. For every point $x$ in $\mathcal{O}$, denote by $B_x$ the connected component of the closure of $\partial Y^\infty_t \backslash \Hq$ that contains $x$. Each $B_x$ is the outer boundary of a union of finite chains of bubbles of generation up to $t$. We will remove part of the wave $\Omega$ that skips the $B_x$'s.

\begin{claim2}
\namedlabel{s3.claim2}{Claim 2}
    Consider a point $x$ in $\mathcal{O}$ and a proper lamination $\mathcal{L}$ in $Y^\infty_t$ protecting $B_x$, i.e. every curve in $Y^\infty_t$ joining $\infty$ and $B_x$ has to intersect every leaf of $\mathcal{L}$. Then, the width of $\mathcal{L}$ is at most some positive constant $K_n$ depending on $n$.
\end{claim2} 

\begin{proof}
    The set $B_x$ contains the outer boundary of a bubble of generation $t=q_{n+m''+2}$, so the harmonic measure of $B_x$ in $Y^\infty_t$ about $\infty$ is at least $d_\infty^{-t}$. The claim immediately follows from Proposition \ref{log-rule}.
\end{proof}

For each $i \in \{1,\ldots,k\}$, we will denote by $x_i$ and $y_i$ the pair of points in $\mathcal{O}$ such that $P_{n_i} \cap \partial Y^\infty = [x_i, y_i]$. By \ref{s3.claim2}, we can take $\mathbf{K}$ to be sufficiently high depending on $n$ and assume that the sublamination $\hat{\Omega}$ consisting of leaves of $\Omega$ that consecutively intersect $B_x$ for $x \in \mathcal{O}$ has width 
\begin{equation}
\label{ineq:new-wave}
W\left(\hat{\Omega}\right) \geq \frac{4}{5} \, W(\Omega).
\end{equation}
In particular, there exist pairwise disjoint proper laminations $\mathcal{G}_2, \mathcal{G}_3,\ldots,\mathcal{G}_{k-1}$ in $Y_t^\infty$ such that for every $i \in \{2,3,\ldots,k-1\}$, $\mathcal{G}_i$ is a \emph{restriction} of $\hat{\Omega}$ (refer to Appendix \ref{ss:extremal-width}) and connects $B_{x_i}$ to $B_{y_i}$. See Figure \ref{fig:bubbles-and-waves}. 

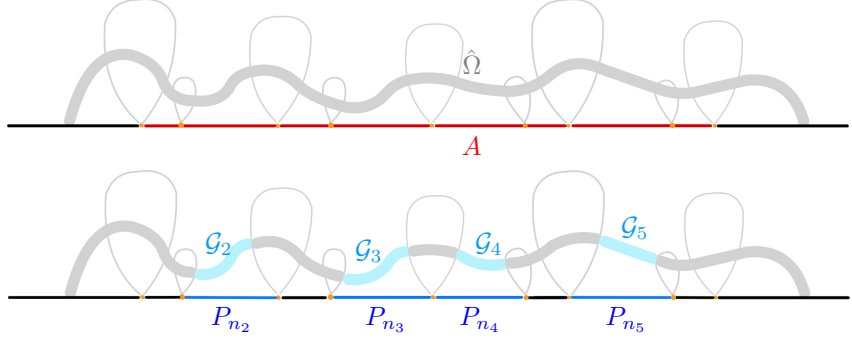
\begin{figure}
    \centering
    
    \begin{tikzpicture}
    \node[anchor=south west,inner sep=0] (image) at (0,0) {\includegraphics[width=0.95\linewidth]{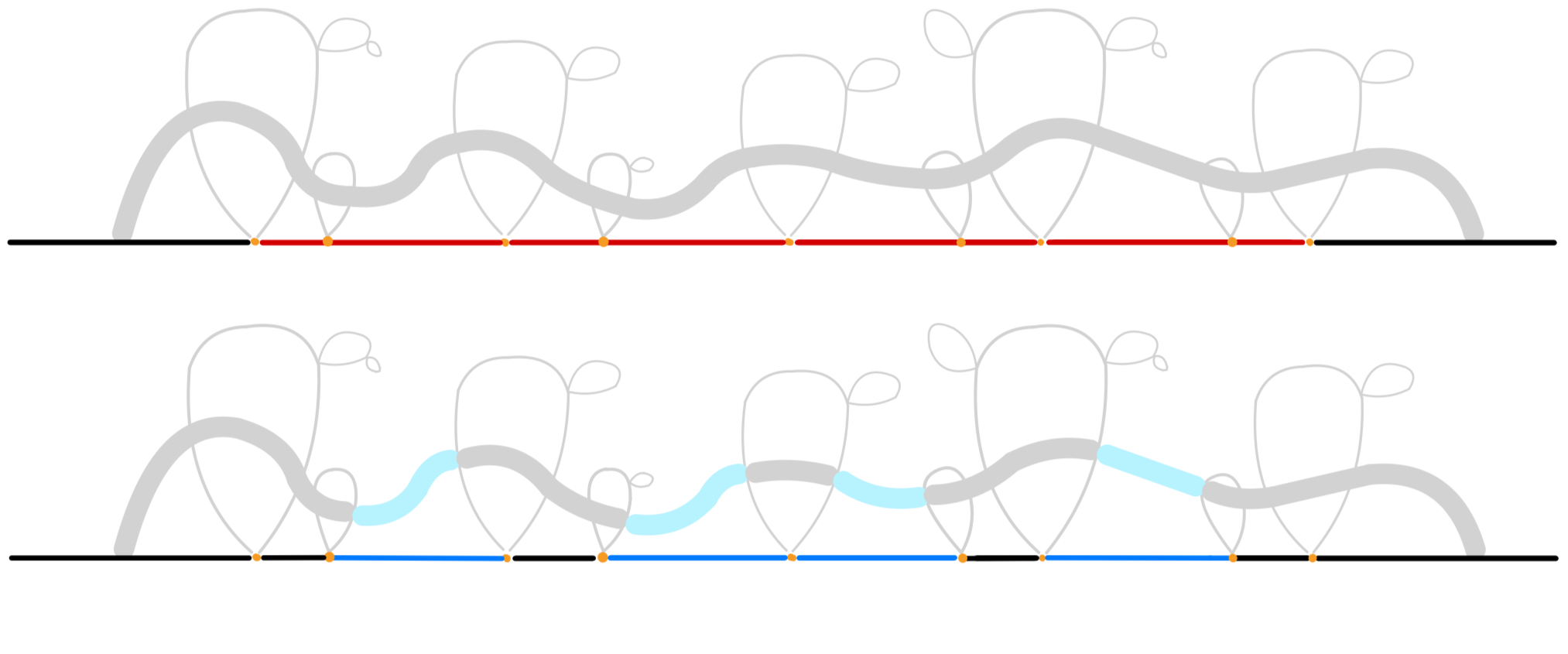}};
    \begin{scope}[
        x={(image.south east)},
        y={(image.north west)}
    ]
        \node [gray, font=\bfseries] at (0.56,0.80) {$\hat{\Omega}$};
        \node [red, font=\bfseries] at (0.53,0.575) {$A$};
        
        \node [blue!40!cyan, font=\bfseries] at (0.25,0.30) {$\mathcal{G}_{2}$};
        \node [blue!40!cyan, font=\bfseries] at (0.43,0.275) {$\mathcal{G}_{3}$};
        \node [blue!40!cyan, font=\bfseries] at (0.56,0.30) {$\mathcal{G}_{4}$};
        \node [blue!40!cyan, font=\bfseries] at (0.74,0.335) {$\mathcal{G}_{5}$};
        \node [blue, font=\bfseries] at (0.27,0.09) {$P_{n_2}$};
        \node [blue, font=\bfseries] at (0.45,0.09) {$P_{n_3}$};
        \node [blue, font=\bfseries] at (0.56,0.09) {$P_{n_4}$};
        \node [blue, font=\bfseries] at (0.73,0.09) {$P_{n_5}$};
    \end{scope}
\end{tikzpicture}
    
    \caption{The outer wave $\hat{\Omega}$ and the laminations $\mathcal{G}_{i}$'s connecting finite bubble chains attached to the endpoints of $P_{n_i}$'s.}
    \label{fig:bubbles-and-waves}
\end{figure}

\begin{claim3}
\namedlabel{s3.claim3}{Claim 3}
    There exists a positive constant $K_n$ depending on $n$ such that for every $i \in \{2,3,\ldots,k-2\}$, there exists a sublamination $\mathcal{G}_i^{\textnormal{new}}$ of $\mathcal{G}_i$ with width
    \[
        W(\mathcal{G}_i^{\textnormal{new}}) \geq W(\mathcal{G}_i) - K_n
    \]
    such that $f^t$ univalently pushes forward $\mathcal{G}_i^{\textnormal{new}}$ to an outer wave protecting $f^t(P_{n_i})$.
\end{claim3} 

\begin{proof}
    Consider a Riemann mapping $\psi : Y^\infty_t \to \C \backslash \overline{\D}$ fixing $\infty$ and equip $\partial \D$ with the harmonic measure about $\infty$. Consider the set $\mathcal{S}_i$ of leaves $\gamma$ in the lamination $\mathcal{G}_i$ such that $\psi(\gamma)$ protects an interval on $\partial{\D}$ of harmonic measure at least $d_\infty^{-t}$. By Proposition \ref{log-rule}, the width of $\mathcal{S}_i$ is at most some constant $K'_n>0$. We set $\mathcal{G}_i^{\textnormal{new}}$ to be $\mathcal{G}_i \backslash \mathcal{S}_i$ minus the outer $1$-buffer. By Proposition \ref{non-crossing-principle}, $\mathcal{G}_i^{\textnormal{new}}$ is disjoint from $\omega(\mathcal{G}_i^{\textnormal{new}})$ where $\omega$ is any non-trivial element of the deck group of $f^t : Y^\infty_t \to Y^\infty$.
\end{proof}

Among $\mathcal{G}_2, \ldots , \mathcal{G}_{k-1}$, suppose the widest one is $\mathcal{G}_s$ for some $s$. By Propositions \ref{series law} and \ref{harmonicsum},
\begin{equation}
\label{ineq:series-law-first}
    W\left(\hat{\Omega}\right) 
    \leq W(\mathcal{G}_2) \oplus \ldots \oplus W(\mathcal{G}_{k-1}) 
    \leq \frac{W(\mathcal{G}_s)}{k-2} .
\end{equation}
By (\ref{ineq:new-wave}), (\ref{ineq:series-law-first}), and the assumption that $\alpha \geq 1$,
\[
    W(\mathcal{G}_s) \geq (k-2) W\left(\hat{\Omega}\right) 
    \geq  \left( 5\alpha-2\right) \cdot \frac{4}{5} \, W(\Omega) 
    \geq 2.4 \, \alpha \: W(\Omega).
\]
By taking $\mathbf{K}$ to be sufficiently high depending on $n$, the sublamination $\mathcal{G}_s^{\textnormal{new}}$ described in \ref{s3.claim3} has width at least $2 \alpha W(\Omega)$. The image $\Omega' = f^t(\mathcal{G}_s^{\textnormal{new}})$ is the wave we are looking for. Indeed, the width of $\Omega'$ is at least $2\alpha \: W(\Omega)$ and, because of \ref{s3.claim1}, the length of $\Omega'$ is at least $ l_{n+m'}$.
\end{proof}

\subsection{Wide waves yield $\tau$-degeneration} 
\label{ss:wide-waves-yields-degeneration}
By an inductive argument, we can now obtain a $\tau$-degeneration out of a wide wave.

\begin{proof}[Proof of Proposition \ref{prop:wave-degeneration}] 
Fix $n \in \N$ and $\alpha \geq 1$. Let $m'$ and $\mathbf{K}$ be the constants from Lemma \ref{bubble-wave-argument} and set $m \in \N$ to be the smallest integer such that $\frac{\tau-1}{2} l_{n+m} \leq l_{n+m'} $. Let $\Omega$ be a wave of combinatorial length $\geq \alpha l_n$ and width $W(\Omega) \geq \mathbf{K}$.

\begin{claim}
    Either there exists a level $n+m$ combinatorial piece $J$ satisfying $W_\tau(J) \succ \alpha W(\Omega)$, or for every $t \geq 1$, there exists a wave $\Omega_t$ of length $\geq l_n$ and width
\begin{equation}
\label{ineq:wave-00}
W(\Omega_t) \geq \left( \frac{3}{2} \right)^{t} \alpha \, W(\Omega).
\end{equation}
\end{claim} 

\begin{proof}
    We will proceed by induction. Suppose there exists a wave $\Omega_t$ of length $\geq l_n$ satisfying (\ref{ineq:wave-00}) for some $t \in \N$. We will also include the initial case $t=0$, in which $\Omega_0:= \Omega$ has length $\geq \alpha l_n$ and width $W(\Omega)$. By Lemma \ref{bubble-wave-argument}, there is a wave $\Omega'_t$ of width
    \begin{equation}
    \label{ineq:wave-01}
    W(\Omega'_t) \geq \begin{cases}
    2 \: W(\Omega_t) & \text{ if } t\geq 1,\\
    2\alpha \: W(\Omega) & \text{ if } t=0,
    \end{cases}
    \end{equation}
    protecting a piece $J_t$ of length $l_{n+m'}$. Note that by (\ref{ineq:wave-00}) and (\ref{ineq:wave-01}), we have
    \begin{equation}
    \label{ineq:wave-02}
        W(\Omega'_t) \geq 2 \alpha \:W(\Omega).
    \end{equation}
    
    Let $I_{t+1}$ be the level $n$ combinatorial piece that shares the same combinatorial mid-segment as $J_t$. We present $\Omega'_t$ as $\Omega_{t+1} \cup \Omega''_t$ where $\Omega_{t+1}$ is the set of leaves of $\Omega'_t$ that protect $I_{t+1}$ and $\Omega''_t$ is the set of leaves that land on $I_{t+1} \backslash J_t$.
    
    If $W(\Omega_{t+1}) \geq \frac{3}{4} W(\Omega'_t)$, then by combining this with (\ref{ineq:wave-00}) and (\ref{ineq:wave-01}), the wave $\Omega_{t+1}$ satisfies (\ref{ineq:wave-00}) and we are done. Suppose instead that 
    \begin{equation}
        \label{ineq:wave-03}
    W(\Omega''_t) > \frac{1}{4} W(\Omega'_t).
    \end{equation}
    There exists a level $n+m$ combinatorial piece $J \subset \overline{I_{t+1} \backslash J_t}$ such that amongst every level $n+m$ subpiece of $\overline{I_{t+1}\backslash J_t}$, the width of leaves of $\Omega'_t$ landing on $J$ is the widest. Our choice of $m$ guarantees that leaves of $\Omega''_t$ that land on $J$ lie in $\mathcal{F}_\tau(J)$, yielding
    \begin{equation}
    \label{ineq:wave-04}
        W_\tau(J) \geq \frac{|J|}{|I_{t+1} \backslash J_t |} W(\Omega''_t) \succ W(\Omega''_t).
    \end{equation}
    Therefore, by combining (\ref{ineq:wave-02}), (\ref{ineq:wave-03}), and (\ref{ineq:wave-04}), we obtain $W_\tau(J) \succ \alpha W(\Omega)$.
\end{proof}

The proposition holds because if otherwise, the claim above would give us an infinite sequence of waves $\Omega_t$ of uniformly bounded length and exponentially increasing width, which contradicts the compactness of $\Hq$.
\end{proof}


\section{Spreading degeneration}
\label{sec:spreading-degeneration}

Recall from Proposition \ref{renormalization-tiling} that for any level $n$ piece $I$, the pieces $I$, $f(I)$, $f^2(I), \ldots, f^{q_{n+1}-1}(I)$ have pairwise disjoint interior.

\begin{definition}
    The \emph{level $n$ almost tiling} $\mathcal{I}$ generated by a level $n$ combinatorial piece $I \subset \Hq$ is the collection of iterated pieces $\{f^i(I)\}_{i=0, \ldots, q_{n+1}-1}$.
\end{definition}

In this section, we will spread a given $\lambda$-degeneration to an almost tiling consisting of pieces that are all comparably $\lambda$-degenerate relative to the original. Recall the threshold parameter $\thres_\lambda$ defined in \S\ref{sss:notation}.

\begin{proposition}
\label{spreading-lambda}
    For any $\Xi > 1$ and $\lambda > \tau$, there are some $\mathbf{K} = \mathbf{K}(\Xi, \lambda) >1$ and $\xi=\xi(\Xi) > 0$ such that if there is a $[K,\lambda]$-wide level $n$ combinatorial piece $I$ where $n \geq \thres_\lambda$ and $K \geq \mathbf{K}$, then either
    \begin{enumerate}[label=\textnormal{(\arabic*)}]
        \item there is a $[\Xi K,\tau]$-wide combinatorial piece of level $n$, or
        \item there is a level $n$ almost tiling consisting of $[\xi K, \lambda]$-wide pieces. 
    \end{enumerate}
\end{proposition}

In the proof, we will apply the Covering Lemma (Lemma \ref{covering-lemma}) to spread $\lambda$-degeneration around $\Hq$. We will introduce \emph{cuts} (Lemma \ref{cuts}) to bound the degree of the appropriate branched covering in terms of $\lambda$.

\subsection{Spreading $\tau$-degeneration} 
\label{ss:spreading-tau-degeneration}
We will first discuss what we can do with $\tau$-degeneration. This can be seen as a special case of Proposition \ref{spreading-lambda} when $\lambda=\tau$.

\begin{proposition}
    \label{spreading-tau}
    There are absolute constants $0< \varepsilon <1$ and $\mathbf{K}>1$ such that for any $[K,\tau]$-wide combinatorial piece $I \subset \Hq$ of level $n$ where $n \geq \thres_\tau$ and $K \geq \mathbf{K}$, every piece in the almost tiling generated by $f^2 (I)$ is $[\varepsilon K, \tau]$-wide.
\end{proposition}

To be more precise, the constant $\varepsilon$ above depends on the small separation constant $\tau = 10$. Throughout this paper, any dependence on $\tau$ will be suppressed.

We will apply Proposition \ref{transformation law} as the main tool to compare the $\tau$-widths of a piece $I$ and its iterate $f^i(I)$. This motivates us to first estimate the degree of $f^i$ near $\tau I$, which we can deduce in a more general way as follows.

\begin{lemma}
\label{counting-critical}
    Suppose $f^a : U \to U'$ is a branched covering map between two open disks $U$ and $U'$ in $\C^*$ where $a \leq q_{n+k}$ and $\overline{U'} \cap \Hq$ is a piece of length $\rho l_n$ for some positive integers $k$ and $n$, and some constant $\rho \geq 1$. Then, 
    \[
    \text{deg}(f^a: U \to U') \leq M
    \]
    for some $M = M(k,\rho)>1$.
\end{lemma}

\begin{proof}
    For $t=0,1,\ldots, a$, let $U_t := f^t(U)$. Observe that each $\overline{U_t} \cap \Hq$ must be a piece of length $\rho l_n$. Let $C>1$ be the constant from Proposition \ref{bounded-type}, then $ l_n \leq C^k l_{n+k}$. Since $a \leq q_{n+k}$, for every critical point $c \in \Hq$, there are at most $C^k \rho$ values of $t \in \{0,1,\ldots, a\}$ such that $U_t$ contains $c$. Since $f$ has $d-1$ free critical points counting multiplicity, the number of different pairs $(c,t)$ such that $U_t$ contains a free critical point $c$ is at most $C^k\rho(d-1)$. Therefore, the degree of $f^a : U \to U'$ is at most $2^{C^k\rho(d-1)}$.
\end{proof}

\begin{remark}
\label{omit-y0}
    In Case \ref{case:boundary-ring} (as outlined in \S\ref{ss:setup}), in order for the lemma above to work, we shall assume additionally that $U$ and $U'$ are disjoint from the connected component $\hat{Y}^0$ of $\RS \backslash \He$ containing $0$ so that every critical value of the mapping $f^a: U \to U'$ lies on the outer boundary $\Hq$. 
\end{remark}

Next, we have to pick the disk $U$ containing $\tau I$ carefully. In particular, we would like to restrain the local degree of an iterate $f^i$ on $U$ so that it is independent of $i$.

\begin{lemma}[Cuts]
\label{cuts}
    For any piece $I$ such that $I$, $f(I)$, and $f^2(I)$ are pairwise disjoint, there exist some $t \in \{0,1,2\}$ and a pair of closed rays $\gamma_0 \subset \overline{Y^0}$ and $\gamma_\infty \subset \overline{Y^\infty}$ connecting a point in $\left(f^t(I)\right)^c$ to $0$ and $\infty$ respectively such that the width of curves in $\RS \backslash (\Hq \cup \gamma_0 \cup \gamma_\infty)$ connecting $f^t(I)$ and $\gamma_0 \cup \gamma_\infty$ is at most $10$.
\end{lemma}

\begin{proof}
    Since $I$, $f(I)$, and $f^2(I)$ are pairwise disjoint, there is some $t \in \{0,1,2\}$ such that for $\bullet \in \{0,\infty\}$, the harmonic measure of $f^t(I) \cap \partial Y^\bullet$ in $Y^\bullet$ about $\bullet$ is less than $\frac{1}{2}$. Then, \cite[Chapter IV Theorem 5.2]{GM05} guarantees the existence of a pair of such rays $\gamma_0$ and $\gamma_\infty$ where the width of curves in $Y^0$ (resp. $Y^\infty$) connecting $f^t(I)$ and $\gamma_0$ (resp. $\gamma_\infty$) is at most $5$.
\end{proof}

The rays $\gamma_0$ and $\gamma_\infty$ satisfying the above will be called \emph{cuts} for the piece $f^t(I)$. These cuts will help us define the appropriate disks.

\begin{proof}[Proof of Proposition \ref{spreading-tau}]
   Let $I \subset \Hq$ be a $[K,\tau]$-wide combinatorial piece of level $n \geq \thres_\tau$ and let $I_s:= f^s(I)$ for any $s \geq 0$. Pick any integer $a \in [2, q_{n+1}+1]$. We can assume that there exist cuts $\gamma_0$ and $\gamma_\infty$ for $\tau I_a$. (Otherwise, replace $I_a$ with $I_{a-i}$ for some $i \in \{1,2\}$ and apply Proposition \ref{transformation law}.)
   
   Let $U'$ denote the open disk $\RS \backslash \left( (\tau I_a)^c \cup \gamma_0 \cup \gamma_\infty \right)$, and let $U$ be the connected component of $f^{-a}(U')$ containing $I$. By Proposition \ref{transformation law},
    \[
        K \leq W(U, I) \leq \text{deg}(f^a: U \to U') \cdot W(U',I_a).
    \]
    By Lemma \ref{counting-critical}, the inequality implies $W(U',I_a) \succ K$. Curves in $\mathcal{F}(U', I_a)$ connect $I_a$ to either $(\tau I_a)^c$ or the cuts $\gamma_0 \cup \gamma_\infty$. The width of those landing at $\gamma_0 \cup \gamma_\infty$ is at most $10$, so when $K \geq \mathbf{K}$ and $\mathbf{K}$ is sufficiently high, we have $W_\tau(I_a) \succ K$. 
\end{proof}

\begin{remark}
\label{removal}
    In Case \ref{case:boundary-ring}, we shall modify the proof above by replacing the topological disk $U'$ with $U' \backslash \hat{Y}^0$. The removal of $\hat{Y}^0$ is necessary in order to apply Lemma \ref{counting-critical} (see Remark \ref{omit-y0}), and harmless because the width of curves in $\mathcal{F}(U', I_a)$ that land on $\hat{Y}^0$ is negligible due to Lemma \ref{negligibility}. 
\end{remark}

\subsection{Spreading $\lambda$-degeneration}
\label{sec:spreading-lambda-degeneration}
Even though the proof of the previous lemma can also be applied to $\lambda$-degeneration, the corresponding multiplicative factor would depend on $\lambda$. We will employ a different spreading approach by applying the Covering Lemma as follows. (See \cite[\S8.1]{DL22} in the case of quadratic Siegel disks.)

\begin{proof}[Proof of Proposition \ref{spreading-lambda}]
Let $I \subset \Hq$ be a $[K,\lambda]$-wide combinatorial piece of level $n \geq \thres_\lambda$, where $K \geq \mathbf{K}$, and let $I_s:= f^s(I)$ for any $s \geq 0$. 

Pick an integer $a \in [2, q_{n+1}+1]$. Since $n \geq \thres_\lambda$, by Lemma \ref{cuts}, there exist cuts $\gamma_0$ and $\gamma_\infty$ for $\lambda I_b$ for some $b \in \{a-2,a-1,a\}$. Then, consider the iterate $f^b : (U,\Lambda, I) \to (V, B, I_b)$ where
\begin{itemize}
    \item[$\rhd$] $V := \RS \backslash \left( (\lambda I_b)^c \cup \gamma_0 \cup \gamma_\infty \right)$;
    \item[$\rhd$] $B := V \backslash (\tau I_b)^c$;
    \item[$\rhd$] $U:=$ the connected component of $f^{-b}(V)$ containing $I$;
    \item[$\rhd$] $\Lambda :=$ the connected component of $f^{-b}(B)$ containing $I$.
\end{itemize}
By Lemma \ref{counting-critical},
\[
\text{deg}(f^b:\Lambda \to B) \leq M(\tau), \quad \text{and} \quad \text{deg}(f^b:U \to V) \leq M(\lambda).
\]

Fix the constant $\Xi > 1$. Since $\partial U$ contains $(\lambda I)^c$, we have $W(U, I) \geq K$. By Lemma \ref{covering-lemma}, for sufficiently high $\mathbf{K}$ depending on $\Xi$ and $\lambda$, either 
\[
W(B, I_b)> (d^2 \Xi +1) K \quad \text{or} \quad W(V, I_b) > \xi_1 K,
\]
where $\xi_1 \in (0,1)$ depends only on $\Xi$. By Lemma \ref{cuts}, the width of curves in $\mathcal{F}_\lambda(I_b)$ landing at the cuts $\gamma_0 \cup \gamma_\infty$ is at most $10$. Therefore, for sufficiently high $\mathbf{K}$, either
\[
W_\tau(I_b) \geq d^2 \Xi K \quad \text{or} \quad W_\lambda(I_b) \geq \xi_2 K,
\]
for some $\xi_2 \in (0,1)$ depending only on $\Xi$. After pushing forward by $f^{a-b}$, we conclude that the piece $I_a$ is either $[\Xi K, \tau]$-wide or $[\xi K, \lambda]$-wide, where $\xi = d^{-2} \xi_2$. Therefore, if there is no $2\leq a \leq q_{n+1}+1$ such that $I_a$ is $[\Xi K, \tau]$-wide, then $I_2$ generates an almost tiling consisting of $[\xi K, \lambda]$-wide pieces.
\end{proof}

\begin{remark}
    \label{removal1}
    In Case \ref{case:boundary-ring}, the proof above needs to be modified by replacing the disk $V$ with $V \backslash \hat{Y}^0$, similar to Remark \ref{removal}.
\end{remark}


\section{Trading \texorpdfstring{$\lambda$}{l}-degeneration for a \texorpdfstring{$\tau$}{t}-degeneration}
\label{sec:trading}

Given a $\lambda$-degeneration, the previous section tells us how to spread and obtain an almost tiling of $\lambda$-degenerate pieces. Next, we would like to convert such an almost tiling into a much larger $\tau$-degeneration with a multiplicative factor that grows with $\lambda$. The main result of this section is the following theorem.

\begin{theorem}
\label{demotion}
    For all sufficiently large $\lambda$, there are parameters $\step_\lambda, \mathbf{K}, \Pi_\lambda>1$ all depending on $\lambda$ where $\lim_{\lambda \to \infty}\Pi_\lambda = +\infty$ such that if 
    \begin{center}
        there is an almost tiling $\mathcal{I}$ consisting of $[K,\lambda]$-wide pieces of level $n\geq \thres_\lambda$
    \end{center}
    where $K\geq \mathbf{K}$, then 
    \begin{center}
        there is a $[\Pi_\lambda K, \tau]$-wide combinatorial piece $J$ of level $n' \geq \thres_\lambda$
    \end{center}
    where $|n'-n| \leq \step_\lambda$.
\end{theorem}

The proof will be split into two cases:
\begin{description}[leftmargin=0.5in]
    \item[deep case] $n \geq \thres_\lambda + \step_\lambda$;
    \item[shallow case] $\thres_\lambda \leq n < \thres_\lambda + \step_\lambda$.
\end{description}
The threshold level $\thres_\lambda$ is essential because we will apply the theorem above inductively in Section \S\ref{sec:a-priori-bounds}.

\subsection{Deep case} 
\label{ss:trading-deep}
A deep almost tiling can be handled through a straightforward application of the Quasi-Additivity Law (Lemma \ref{quasi-additivity-law}).

\begin{proof}[Proof of Theorem \ref{demotion} in the deep case]

Assume $\lambda \gg \tau^2$ and set $N := \lfloor \frac{\lambda}{3\tau^2} \rfloor$. Suppose there is a level $n$ almost tiling $\mathcal{I}$ consisting of $[K,\lambda]$-wide pieces where $K \geq \mathbf{K}$. There exists a sequence $\{I_j\}_{j = 1,\ldots, N}$ of distinct pieces in the almost tiling $\mathcal{I}$, labelled in consecutive order, such that for every $j \in \{1,2,\ldots, N-1\}$, $I_j$ and $I_{j+1}$ have controlled combinatorial distance: 
\[
(\tau-1) l_n < \dist(I_j, I_{j+1}) \leq  \tau l_n.
\]
This ensures that $\tau I_j$ and $\cup_{i\neq j} I_i$ are always disjoint but not too far apart.

Let $P$ be the unique shortest piece containing $\bigcup_{j=1}^N I_j$. We set $\step_\lambda$ to be the largest integer less than $n$ such that $|P| \geq l_{n-\step_\lambda}$. Our choice of $N$ ensures that each $\lambda I_j$ contains $\tau P$. Consider the disk $S := \RS \backslash \bigcup_{j=1}^N (\lambda I_j)^c$. Following \cite{A06}, we will use the notation $\mathcal{H} < \mathcal{G}$ to denote that $\mathcal{G}$ \emph{overflows} $\mathcal{H}$. (See Appendix \ref{ss:extremal-width}.) Then, for every $j \in \{1,\ldots, N \}$,
\[
\mathcal{F}(S, I_j) < \mathcal{F}_{\lambda}(I_j), 
\quad \mathcal{F}_\tau (I_j) < \mathcal{F}\left(S\backslash \bigcup_{i\neq j} I_i, I_j\right), 
\quad \mathcal{F}_\tau (P) < \mathcal{F}\left(S, \bigcup_{i=1}^N I_i\right).
\]

We are under the assumption that for each $j$,
\[
W(S, I_j) \geq W_\lambda(I_j) \geq K.
\]
For sufficiently large $\mathbf{K}$, we can apply Lemma \ref{quasi-additivity-law} and obtain
\[
\max\{W_\tau(P), W_\tau(I_1),\ldots, W_\tau(I_N)\} \geq \frac{1}{\sqrt{2N}}\sum_{j=1}^N W(S,I_j) \geq \sqrt{\frac{N}{2}}K.
\]
Since $N \asymp \lambda$, we conclude that either 
\[
W_\tau(P) \succ\sqrt{\lambda} K, \quad \text{or} \quad W_\tau(I_j) \succ \sqrt{\lambda} K
\]
for some $j$. If the former, there exists a combinatorial subpiece $J \subset P$ of level $n-\step_\lambda$ and $\tau$-width $W_\tau(J) \succ \sqrt{\lambda} K$.
\end{proof}

Notice that, in the proof above, the piece $J$ is longer than the original piece $I$. This justifies the need for a different approach when $n$ is shallow.

\subsection{Shallow case} 
\label{ss:trading-shallow}
The main ingredient in the shallow case is Proposition \ref{prop:wave-degeneration}. Given a wide lamination, we are split into two different situations: either it forms combinatorially long wide waves or it intersects $\Hq$ frequently in a snake-like pattern (see Figure \ref{fig:no-wide-waves}). Both cases will produce a large $\tau$-degeneration.

\begin{proof}[Proof of Theorem \ref{demotion} in the shallow case] 

Fix $\mathbf{K}$, and suppose there is a $[K,\tau]$-wide piece $I$ of shallow level $n$ (e.g. by picking any piece from the almost tiling $\mathcal{I}$ in the hypothesis) where $K \geq \mathbf{K}$. 

If there is a wave $\Omega$ of width $\geq K/10$ and length $\geq \sqrt[3]{\lambda} l_n$ , then by Proposition \ref{prop:wave-degeneration}, there is a level $n+m$ a combinatorial piece $J$ such that
\[
W_\tau(J) \succ \sqrt[3]{\lambda} \cdot \frac{K}{10}
\]
and we are done, assuming $\mathbf{K}$ is sufficiently high depending on $\thres_\lambda + \step_\lambda$, and $\lambda$ is sufficiently large such that $\step_\lambda \geq m$. As such, we proceed under the following assumption. 
\vspace{0.1in}

\noindent \emph{No-wide-wave assumption:} Every wave of length $\geq \sqrt[3]{\lambda} l_n$ has width $\leq K/10$. 
\vspace{0.1in}

Let us decompose $(\lambda I)^c$ into $T^+ \cup T \cup T^-$, where $T^+$ and $T^-$ are the left and rightmost pieces of $(\lambda I)^c$ of length $\sqrt[3]{\lambda} l_n$, and $T:=(\lambda I)^c \backslash (T^+ \cup T^-)$. Let $\mathcal{F}$ be the set of leaves of the canonical radial foliation of the conformal annulus $\RS \backslash (I \cup (\lambda I)^c)$ that never restrict to curves protecting $T^\dagger$ from $\bullet$ for any $\bullet \in\{0,\infty\}$ and $\dagger \in \{+,-\}$. 

\begin{claim1} 
\namedlabel{s5.claim1}{Claim 1}
    Every leaf of $\mathcal{F}$ connects $I$ and $T^+ \cup T^-$.
\end{claim1} 

\begin{proof}
    If a radial leaf lands on $T$, then it must restrict to a subcurve that protects either $T^+$ or $T^-$.
\end{proof}

We can decompose $\mathcal{F}$ into $\mathcal{F}^+ \cup \mathcal{F}^-$ according to whether leaves land on $T^+$ or $T^-$. From the no-wide-wave assumption, the width $W(\mathcal{F})$ of $\mathcal{F}$ is at least $6K/10$. Without loss of generality, assume that $\mathcal{F}^+$ is wider, so then 
\begin{equation}
    \label{ineq-f+}
W\left(\mathcal{F}^+\right) \geq \frac{3}{10}K.
\end{equation}

Recall the notion of conformal rectangles and buffers from Appendix \ref{ss:extremal-width}. We say that a lamination $\mathcal{L}$ is \emph{rectangular} if it is a sublamination of the vertical foliation of a conformal rectangle $R$. Moreover, a sublamination of a rectangular lamination $\mathcal{L}$ is a \emph{buffer} of $\mathcal{L}$ if it is the set of leaves of $\mathcal{L}$ that lie in a buffer of $R$.

\begin{claim2}
    $\mathcal{F}^+$ is rectangular\footnote{In fact, $\mathcal{F}^+$ is the vertical foliation of a conformal rectangle. An approach similar to \ref{s5.claim3} can be used to prove this.}.
\end{claim2}

\begin{proof}
By construction, the set $\tilde{\mathcal{F}}^+$ of leaves in the radial foliation of $\RS \backslash (I \cup (\lambda I)^c)$ that land on $T^+$ forms a single conformal rectangle. By \ref{s5.claim1}, $\mathcal{F}^+$ is a sublamination of $\tilde{\mathcal{F}}^+$.
\end{proof}

Consider a finite sequence of distinct pieces $I_0:=I, I_1,\ldots, I_N$ labelled in consecutive order such that for each $j=1,\ldots, N$,
\begin{enumerate}
    \item[(i)] $I_j$ is a subpiece of $\lambda I$ located between $I$ and $T^+$;
    \item[(ii)] $|I_j| = \sqrt[3]{\lambda} l_n$;
    \item[(iii)] $I_j$ is of distance at least $\frac{\tau-1}{2}l_n$ away from $I_i$ for all $i\neq j$.
\end{enumerate}
We pick $N$ to be the maximum possible integer such that (i)-(iii) holds.

\begin{claim3}
\namedlabel{s5.claim3}{Claim 3}
    For any $\bullet \in \{0,\infty\}$ and $j \in \{1,\ldots, N\}$, the set $\mathcal{F}^+_{j,\bullet}$ of leaves of $\mathcal{F}^+$ that contain a subcurve protecting $I_j$ from $\bullet$ is a buffer of $\mathcal{F}^+$.
\end{claim3}

\begin{proof}
    Every leaf of $\mathcal{F}^+$ cannot contain a subcurve protecting $(\lambda I)^c$, because otherwise it would protect $T^+$. As such, for each $\bullet \in \{0,\infty\}$, there exists a ray $\sigma^\bullet$ in $Y^\bullet$ that is disjoint from $\mathcal{F}^+$ and connects $\bullet$ and $T$. Let $D := \RS \backslash \big( (\lambda I)^c \cup \overline{\sigma^0 \cup \sigma^\infty}\big)$.
    
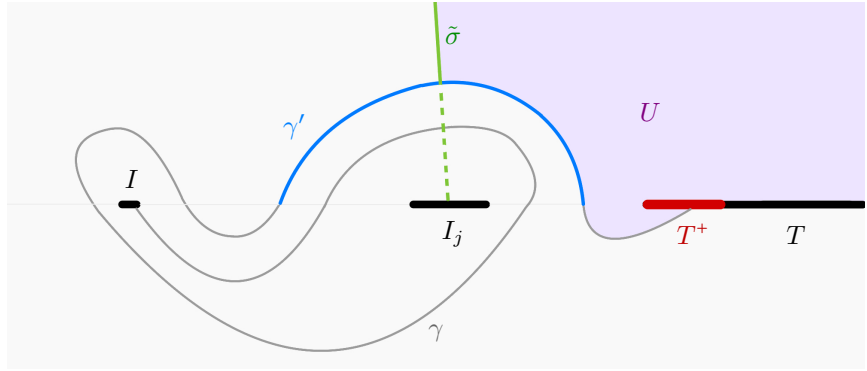
\begin{figure}
    \centering    
    \begin{tikzpicture}
    \node[anchor=south west,inner sep=0] (image) at (0,0) {\includegraphics[width=0.9\linewidth]{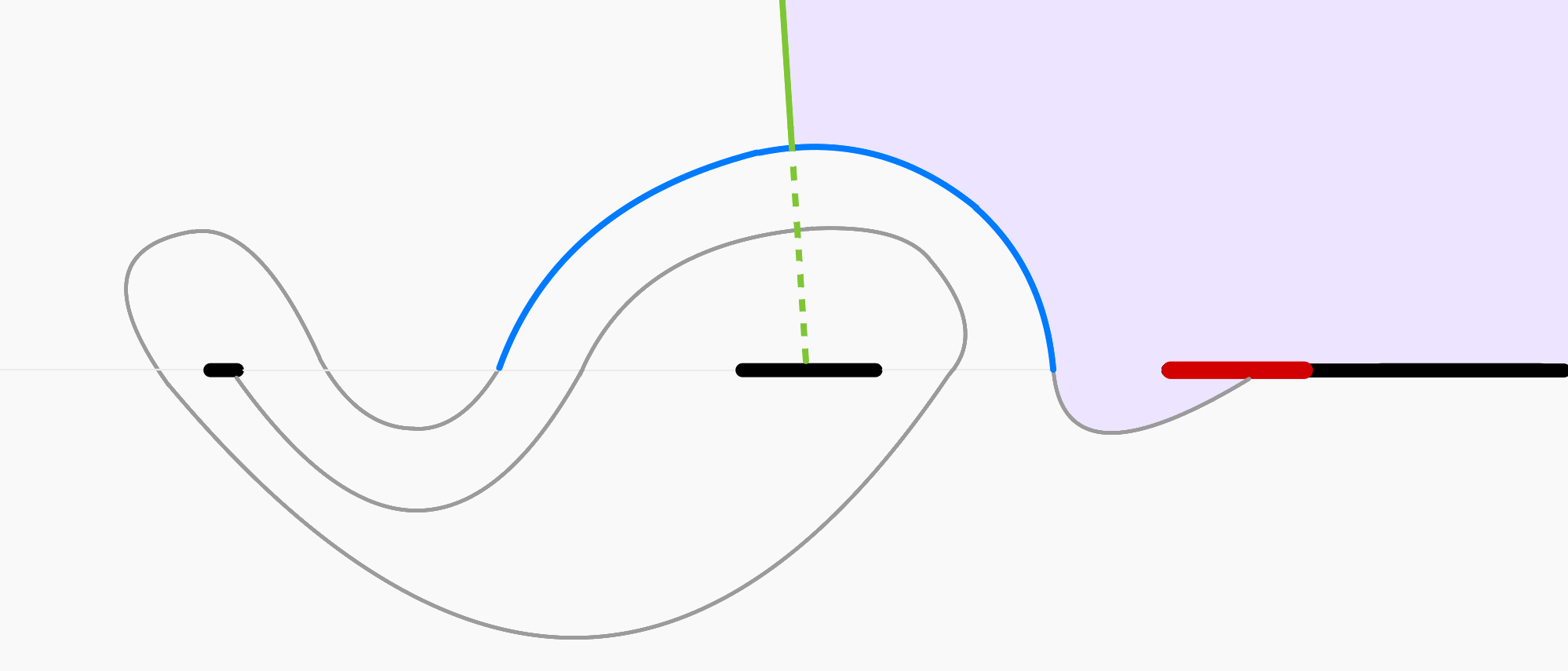}};
    \begin{scope}[
        x={(image.south east)},
        y={(image.north west)}
    ]
        \node [black, font=\bfseries] at (0.145,0.52) {$I$};
        \node [black, font=\bfseries] at (0.52,0.37) {$I_j$};
        \node [blue!50!cyan, font=\bfseries] at (0.335,0.66) {$\gamma'$};
        \node [red!50!blue, font=\bfseries] at (0.75,0.70) {$U$};
        \node [red!75!black, font=\bfseries] at (0.8,0.37) {$T^+$};
        \node [black, font=\bfseries] at (0.92,0.365) {$T$};
        \node [green!55!black, font=\bfseries] at (0.52,0.91) {$\tilde{\sigma}$};
        \node [white!40!black, font=\bfseries] at (0.50,0.10) {$\gamma$};
    \end{scope}
\end{tikzpicture}
    
    \caption{Domain $U$ defined by the ray $\sigma$ and the leaf $\gamma$.}
    \label{fig:u-gamma-sigma}
\end{figure}
    
    Suppose a leaf $\gamma$ of $\mathcal{F}^+$ contains a subcurve protecting $I_j$ from $\bullet$; label by $\gamma'$ the corresponding subcurve of $\gamma$ that is furthest from $I_j$. Pick any simple curve $\sigma$ in $D \cap Y^\bullet$ connecting $\bullet$ and a point on $I_j$. Then, $\sigma$ intersects $\gamma'$ and contains a subcurve $\tilde{\sigma}$ that connects $\bullet$ and a point $w \in \gamma' \cap \sigma$ and is disjoint from $\gamma$ away from $w$. Clearly, $\tilde{\sigma}$ splits the disk $D\backslash (I \cup \gamma)$ into two components, one of which, labelled by $U$, has closure that is disjoint from $I$. See Figure \ref{fig:u-gamma-sigma}.
    
    The leaf $\gamma$ splits $\mathcal{F}^+$ into two rectangular sublaminations on opposite sides of $\gamma$. One of the sublaminations, labelled by $\mathcal{F}^+_\gamma$, has support that intersects $U$. Since every leaf of $\mathcal{F}^+_\gamma$ must land on $I$ and avoid $\gamma$, then every leaf of $\mathcal{F}^+_\gamma$ must intersect $\sigma$. As $\sigma$ is arbitrary, this implies that every leaf of $\mathcal{F}^+_\gamma$ contains a subcurve protecting $I_j$ from $\bullet$. Then, the claim follows from the fact that $\mathcal{F}^+_\gamma$ is a buffer of $\mathcal{F}^+$.
\end{proof}

Let $\mathcal{G}$ be the sublamination of $\mathcal{F}^+$ consisting of all leaves that intersect all the $I_j$'s in consecutive order ($I_{j-1}$ before $I_j$). If a leaf $\gamma$ of $\mathcal{F}^+$ is not in $\mathcal{G}$, then it must contain a subcurve protecting some $I_j$. Therefore, by \ref{s5.claim3}, there exist pairs $j,k \in \{1,\ldots, N\}$ and $\sharp, \flat \in \{0,\infty\}$ such that $\mathcal{F}^+ \backslash \mathcal{G}$ is contained in a union of two maximal buffers $\mathcal{F}^+_{j,\sharp}$ and $\mathcal{F}^+_{k,\flat}$. In particular, $\mathcal{F}^+ \backslash \mathcal{G}$ overflows a union of two waves of length at least $\sqrt[3]{\lambda} l_n$. By (\ref{ineq-f+}) and the no-wide-wave wave assumption,
\[
     W(\mathcal{G}) \geq \frac{K}{10}.
\]

\begin{figure}
    \centering
    
    \begin{tikzpicture}
    \node[anchor=south west,inner sep=0] (image) at (0,0) {\includegraphics[width=0.8\linewidth]{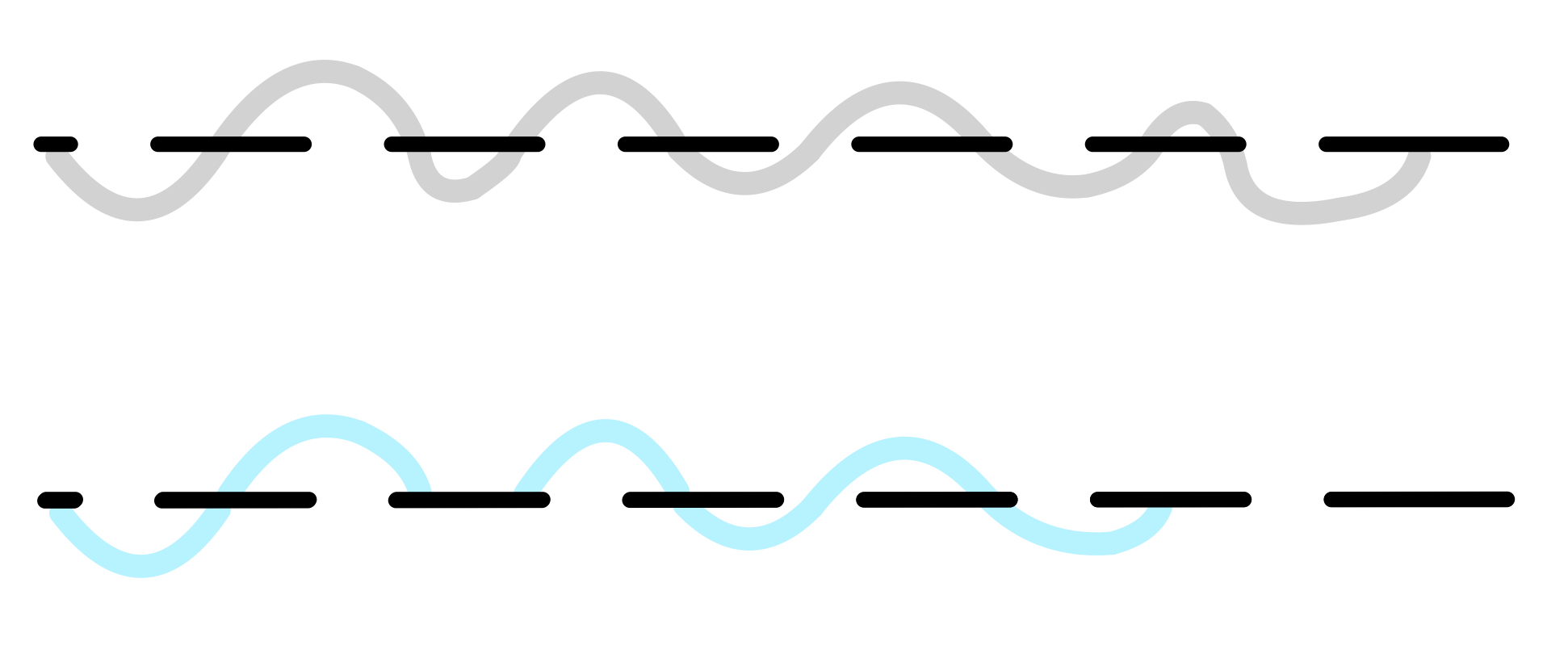}};
    \begin{scope}[
        x={(image.south east)},
        y={(image.north west)}
    ]
        \node [black, font=\bfseries] at (0.04,0.67) {$I_0$};
        \node [black, font=\bfseries] at (0.16,0.67) {$I_1$};
        \node [black, font=\bfseries] at (0.30,0.67) {$I_2$};
        \node [black, font=\bfseries] at (0.44,0.67) {$I_3$};
        \node [black, font=\bfseries] at (0.59,0.67) {$I_4$};
        \node [black, font=\bfseries] at (0.75,0.67) {$I_5$};
        \node [black, font=\bfseries] at (0.91,0.67) {$T^+$};
        
        \node [gray, font=\bfseries] at (0.25,0.91) {$\mathcal{G}$};
        \node [blue!40!cyan, font=\bfseries] at (0.06,0.10) {$\mathcal{G}_{1}$};
        \node [blue!40!cyan, font=\bfseries] at (0.22,0.405) {$\mathcal{G}_{2}$};
        \node [blue!40!cyan, font=\bfseries] at (0.40,0.40) {$\mathcal{G}_{3}$};
        \node [blue!40!cyan, font=\bfseries] at (0.48,0.11) {$\mathcal{G}_{4}$};
        \node [blue!40!cyan, font=\bfseries] at (0.69,0.105) {$\mathcal{G}_{5}$};
    \end{scope}
\end{tikzpicture}
    
    \caption{The sublamination $\mathcal{G} \subset \mathcal{F}^+$ intersects all $I_i$'s in order.}
    \label{fig:no-wide-waves}
\end{figure}
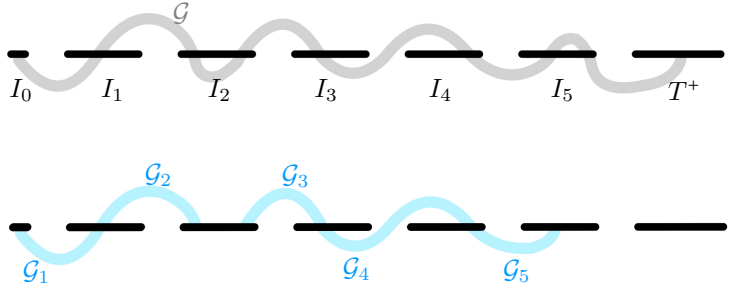

There exist pairwise disjoint laminations $\mathcal{G}_1, \ldots , \mathcal{G}_N$ such that each $\mathcal{G}_j$ is a restriction of $\mathcal{G}$ and connects $I_{j-1}$ and $I_j$. See Figure \ref{fig:no-wide-waves} for illustration. Suppose $\mathcal{G}_s$ has the largest width amongst all the $\mathcal{G}_i$'s. By Propositions \ref{series law} and \ref{harmonicsum}, 
\[
W(\mathcal{G}_s) \geq \frac{N}{10} K.
\]
Since the gaps between the $I_j$'s are at least $\frac{\tau -1 }{2} l_n$ in length, there must be a level $n$ combinatorial subpiece $J \subset I_s$ such that 
\[
W_\tau (J) \geq \frac{1}{\sqrt[3]{\lambda}} W(\mathcal{G}_s) .
\]
For sufficiently high $\lambda$, the maximum possible value of $N$ satisfies $N \asymp \sqrt[3]{\lambda^2}$. By combining the two inequalities above, we are done:
\[
W_\tau (J) \geq \frac{N}{10 \sqrt[3]{\lambda}} K \asymp \sqrt[3]{\lambda} K. \qedhere
\]
\end{proof}


\section{Amplifying \texorpdfstring{$\tau$}{t}-degeneration}
\label{sec:amplifying-tau-degeneration}
In this section, we work our way towards the amplification of a $\tau$-degeneration. More precisely, we aim to find a way to promote a $\tau$-degeneration in $\Hq$ into either a significantly larger $\tau$-degeneration or a comparable $\lambda$-degeneration. Unlike the previous section, the multiplicative factor will be independent of $\lambda$.

\begin{theorem}
    \label{promotion}
    There are absolute constants $\chi \in (0,1)$ and $\step \in \N$ such that for sufficiently large $\lambda $, there is some $\mathbf{K} = \mathbf{K}(\lambda) >1$ such that if 
    \begin{center}
        there is a $[K,\tau]$-wide combinatorial piece $I$ of level $n\geq \thres_\lambda$
    \end{center}
    where $K \geq \mathbf{K}$, then 
    \begin{center}
        there is a combinatorial piece $J$ of level $n'\geq \thres_\lambda$, where $|n'-n| \leq \step$, \\
        that is either $[\chi K, \lambda]$-wide or $[2K, \tau]$-wide.
    \end{center}
\end{theorem}

Similar to the previous section, we shall split the proof into two cases:
\begin{description}
    \item[deep case] $n\geq \thres_\lambda + \step$;
    \item[shallow case] $\thres_\lambda \leq n < \thres_\lambda + \step$.
\end{description}
This splitting is slightly different than the one presented in \S\ref{sec:trading}. Here, the increment $\step$ is smaller than the increment $\step_\lambda$ in Theorem \ref{demotion}.

\subsection{Shallow case}
\label{ss:amplifying-shallow}
The shallow case can again be handled using waves almost the exact same way as our treatment in \S\ref{ss:trading-shallow}. Any repeated details will be spared.

\begin{proof}[Proof of Theorem \ref{promotion} in the shallow case]

Fix $\mathbf{K}$ and suppose $I$ is a $[K,\tau]$-wide level $n$ combinatorial piece where $K \geq \mathbf{K}$ and $n$ is shallow. Fix a pair of positive integers $m'$ and $m''$; both are independent of $\lambda$ and will be determined later.

Assume that $m''$ is high enough such that $l_{n+m'} \geq \tau l_{n+m'+m''}$. If there is a wave of width $\geq K/10$ and combinatorial length $\geq l_{n+m'}$, then by Proposition \ref{prop:wave-degeneration}, there is a level $n+m+ m'+m''$ combinatorial piece $J$ such that
    \[
    W_\tau(J) \succ \frac{ l_{n+m'}}{l_{n+m'+m''}} \cdot \frac{K}{10} \succ \tilde{C}^{m''} K,
    \]
for some absolute constants $\tilde{C}>1$ and $m \in \N$. By picking a sufficiently high $m''$ such that $J$ is $[2K, \tau]$-wide, and by picking the threshold increment $\step \in \N$ such that $\step \geq m+m'+m''$, we are done. As such, we will proceed under the following assumption. 
\vspace{0.1in}

\noindent \emph{No-wide-wave assumption:} Every wave of length $\geq l_{n+m'}$ has width $\leq K/10$.
\vspace{0.1in}

Let $T^+$ and $T^-$ be the leftmost and rightmost level $n+m'$ combinatorial subpieces of $(\tau I)^c$. Let $\mathcal{F}$ be the set of leaves of the canonical radial foliation of the conformal annulus $\RS \backslash (I \cup (\tau I)^c)$ that never restrict to curves protecting $T^\dagger$ from $\bullet$ for any $\bullet \in\{0,\infty\}$ and $\dagger \in \{+,-\}$. Since leaves of $\mathcal{F}$ must connect $I$ and $T^+ \cup T^-$, we can decompose $\mathcal{F}$ into $\mathcal{F}^+ \cup \mathcal{F}^-$ according to whether leaves land on $T^+$ or $T^-$. Without loss of generality, assume that $\mathcal{F}^+$ is wider.

Consider a finite sequence of distinct combinatorial pieces $I_0:=I, I_1,\ldots, I_N$ labelled in consecutive order such that for each $j=1,\ldots, N$,
\begin{enumerate}
    \item[(i)] $I_j$ is a subpiece of $\tau I$ located between $I$ and $T^+$;
    \item[(ii)] $I_j$ is of level $n+m'$;
    \item[(iii)] $I_j$ is of distance at least $\frac{\tau-1}{2}l_{n+m'}$ away from $I_i$ for all $i\neq j$.
\end{enumerate}
We pick $N$ to be the maximum possible integer such that (i)-(iii) holds.

Similar to the argument in Section \S\ref{ss:trading-shallow}, the no-wide-wave assumption implies that there exists some $s \in \{1,\ldots, N\}$ and a lamination $\mathcal{G}_s$ connecting $I_{s-1}$ and $I_s$ such that 
\[
W(\mathcal{G}_s) \geq \frac{N}{10}K.
\]
The $I_j$'s are constructed such that $(\tau I_s)^c$ contains every $I_i$ for $i \neq s$. In particular, $\mathcal{G}_s$ overflows $\mathcal{F}_\tau(I_s)$ and thus the piece $I_s$ is $[\frac{N}{10}K, \tau]$-wide. Since $N \succ \tilde{C}^{m'}$, we can pick $m'$ to be just high enough such that $N \geq 20$. Hence, $I_s$ is a $[2 K , \tau]$-wide combinatorial piece of level $n+m'$.
\end{proof}

\subsection{Deep case}
\label{ss:amplifying-deep}
In the deep case, our approach below is inspired by Kahn's push-forward argument in \cite[\S7]{K06}. The proof below contains a series of reductive steps before we finally adapt the push-forward argument at the very end. We would like to emphasize that Lemma \ref{positive entropy} will be established in \S\ref{sec:loss-of-horizontal-width} (Remark \ref{rem:final-proof}).

\begin{proof}[Proof of Theorem \ref{promotion} in the deep case]
    Suppose there is a $[K,\tau]$-wide combinatorial piece in $\Hq$ of some deep level $n$ with $K \geq \mathbf{K}$. By Proposition \ref{spreading-tau}, we have a level $n$ almost tiling $\mathcal{I}$ consisting of $[\varepsilon K, \tau]$-wide pieces for some absolute constant $0 < \varepsilon <1$.
    
\begin{lemma}[Localization of $\tau$-degeneration]
\label{introducing-rho}
        There are absolute constants $\rho$, $m_0$, $m_* \in \mathbb{N}_{>1}$, where $\rho \gg\tau$, such that for sufficiently large $\mathbf{K}$ and for $n \geq m_*$, either
    \begin{enumerate}[label=\textnormal{(\arabic*)}]
        \item there is a $[2K, \tau]$-wide combinatorial piece of level between $n-m_0$ and $n$, or
        \item there is some $L^* \in \mathcal{I}$ such that the width of curves in $\mathcal{F}_\tau(L^*)$ that land on $\rho L^* \cap ( \tau L^*)^c$ is greater than $\varepsilon K/2$.
    \end{enumerate}
\end{lemma}
    
    Roughly speaking, if (2) fails, then we apply the Quasi-Additivity Law to the family $\mathcal{F}_\rho(I)$ (for a fixed $\rho$) to obtain (1) in a way that is similar to Section \S\ref{ss:trading-deep}.
    
\begin{proof}
    Fix $m_*$. There exists a finite sequence of distinct pieces $I_1, \ldots, I_N$ in $\mathcal{I}$ labelled in consecutive order such that any pair of adjacent pieces $I_j$ and $I_{j+1}$ have bounded combinatorial distance:
    \[
    (\tau-1) l_n \leq \dist(I_j, I_{j+1}) \leq \tau l_n.
    \]
    This condition ensures that for each $j$, $\tau I_j$ and $\cup_{i\neq j} I_i$ are disjoint but not too far apart. The integer $N \geq 2$ will be specified later, but nonetheless it must be bounded above by some constant depending on $m_*$. 
    
    Let $P$ be the unique shortest piece containing $\bigcup_{j=1}^N I_j$. We set $m_0 = m_0(N)$ to be the largest integer such that $|P| \geq l_{n-m_0}$. Also, set $\rho=\rho(N)$ to be the smallest integer such that for every $j$, $\rho I_j$ contains $\tau P$. Let $S := \RS \backslash \bigcup_{j=1}^N (\rho I_j)^c$. We will again use the notation $\mathcal{H} < \mathcal{G}$ to denote that $\mathcal{G}$ overflows $\mathcal{H}$. Then, for every $j \in \{1,\ldots, N \}$,
    \[
    \mathcal{F}(S, I_j) < \mathcal{F}_{\rho}(I_j), 
    \quad \mathcal{F}_\tau (I_j) < \mathcal{F}\left( S\backslash \bigcup_{i\neq j} I_i, I_j \right), 
    \quad \mathcal{F}_\tau (P) < \mathcal{F}\left(S, \bigcup_{i=1}^N I_i \right).
    \]

    Suppose (2) does not hold. For each $j$, the width of curves connecting $I_j$ and $(\rho I_j)^c$ exceeds $\varepsilon K/2$ and consequently, 
    \[
        W(S, I_j) \geq W_{\rho}( I_j) \geq \frac{ \varepsilon K}{2}.
    \]
    For sufficiently large $\mathbf{K}$, we can apply Lemma \ref{quasi-additivity-law} and obtain
    \[
    \max\{W_\tau(P), W_\tau( I_1), \ldots, W_\tau( I_N) \} \geq \frac{1}{\sqrt{2N}} \sum_{j=1}^N W(S, I_j) \succ \sqrt{N}  K.
    \]
    Suppose the maximum $\tau$-width is attained by $J' \in \{P, I_1, \ldots, I_N\}$. Then, there is a combinatorial subpiece $J \subset J'$ of width $W_\tau(J) \succ \sqrt{N} K$ and level between $n-m_0$ and $n$. Finally, pick $N$ (and ultimately $m_*$) to be just high enough such that $J$ is $[2K, \tau]$-wide. This leads to (1).
\end{proof}

Let $\lambda$ be sufficiently large such that $\lambda > \rho$ and $\thres_\lambda \geq m_0+ m_*$. Then, by the lemma, it is sufficient to consider case (2). In this case, there is a connected component $R^*$ of $\rho L^* \cap (\tau L^*)^c$ such that the family $\mathcal{F}(L^*,R^*)$ of curves connecting $L^*$ and $R^*$ has width 
\begin{equation}
    \label{ineq:F*}
    W\left( \mathcal{F}(L^*,R^*) \right) \geq \frac{\varepsilon}{4}K.
\end{equation}

\begin{lemma}
\label{lemma-set-G}
    There exist $t \in \{0,1,2\}$ and a closed set $G \subset \RS$ such that
    \begin{enumerate}[label=\textnormal{(\arabic*)}]
        \item $U^0 := \RS \backslash \big( (\lambda f^t(L^*))^c \cup G\big)$ is a topological disk containing $U^0 \cap \Hq=\lambda f^t(L^*)$;
        \item for all $j \in \N$, every critical value of $f^j$ in $U^0$ lies inside $U^0 \cap \Hq$;
        \item the width of curves in $\RS \backslash (\Hq \cup G)$ connecting $\lambda f^t(L^*)$ and $G$ is $O(1)$.
    \end{enumerate}
\end{lemma}

\begin{proof}
    By Lemma \ref{cuts}, there exist $t \in \{0,1,2\}$ and a pair of closed rays $\gamma_0 \subset \overline{Y^0}$ and $\gamma_\infty \subset \overline{Y^\infty}$ connecting a point in $\left(\lambda f^t(L^*)\right)^c$ to $0$ and $\infty$ respectively such that the family of curves in $\RS \backslash (\Hq \cup \gamma_0 \cup \gamma_\infty)$ connecting $\lambda f^t(L^*)$ and $\gamma_0 \cup \gamma_\infty$ has width at most $10$. In Cases \ref{case:herman-ring} and \ref{case:herman-curve}, the desired closed set is $G:=\gamma_0 \cup \gamma_\infty$. In Case \ref{case:boundary-ring}, we can set $G:= \gamma_0 \cup \gamma_\infty \cup \hat{Y}^0$, where $\hat{Y}^0$ is the connected component of $\RS \backslash \He$ containing $0$. See Lemma \ref{negligibility} and Remark \ref{removal}.
\end{proof}

Set $U^0$ to be the disk in Lemma \ref{lemma-set-G} and set
\[
    L := f^t(L^*), \quad R := f^t(R^*).
\]
For every $j \in \N$, define the corresponding lifts under $f^j$:
\begin{enumerate}[label=\text{(\alph*)}]
    \item\label{item1} $L^j :=$ the connected component of $f^{-j}(I)$ intersecting $\Hq$;
    \item\label{item2} $R^j :=$ the connected component of $f^{-j}(L)$ intersecting $\Hq$;
    \item\label{item3} $\Upsilon^j := L^j \cup R^j$;
    \item\label{item4} $U^j :=$ the connected component of $f^{-j} (U^0)$ containing $\Upsilon^j$.
    \item\label{item5} $\mathcal{F}^j :=$ the canonical horizontal lamination $\mathcal{F}_{\textnormal{can}}^h(U^j , \Upsilon^j)$ of $U^j \backslash \Upsilon^j$.
\end{enumerate}

\begin{lemma}[Width of $\mathcal{F}^0$]
\label{ineq:assumption-on-X0}
    There exists an absolute constant $\varepsilon_1 >0$ such that for sufficiently large $\mathbf{K}$, either
    \begin{enumerate}[label=\textnormal{(\arabic*)}]
        \item $W_\lambda(I) \succ K$, or
        \item $W\left(\mathcal{F}^0\right) \geq \varepsilon_1 K$.
    \end{enumerate}
\end{lemma}

\begin{proof}
    Let $\mathcal{S}$ denote the family of curves connecting $L$ and $R$. Since $\mathcal{S}$ contains $f^t\left(\mathcal{F}(L^*,R^*)\right)$, then by Proposition \ref{transformation law} and inequality (\ref{ineq:F*}),
    \[
    W(\mathcal{S}) \geq \frac{\varepsilon}{4d^2} K.
    \]
    Let $\mathcal{G}$ be the set of curves in $\mathcal{S}$ that lie within $U^0$. If more than half of curves in $\mathcal{S}$ are in $\mathcal{G}$, then
    \[
        W(\mathcal{G}) > \frac{1}{2} W(\mathcal{S}) \geq \frac{\varepsilon}{8 d^2} K.
    \]
    By Proposition \ref{maximality}, this inequality yields (2). Otherwise, at least half of curves in $\mathcal{S}$ intersect either $(\lambda L)^c$ or $G$. Therefore, by Lemma \ref{lemma-set-G} (3),
    \[
        W_\lambda(I) \geq \frac{1}{2} W(\mathcal{S}) - O(1) \geq \frac{\varepsilon}{8 d^2} K - O(1).
    \]
    For sufficiently high $\mathbf{K}$, this yields (1).
\end{proof}

To proceed, it is then sufficient to consider case (2) of the lemma above.

Fix a positive integer $r > n$ that is to be determined later. We would like to show that since every piece is almost invariant under $f^{q_r}$, a definite amount of canonical horizontal leaves of $\mathcal{F}^0$ should restrict to vertical curves in $U^{q_r} \backslash \Upsilon^{q_r}$. To do this, some technical adjustments are required.

Let us define 
\[
\hat{L} := L \cup f^{q_r}(L), \quad \hat{R} := R \cup f^{q_r}(R), \quad \hat{U}^0 := U^0 \backslash f^{q_r}\left((\lambda L)^c\right).
\]
The thickened pieces $\hat{L}$, $\hat{R}$ and the new domain $\hat{U}^0$ are combinatorially very close to $L$, $R$, and $U^0$ respectively. They come with new separation constants $1< \hat{\tau} \ll \hat{\rho} \ll \hat{\lambda}$ such that $\overline{\hat{U}^0 \cap \Hq} = \hat{\lambda} \hat{L}$, $\hat{R}$ is a component of $\overline{\hat{\rho}\hat{L}\backslash \hat{\tau} \hat{L}}$, and $\hat{\tau}\asymp \tau$, $\hat{\rho}\asymp \rho$ and $\hat{\lambda}\asymp \lambda$.

Similar to \ref{item1}--\ref{item5}, we denote for every $j\in \N$ the corresponding lifts:
\begin{enumerate}[label=\text{($\hat{\text{\alph*}}$)}]
    \item\label{item1hat} $\hat{L}^j :=$ the connected component of $f^{-j}\left(\hat{L}\right)$ intersecting $\Hq$;
    \item\label{item2hat} $\hat{R}^j :=$ the connected component of $f^{-j}\left(\hat{R}\right)$ intersecting $\Hq$;
    \item\label{item3hat} $\hat{\Upsilon}^j := \hat{L}^j \cup \hat{R}^j$;
    \item\label{item4hat} $\hat{U}^j :=$ the connected component of $f^{-j}\left(\hat{U}^0\right)$ containing $\hat{\Upsilon}^j$;
    \item\label{item5hat} $\hat{\mathcal{F}}^j := \mathcal{F}^h_{\textnormal{can}}\left(\hat{U}^j , \hat{\Upsilon}^j\right)$.
\end{enumerate}
Note the following relations: 
\begin{align}
    \label{ineq:relation-01} \hat{\Upsilon}^0 \supset \Upsilon^0, \quad \partial \hat{U}^0 &\supset \partial U^0, \quad \hat{U}^0 \subset U^0,\\
    \label{ineq:relation-02} \hat{U}^{q_r} \backslash \hat{\Upsilon}^{q_r} &\subset U^0 \backslash \Upsilon^0.
\end{align}

The relationship between $\hat{\mathcal{F}}^0$ and $\mathcal{F}^0$ is not trivial. Nonetheless, the following lemma states that we can reduce the problem to the case where the widths of $\hat{\mathcal{F}}^0$ and $\mathcal{F}^0$ are comparable.

\begin{lemma}[Comparability between $\hat{\mathcal{F}}^0$ and $\mathcal{F}^0$]
    \label{new-pieces}
    There is an absolute constant $\varepsilon_2>1$ such that for sufficiently large $\lambda \gg \rho$ and $\mathbf{K}$, either
    \begin{enumerate}[label=\textnormal{(\arabic*)}]
        \item there is a level $r$ combinatorial piece $J$ such that either $W_\tau(J) \geq 2K$ or $W_\lambda (J) \succ K$, or
        \item $\frac{1}{2} W\left(\mathcal{F}^0\right) \leq W\left(\hat{\mathcal{F}}^0\right) \leq \varepsilon_2 W\left(\mathcal{F}^0\right)$.
    \end{enumerate}
\end{lemma}

In the proof, we will show that either such a piece $J$ in (1) can be found from the symmetric difference between $\partial U^0 \cup \Upsilon^0$ and $\partial \hat{U}^0 \cup \hat{\Upsilon}^0$, or (2) holds.

\begin{proof}
Suppose (1) does not hold. Let us present the canonical horizontal lamination $\hat{\mathcal{F}}_0$ as the union $\mathcal{S}_1 \cup \mathcal{S}_2 \cup \mathcal{S}_3$ where: 
\begin{itemize}
    \item[$\rhd$] $\mathcal{S}_1 =$ set of leaves in $\hat{\mathcal{F}}_0$ that has an endpoint on $P_1 := \overline{\hat{L} \backslash L}$;
    \item[$\rhd$] $\mathcal{S}_2 =$ set of leaves in $\hat{\mathcal{F}}_0$ that has an endpoint on $P_2 := \overline{\hat{R}\backslash R}$;
    \item[$\rhd$] $\mathcal{S}_0 = \hat{\mathcal{F}}_0 \backslash ( \mathcal{S}_1 \cup \mathcal{S}_2)$.
\end{itemize}
Note that $P_1$ and $P_2$ are combinatorial pieces of level $r$. See Figure \ref{fig:adjustment}.

For each $i \in \{1,2\}$, $\mathcal{S}_i$ restricts to a sublamination of $\mathcal{F}_\tau(P_i)$ because the combinatorial distance between $\hat{L}$ and $\hat{R}$ is greater than $\frac{\tau -1}{2}l_r$. We can assume that $W(\mathcal{S}_i) < 2  \varepsilon_1^{-1} W\left(\mathcal{F}^0\right)$ because otherwise, by Lemma \ref{ineq:assumption-on-X0} (2), we would have
\[
W_{\tau}(P_i) \geq W(\mathcal{S}_i) \geq 2  \varepsilon_1^{-1} W\left(\mathcal{F}^0\right) \geq 2K
\]
and this would yield (1) instead. Meanwhile, since every leaf of $\mathcal{S}_0$ is a horizontal curve in $U^0 \backslash \Upsilon^0$, then by Proposition \ref{maximality},
\begin{align*}
    W(\mathcal{S}_0) &\leq W\left(\mathcal{F}^0\right) + O(1).
\end{align*}
Therefore,
\[
    W\left(\hat{\mathcal{F}}^0\right) \leq \sum_{i=0}^2 W(\mathcal{S}_i) < (4  \varepsilon_1^{-1} + 1) W\left(\mathcal{F}^0\right) + O(1).
\]
For sufficiently large $\mathbf{K}$, the inequality above yields the upper bound in (2).

Next, to obtain the lower bound in (2), we will consider the set $\mathcal{S}_3$ of leaves of $\mathcal{F}^0$ that intersect the level $r$ combinatorial interval $P_3: =\overline{\partial \hat{U}^0 \backslash \partial U^0}$. By (\ref{ineq:relation-01}), every horizontal leaf of $\mathcal{F}^0$ either intersects $P_3$ or restricts to a horizontal curve in $\hat{U}^0 \backslash \hat{\Upsilon}^0$. As such,
\begin{equation}
    \label{ineq:s3}
    W\left(\mathcal{F}^0\right) \leq W\left(\hat{\mathcal{F}}^0\right) + W(\mathcal{S}_3) + O(1)
\end{equation}
Observe that $\mathcal{S}_3$ overflows $\mathcal{F}_\lambda(P_3)$. We can assume that $W(\mathcal{S}_3) < \frac{1}{3} W\left(\mathcal{F}^0\right)$ because otherwise 
\[
W_\lambda(P_3) \geq \frac{1}{3} W\left(\mathcal{F}^0\right) \succ K.
\]
which would again yield (1). By applying this assumption to (\ref{ineq:s3}),
\begin{align*}
    \frac{2}{3} W\left(\mathcal{F}^0\right) \leq W\left(\hat{\mathcal{F}}^0\right) + O(1).
\end{align*}
Hence, for sufficiently high $\mathbf{K}$, we immediately obtain the lower bound in (2).
\end{proof}

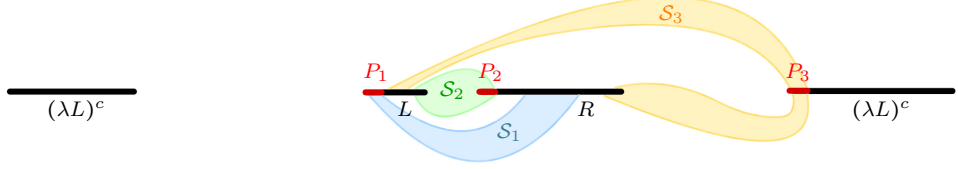
\begin{figure}
    \centering
    \begin{tikzpicture}
    \node[anchor=south west,inner sep=0] (image) at (0,0) {\includegraphics[width=\linewidth]{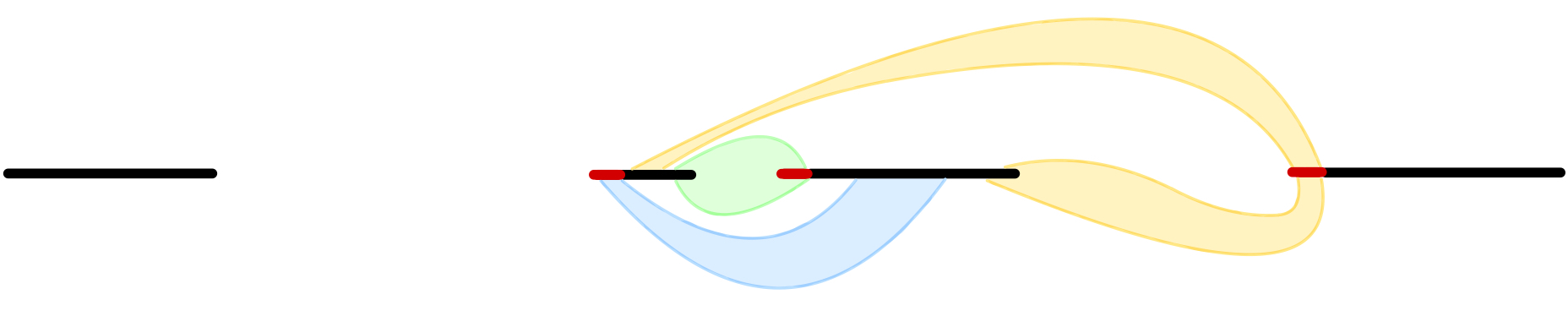}};
    \begin{scope}[
        x={(image.south east)},
        y={(image.north west)}
    ]
        \node [orange, font=\bfseries] at (0.70,0.86) {\Small $\mathcal{S}_3$};
        \node [green!60!black, font=\bfseries] at (0.47,0.48) {\Small $\mathcal{S}_2$};
        \node [cyan!50!black, font=\bfseries] at (0.53,0.265) {\Small $\mathcal{S}_1$};
        \node [black, font=\bfseries] at (0.42,0.385) {\Small $L$};
        \node [black, font=\bfseries] at (0.61,0.385) {\Small $R$};
        \node [red, font=\bfseries] at (0.39,0.57) {\Small $P_1$};
        \node [red, font=\bfseries] at (0.51,0.57) {\Small $P_2$};
        \node [red, font=\bfseries] at (0.833,0.57) {\Small $P_3$};
        \node [black, font=\bfseries] at (0.075,0.375) {\Small $(\lambda L)^c$};
        \node [black, font=\bfseries] at (0.92,0.375) {\Small $(\lambda L)^c$};
    \end{scope}
\end{tikzpicture}
    
    \caption{$\mathcal{S}_1$, $\mathcal{S}_2$, and $\mathcal{S}_3$.}
    \label{fig:adjustment}
\end{figure}

It is sufficient to proceed under the assumption that Lemma \ref{new-pieces} (2) holds. In this case, $W(\hat{\mathcal{F}}^0) \geq \frac{\varepsilon_1}{2} K$.

Before finishing the proof of the theorem, the following important lemma is needed. It states that either $\tau$-degeneration doubles or there is significant loss in horizontal width after a certain number of pullbacks.

\begin{lemma}[Loss of horizontal width]
    \label{positive entropy}
    There exist absolute constants $m_1, m_2 \in \mathbb{N}$ such that for sufficiently large $\lambda \gg \rho$ and $\mathbf{K}$, either 
    \begin{enumerate}[label=\textnormal{(\arabic*)}]
        \item there is a $[2K, \tau]$-wide combinatorial piece of level $n+m_1$, or
        \item $W\left(\hat{\mathcal{F}}^{q_{n+m_2}}\right) < (2\varepsilon_2)^{-1} W\left(\hat{\mathcal{F}}^0\right)$.
    \end{enumerate} 
\end{lemma}

This lemma is a replacement for Kahn's entropy argument in \cite[\S6.3]{K06}, and it will directly follow from Proposition \ref{pos-entropy} in the next section. See Remark \ref{rem:final-proof}. Let us finally set $r=n+m_2$ and assume that Lemma \ref{positive entropy} (2) holds. We are now in position to adapt the push-forward argument.

From the embedding in (\ref{ineq:relation-02}), horizontal leaves in $U^0 \backslash \Upsilon^0$ restrict to either horizontal or vertical curves in $\hat{U}^{q_r} \backslash \hat{\Upsilon}^{q_r}$. In other words, 
\[
    W\left(\mathcal{F}^0\right) \leq W\left((\mathcal{F}_{\textnormal{can}}^v(\hat{U}^{q_r} , \hat{\Upsilon}^{q_r})\right)+ W\left(\hat{\mathcal{F}}^{q_r}\right) + O(1).
\]
By Lemma \ref{new-pieces} (2) and Lemma \ref{positive entropy} (2),
\begin{align*}
    W\left(\mathcal{F}^0\right) &\leq W\left((\mathcal{F}_{\textnormal{can}}^v(\hat{U}^{q_r} , \hat{\Upsilon}^{q_r})\right) + \frac{1}{2} W\left(\mathcal{F}^0\right) + O(1).
\end{align*}
For sufficiently high $\mathbf{K}$, the inequality simplifies to
\[
W\left((\mathcal{F}_{\textnormal{can}}^v(\hat{U}^{q_r} , \hat{\Upsilon}^{q_r})\right) \geq \frac{1}{3} W\left(\mathcal{F}^0\right).
\]
By Lemma \ref{ineq:assumption-on-X0} (2), this implies that $W \left(\hat{U}^{q_r}, \hat{J}^{q_r}\right) \succ K$ for some $\hat{J} \in \left\{ \hat{L}, \hat{R}\right\}$.

Consider the iterate
\[
f^{q_r}: \left(\hat{U}^{q_r}, \hat{U}^{q_r}_\tau , \hat{J}^{q_r}\right) \to \left(\hat{U}^0, \hat{U}^0_\tau, \hat{J}\right)
\]
where $\hat{U}^0_\tau := \hat{U}^0 \backslash \left(\tau \hat{J} \right)^c$ and $\hat{U}^{q_r}_\tau$ is the pullback of $\hat{U}^0_\tau$ under $f^{q_r}$ containing $\hat{J}^{q_r}$. By Lemma \ref{counting-critical}, the degree of $f^{q_r}$ on $\hat{U}^{q_r}_\tau$ remains independent of $\lambda$. By Lemma \ref{covering-lemma}, for sufficiently high $\mathbf{K}$, either 
\[
W \left(\hat{U}^0, \hat{J} \right) \succ K, \quad \text{or} \quad W \left(\hat{U}^0_\tau,\hat{J} \right) \geq (2 C+1) K,
\]
where $C$ is the constant from Proposition \ref{bounded-type}. In either case, Lemma \ref{lemma-set-G} asserts that the width of curves that land on $G$ is negligible. Hence, for sufficiently large $\mathbf{K}$, there exists a combinatorial subpiece $J \subset \hat{J}$ such that either
\[
    W_\lambda(J) \succ K, \quad \text{or} \quad W_\tau(J) \geq 2K.
\]

At last, pick the threshold increment $\step \in \N$ such that $\step\geq \max\{m_0,m_1,m_2\}$ and that $n- \step$ is less than the level of $\hat{R}$. This concludes the proof of Theorem \ref{promotion}.
\end{proof}


\section{Loss of horizontal width}
\label{sec:loss-of-horizontal-width}
Let us fix $\lambda \gg \tau$ and let $n \in \mathbb{N}$ be such that $2\lambda l_n < 1$. The key players of this section are as follows. 
\begin{itemize}
    \item[$\rhd$] $L$ and $R$ are pieces in $\Hq$ of combinatorial distance $\dist(I,L) \asymp l_n$ and length at least $l_n$ satisfying $|L| \asymp |R| \asymp l_n$;
    \item[$\rhd$] $U^0 \subset \C^*$ is a topological disk containing $L \cup R$ such that $\Hq \backslash U^0 = (\lambda L)^c$.
\end{itemize}

\begin{remark}
\label{adjustment-in-case-c}
In Case \ref{case:boundary-ring} (see \S\ref{ss:setup}), we will also impose the additional assumption that $U^0$ is disjoint from the connected component of $\RS \backslash \He$ containing $0$, so that for all $j$, every critical value of $f^j$ in $U^0$ must lie in $U^0 \cap \Hq$.
\end{remark}

Similar to \ref{item1}--\ref{item5} in \S6.2, we define the corresponding lifts under $f^j$ for $j \in \N$:
\begin{enumerate}[label=\text{(\alph*.)}]
    \item\label{item01} $L^j :=$ the connected component of $f^{-j}(L)$ intersecting $\Hq$;
    \item\label{item02} $R^j :=$ the connected component of $f^{-j}(R)$ intersecting $\Hq$;
    \item\label{item03} $\Upsilon^j := L^j \cup R^j$;
    \item\label{item04} $U^j :=$ the connected component of $f^{-j} (U^0)$ containing $\Upsilon^j$.
    \item\label{item05} $\mathcal{F}^j :=$ the canonical horizontal lamination $\mathcal{F}_{\textnormal{can}}^h(U^j , \Upsilon^j)$ of $U^j \backslash \Upsilon^j$.
\end{enumerate}
We are restricting our map $f$ to a sequence of branched coverings 
\[
\ldots \xrightarrow[]{\enspace f \enspace} \left(U^3, \Upsilon^3\right) \xrightarrow[]{\enspace f \enspace} \left(U^2, \Upsilon^2\right) \xrightarrow[]{\enspace f \enspace} \left(U^1,\Upsilon^1\right) \xrightarrow[]{\enspace f \enspace} \left(U^0, \Upsilon^0\right)
\]
and we will study how the width of $\mathcal{F}^j$ behaves as $j$ increases. The goal of this section is to prove the following proposition.

\begin{proposition}
\label{pos-entropy}
    For any $\Delta>1$, $\delta\in(0,1)$, and sufficiently large $\lambda$,  there are some constants $m_1 = m_1(\Delta, \delta) \in \mathbb{N}$, $m_2= m_2(\Delta, \delta) \in \mathbb{N}$ and $\mathbf{K} = \mathbf{K}(\lambda, \Delta, \delta) > 1$ such that if $W(\mathcal{F}^0) = K \geq \mathbf{K}$, then either
\begin{enumerate}[label=\textnormal{(\arabic*)}]
        \item\label{goal-number-1} there is a level $n+m_1$ combinatorial piece $J$ of width $W_\tau(J) \geq \Delta K$, or
        \item\label{goal-number-2} there is significant loss in horizontal width: $W(\mathcal{F}^{q_{n+m_2}}) < \delta K$.
\end{enumerate}
\end{proposition}

\begin{remark}
\label{rem:final-proof}
    Recall that the final missing ingredient of the proof of Theorem \ref{promotion} is Lemma \ref{positive entropy}. We can apply Proposition \ref{pos-entropy} in the context of Lemma \ref{positive entropy} (e.g. setting $\Delta = 4\varepsilon_1^{-1}$ and $\delta = (2 \varepsilon_2)^{-1}$), thereby proving the lemma immediately.
\end{remark}

\subsubsection{Outline}
\label{sss:rough-outline}

For every $j \in \N$, since $U^j\backslash \Upsilon^j$ is a disk with two connected compact sets removed, the leaves of $\mathcal{F}^j$ belong to at most two proper homotopy classes in $U^j \backslash \Upsilon^j$. We say that a homotopically non-trivial proper curve in $U^j \backslash \Upsilon^j$ is of type
\begin{enumerate}[label=\text{\Alph*}]
    \item\label{type-A} if it connects $L^j$ and $R^j$, and
    \item\label{type-B} if it starts and ends at the same component of $\Upsilon^j$.
\end{enumerate}
Naturally, we split $\mathcal{F}^j$ into a disjoint union of type \ref{type-A} and \ref{type-B} sublaminations $\mathcal{A}^j \cup \mathcal{B}^j$. See Figure \ref{fig:wad00}.

To illustrate the main idea, let us consider the unbranched covering map 
\[
f^r : U^{r} \backslash f^{-r}(\text{CV}) \to U^0 \backslash \text{CV}
\]
of large degree (depending on $\lambda$, $\delta$, and $\Delta$) where $r=r(\delta,\Delta)\geq 1$ is a fixed integer and $\text{CV}$ is the set of critical values of $f^r$. Consider the horizontal thick-thin decomposition $\tilde{\mathcal{T}}$ of $U^0 \backslash ( \Upsilon^0 \cup \text{CV})$. (Refer to Appendix \ref{ss:canonical-lamination}.) It has the key property that, up to an additive constant, curves in $\mathcal{F}^r$ must travel through the canonical rectangles in $\tilde{\mathcal{T}}'=(f^r)^*\tilde{\mathcal{T}}$.

We present $\tilde{\mathcal{T}}$ as $\mathcal{A} \cup \mathcal{B} \cup \mathcal{P}$ where $\mathcal{A}$ consists of type \ref{type-A} rectangles, $\mathcal{B}$ consists of type \ref{type-B} rectangles, and $\mathcal{P}$ consists of \emph{peripheral} rectangles, i.e. those which become trivial when the punctures at the critical values are forgotten. See Figure \ref{fig:wad01}. The widths of $\mathcal{A}$ and $ \mathcal{B}$ are essentially the widths of $\mathcal{A}^0$ and $\mathcal{B}^0$ respectively, whereas the width of $\mathcal{P}$ may depend on the degree of $f^r$. To tame $\mathcal{P}$, we further split it into $\mathcal{D} \cup \mathcal{E}$ where the canonical leaves of $\mathcal{E}$ are disjoint from the complement of either $\eta L^0$ or $\eta R^0$ for some small $\eta >1$. We eliminate $\mathcal{E}$ by absorbing it into $\Upsilon^0 = L^0 \cup R^0$, giving us a thickening $\mathbf{\Upsilon} = \mathbf{L} \cup \mathbf{R}$. The width of the remainder $\mathcal{D}$ can be assumed to be bounded in terms of $K$ independent of the degree of $f^r$ (Lemma \ref{lem:DE-degeneration}). 

There is a unique (up to homotopy) proper curve $\boldsymbol{b}$ in $U^0 \backslash \mathbf{\Upsilon}$ separating $\mathbf{L}$ and $\mathbf{R}$. The curve $\boldsymbol{b}$ can be assumed to be disjoint from $\mathcal{D}$, is crossed by $\mathcal{A}$ exactly once and by $\mathcal{B}$ twice. This observation implies that the \emph{asymmetric width}
\[
    Z^j := W\left(\mathcal{A}^j\right) + 2 W\left(\mathcal{B}^j\right)
\]
is non-increasing (Corollary \ref{cor:monotonicity}). To upgrade this monotonicity to a strict loss, we construct ``leftmost`` and ``rightmost`` separating proper curves $\beta_L$ and $\beta_R$ in $U^{r} \backslash f^{-r}(\mathbf{\Upsilon} \cup \text{CV})$ such that both $f^r(\beta_L)$ and $f^r(\beta_R)$ are homotopic to $\boldsymbol{b}$ in $U^0 \backslash \mathbf{\Upsilon}$. Unlike $\boldsymbol{b}$, the images of $\beta_L$ and $\beta_R$ will cross $\mathcal{D}$. See Figure \ref{fig:wad02} for an illustration. 

We say that a rectangle in $\tilde{\mathcal{T}}'$ is \emph{persistent} if it crosses both $\beta_L$ and $\beta_R$. Ultimately, the analysis is split into two cases.

\begin{enumerate}[leftmargin=0.6in] 
    \item[Case 1:] Persistent rectangles are wide ($\geq 0.1K$). \\
    All persistent canonical leaves are homotopic rel critical points of $f^r$ (Lemma \ref{homotopy}) and they submerge in and out through $\Hq$ (Figures \ref{fig:fences-01} and \ref{fig:fences-02}) many times. This leads to amplification as stated in Proposition \ref{pos-entropy} \ref{goal-number-1}.
    \item[Case 2:] Persistent rectangles are not too wide ($< 0.1K$). \\
    Non-persistent rectangles are subject to the pull-off principle in a similar vein as \cite[\S6.3]{K06}. By taking into account the bound on the width of $\mathcal{D}$, we show that the asymmetric width $Z^j$ is strictly decreasing (the second case of Lemma \ref{lem:deg-or-loss-of-width}). Then, we run an inductive process to get the desired shrinking factor $\delta$ in Proposition \ref{pos-entropy} \ref{goal-number-2}. 
\end{enumerate}

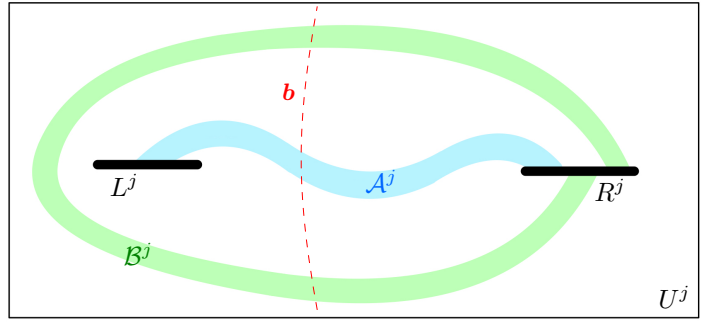
\begin{figure}
    \centering
    \begin{tikzpicture}
    \node[anchor=south west,inner sep=0] (image) at (0,0) {\includegraphics[width=0.75\linewidth]{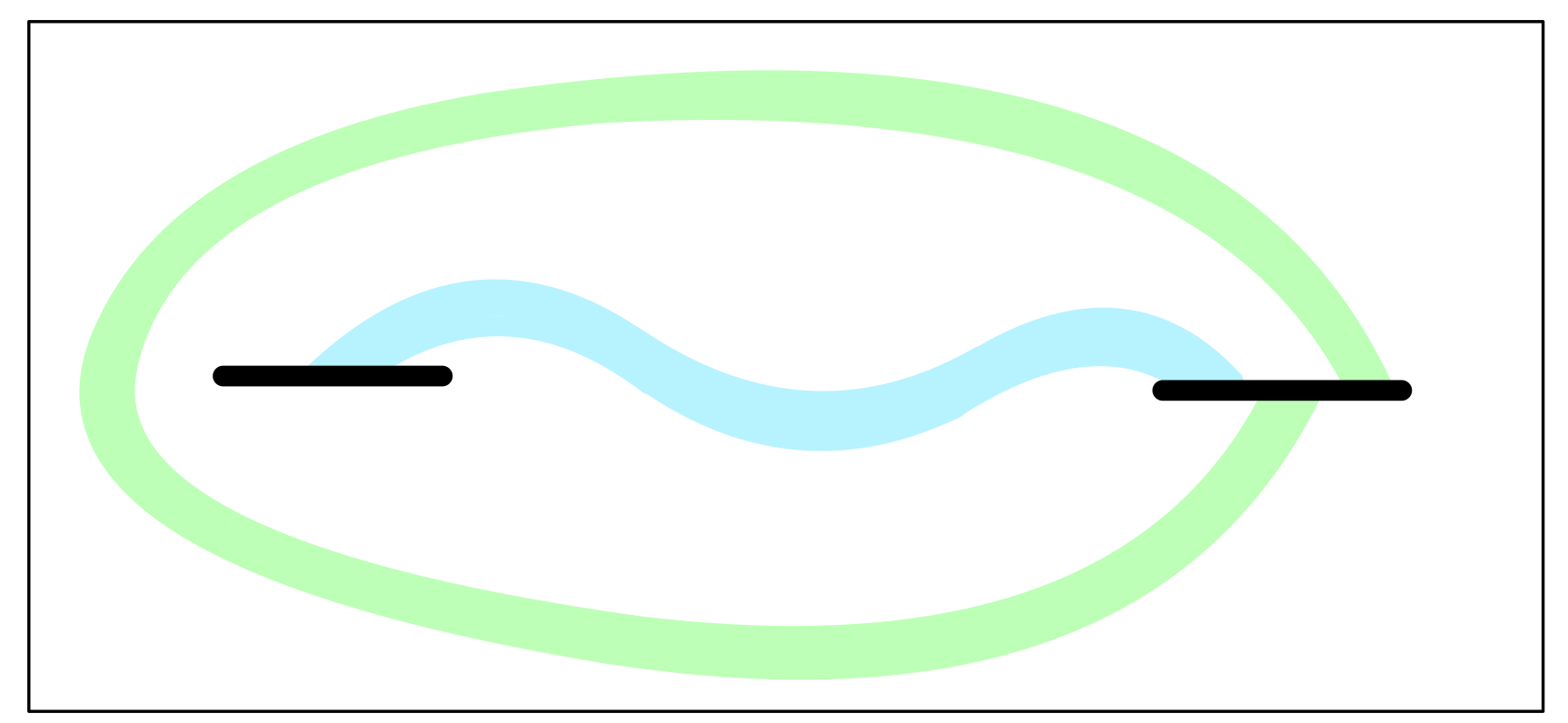}};
    \begin{scope}[
        x={(image.south east)},
        y={(image.north west)}
    ]
        \node [black, font=\bfseries] at (0.18,0.42) {$L^j$};
        \node [black, font=\bfseries] at (0.86,0.405) {$R^j$};
        \node [black, font=\bfseries] at (0.95,0.08) {$U^j$};
        
        \node [green!50!black, font=\bfseries] at (0.20,0.21) {$\mathcal{B}^j$};
        \node [blue!60!cyan, font=\bfseries] at (0.54,0.43) {$\mathcal{A}^j$};
        \draw[red, line width=0.1pt, dashed] (0.45,0.04) .. controls (0.42,0.4) and (0.42,0.6) .. (0.45,0.96);
        \node [red, font=\bfseries] at (0.41,0.70) {\small $\boldsymbol{b}$};
    \end{scope}
\end{tikzpicture}
    
    \caption{The decomposition of $\mathcal{F}^j$. The intersection number between $\boldsymbol{b}$ and $\mathcal{A}^j$ is one, and that between $\boldsymbol{b}$ and $\mathcal{B}^j$ is two.}
    \label{fig:wad00}
\end{figure}

\subsection{Decomposition of canonical laminations}
\label{ss:decomposition}

Let us fix 
\[
K := W\left(\mathcal{F}^0\right) \quad \text{and} \quad r := q_{n+m}
\]
for some sufficiently large integer $m \in \N$. From now on until the end of \S\ref{ss:persistence}, we consider the unbranched covering map 
\[
f^r : U^{r} \backslash f^{-r}(\text{CV}) \to U^0 \backslash \text{CV}
\]
where $\text{CV}=\text{CV}(f^r)$ denotes the set of critical values of $f^r$ in $U^0$.

\begin{lemma}
    \label{lem:counting-cv}
    There is an absolute constant $C>1$ such that $|\textnormal{CV}| = O(\lambda C^m)$. In particular, the degree of $f^r$ depends only on $m$ and $\lambda$.
\end{lemma}

\begin{proof}
    Consider the piece $J=\overline{U^0} \cap \Hq$. It suffices to fix a critical value $v \in \Hq$ of $f$ and estimate the size of $\mathcal{O}_v := \{ (f|_{\Hq})^{-i}(v) \in J \: : \: 0 \leq i \leq q_{n+m}-1\}$. By Proposition \ref{renormalization-tiling}, $\mathcal{O}_v$ divides $J$ into $|\mathcal{O}_v|+1$ pieces $J_1,\ldots,J_{|\mathcal{O}_v|+1}$ of length at least $l_{n+m}$. Let $C>1$ be the constant from Proposition \ref{bounded-type}. Then, 
    \[
    \lambda l_n \asymp |J| = \sum_i |J_i| \geq |\mathcal{O}_v| \cdot l_{n+m} \geq |\mathcal{O}_v| \cdot C^{-m} l_n
    \]
    which implies that $|\mathcal{O}_v| =O(\lambda C^m)$.
\end{proof}

Let $\tilde{\mathcal{T}}:= \textnormal{TTD}^h(U^0,\Upsilon^0 \cup \text{CV})$ denote the horizontal thick-thin decomposition of $U^0 \backslash (\Upsilon^0 \cup \text{CV})$. (Refer to Appendix \ref{ss:canonical-lamination} for details.) According to Proposition \ref{naturality}, the horizontal thick-thin decomposition $\tilde{\mathcal{T}}'$ of $U^r \backslash f^{-r}(\Upsilon^0 \cup \text{CV})$ is the full lift $(f^r)^*\tilde{\mathcal{T}}$ of $\tilde{\mathcal{T}}$. 

By taking into account the critical values of $f^r$, we enrich the canonical lamination with the presence of peripheral arcs. We say that a proper curve in $U^0\backslash \left(\Upsilon^0 \cup \text{CV}\right)$ is \emph{peripheral} if it has a trivial proper homotopy class in $U^0 \backslash \Upsilon^0$. We will decompose $\tilde{\mathcal{T}}$ into a disjoint union
\[
    \tilde{\mathcal{T}} := \mathcal{A} \cup \mathcal{B} \cup \mathcal{P} 
\]
where $\mathcal{A}$ consists of canonical rectangles of type \ref{type-A} (leaves of $\mathcal{F}(\mathcal{A})$ are of type \ref{type-A}) in $U^0\backslash \Upsilon^0$, $\mathcal{B}$ consists of canonical rectangles of type \ref{type-B}, and $\mathcal{P}$ consists of peripheral rectangles. Denote by $A$ and $B$ the total widths of $\mathcal{A}$ and $\mathcal{B}$ respectively.

Observe that the width of $\mathcal{F}^0$ should be close to $A+B$. The next lemma follows directly from Proposition \ref{maximality} and Lemma \ref{lem:counting-cv}.

\begin{lemma}
\label{lem:correction}
    There is some $C = C(m,\lambda)>0$ such that
    \[
        \left| A - W\left(\mathcal{A}^0\right)  \right|  \leq C \quad \text{and} \quad \left| B- W\left(\mathcal{B}^0\right) \right|  \leq C
    \]
\end{lemma}

Let us pick a definite constant $\eta>1$ such that the combinatorial distance between $\eta I$ and $\eta L$ is still $\asymp l_n$. Let us consider collections of disjoint rectangles 
\[
    \mathcal{D} = \mathcal{D}_L \cup \mathcal{D}_R \quad \text{ and } \quad \mathcal{E} = \mathcal{E}_L \cup \mathcal{E}_R
\]
defined as follows. Consider any peripheral rectangle $\mathcal{R}$ in $\mathcal{P}$ attached to $L^0$.
Let us split $\mathcal{R}$ into two conformal subrectangles $\mathcal{R}_{\mathcal{D}}$ and $\mathcal{R}_{\mathcal{E}}$ where every leaf of $\mathcal{F}(\mathcal{R}_{\mathcal{D}})$ intersects $(\eta L^0)^c$ and every leaf of $\mathcal{F}(\mathcal{R}_{\mathcal{E}})$ is disjoint from $(\eta L^0)^c$. 
The rectangle $\mathcal{R}_{\mathcal{D}}$ is then added to $\mathcal{D}_L$. If $W(\mathcal{R}_{\mathcal{E}}) > 1$, let us remove the outer $1$-buffer of $\mathcal{R}_{\mathcal{E}}$ and add the remaining subrectangle into $\mathcal{E}_L$. In a similar manner, we define the collections $\mathcal{D}_R$ and $\mathcal{E}_R$ out of splitting peripheral rectangles attached to $R^0$.

Peripheral rectangles in $\mathcal{D}$ are a source of $\tau$-degeneration. Let us denote the widths of $\mathcal{D}$, $\mathcal{D}_L$, and $\mathcal{D}_R$ by $D$, $D_L$, and $D_R$ respectively.

\begin{lemma}
\label{lem:DE-degeneration}
    There is an absolute constant $m_0 \in \N$ such that if
    \[
        m \geq m_0 \quad \text{ and } \quad D \geq \kappa K
    \]
    for some $\kappa >1$, then there is a combinatorial piece $J$ of level $n+m$ and width $W_\tau(J) \geq \varepsilon \kappa K$ for some constant $\varepsilon=\varepsilon(m)>0$.
\end{lemma}

\begin{proof}
    Suppose without loss of generality that $D_L \geq \frac{\kappa}{2} K$. There exists a constant $m_0 \in \N$ depending on $\eta$ such that for any level $n+m_0$ combinatorial subpiece $J$ of $L^0$, the thickened piece $\tau J$ is contained in $\eta L^0$. Let us assume $m \geq m_0$. The piece $L^0$ can be covered by $N$ level $n+m$ combinatorial pieces $J_1,\ldots, J_N$ for some integer $N=N(m) \in \N$. Let us split $\mathcal{F}(\mathcal{D}_L)$ into sublaminations $\mathcal{L}_1,\ldots, \mathcal{L}_N$ where leaves of $\mathcal{L}_i$ start at points on $J_i$. For each $I$, since $\mathcal{D}_L$ crosses $(\eta L^0)^c$ which is contained in $(\tau J_i)^c$, then $\mathcal{L}_i$ overflows the curve family $\mathcal{F}_\tau(J_i)$. Therefore,
    \[
        \frac{\kappa}{2} K \leq D_L = \sum_{i=1}^N W(\mathcal{L}_i) \leq \sum_{i=1}^N W_\tau(J_i) \leq N \max_i W_\tau(J_i).
    \]
    Consequently, there is some $i \in \{1,\ldots,N\}$ such that $W_\tau(J_i)\geq \frac{\kappa}{2N} K$. 
\end{proof}

In contrast, peripheral rectangles in $\mathcal{E}$ are combinatorially close to $\Upsilon^0$. We will remove $\mathcal{E}$ from consideration by absorbing it into $\Upsilon^0$ as follows. Let us define $\mathbf{L}$ to be the hull of the union of $L^0$ and $\mathcal{E}_L$, i.e. the smallest compact full subset of $U^0$ containing $L^0 \cup \mathcal{E}_L$. Similarly, we define $\mathbf{R}$ to be the hull of $R^0 \cup \mathcal{E}_R$. Denote by $\mathbf{L}'$ and $\mathbf{R}'$ the connected components of $f^{-r}\left(\mathbf{L}\right)$ and $f^{-r}\left(\mathbf{R}\right)$ that contain $L^0$ and $R^0$
respectively. Let 
\[
    \mathbf{\Upsilon} := \mathbf{L} \cup \mathbf{R} \quad \text{ and } \quad \mathbf{\Upsilon}' := \mathbf{L}' \cup \mathbf{R}'.
\]
By construction, we have the following properties.
\begin{itemize}
    \item[$\rhd$] The intersection of $\mathbf{\Upsilon}$ and $\mathbf{H}$ is within $\eta L^0 \cup \eta R^0$.
    \item[$\rhd$] The thick-thin decompositions of $U^0 \backslash (\mathbf{\Upsilon} \cup \text{CV})$ and $U^r \backslash f^{-r}(\mathbf{\Upsilon} \cup \text{CV})$ are essentially 
    \[
        \mathcal{T} := \mathcal{A} \cup \mathcal{B} \cup \mathcal{D} \quad \text{ and } \quad
        \mathcal{T}' := (f^r)^* \mathcal{T}
    \]
    respectively. By Lemma \ref{lem:counting-cv}, the difference between $\mathcal{F}(\mathcal{T})$ and $\mathcal{F}\left(\tilde{\mathcal{T}}\right)$ has width bounded by some constant depending on $m$ and $\lambda$.
\end{itemize}
See Figure \ref{fig:wad01} for an illustration. 

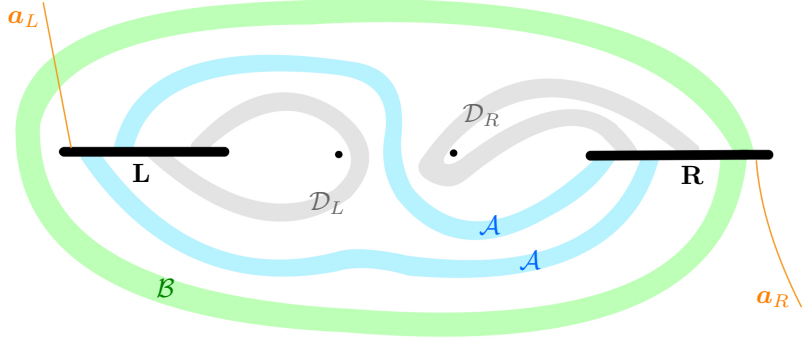
\begin{figure}
    \centering
    \begin{tikzpicture}
    \node[anchor=south west,inner sep=0] (image) at (0,0) {\includegraphics[width=0.85\linewidth]{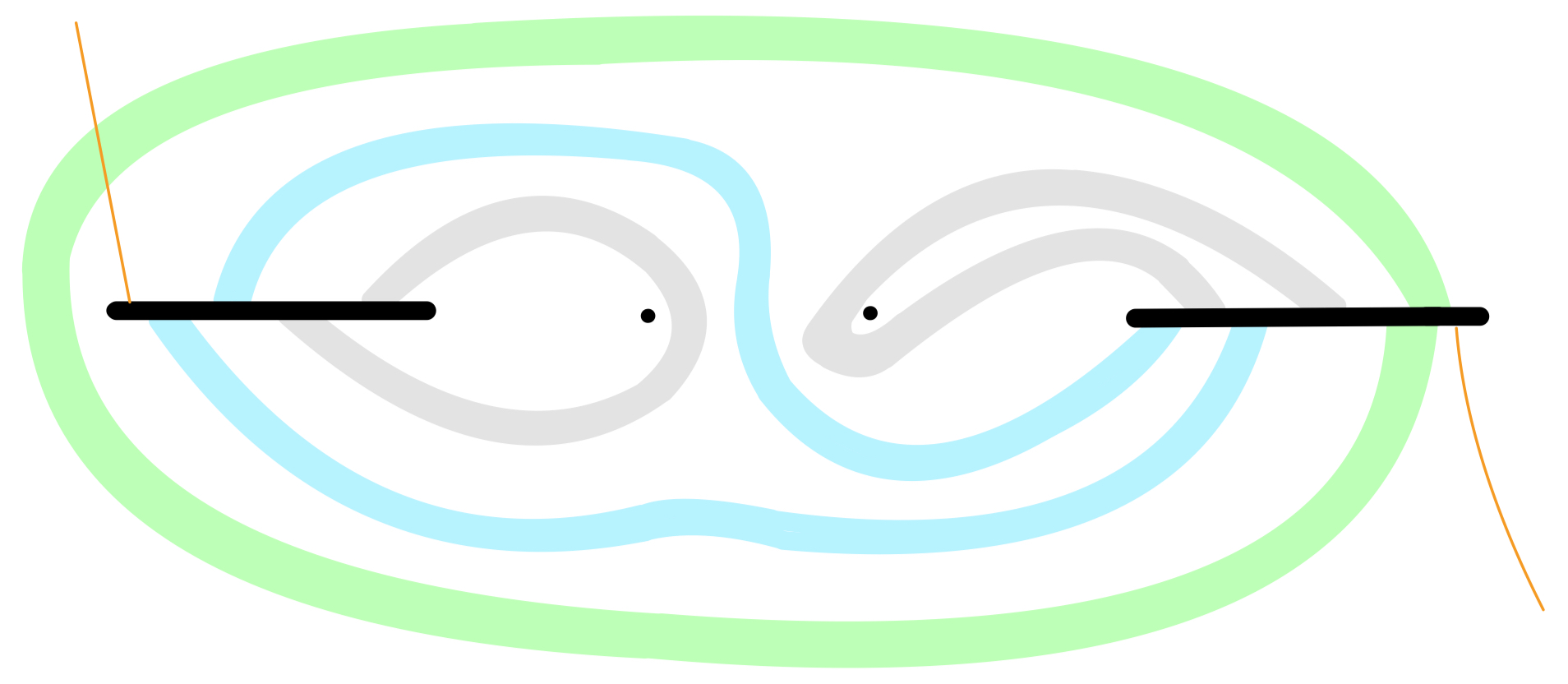}};
    \begin{scope}[
        x={(image.south east)},
        y={(image.north west)}
    ]
        \node [black, font=\bfseries] at (0.17,0.485) {$\mathbf{L}$};
        \node [black, font=\bfseries] at (0.85,0.48) {$\mathbf{R}$};
        \node [orange, font=\bfseries] at (0.025,0.92) {$\boldsymbol{a}_L$};
        \node [orange, font=\bfseries] at (0.95,0.13) {$\boldsymbol{a}_R$};
        
        \node [green!50!black, font=\bfseries] at (0.20,0.16) {$\mathcal{B}$};
        \node [blue!60!cyan, font=\bfseries] at (0.60,0.34) {$\mathcal{A}$};
        \node [blue!60!cyan, font=\bfseries] at (0.65,0.24) {$\mathcal{A}$};
        \node [gray!80!black, font=\bfseries] at (0.40,0.40) {$\mathcal{D}_L$};
        \node [gray!80!black, font=\bfseries] at (0.59,0.65) {$\mathcal{D}_R$};
    \end{scope}
\end{tikzpicture}
    
    \caption{An example of the decomposition of $\mathcal{T}$. Note that $\mathcal{E}$ is inside of $\mathbf{\Upsilon} = \mathbf{L} \cup \mathbf{R}$.}
    \label{fig:wad01}
\end{figure}

As we apply Lemma \ref{domination} to the inclusion $U^r \backslash f^{-r}(\Upsilon^0 \cup \text{CV}) \subset U^r \backslash \Upsilon^r$, we obtain the following fundamental property relating the widths of $\mathcal{F}^r$ and $\mathcal{T}'$.

\begin{lemma}
    \label{domination-property}
    There exist some sublamination $\mathcal{F}^r_{\textnormal{sub}} \subset \mathcal{F}^r$ and some constant $C=C(m,\lambda)>0$ such that 
    \[
        W\left(\mathcal{F}^r\right) - C \leq W(\mathcal{F}^r_{\textnormal{sub}}),
    \]
    and for every leaf $\gamma$ of $\mathcal{F}^r_{\textnormal{sub}}$, every component of $\gamma \backslash f^{-r}(\Upsilon^0 )$ is either
    \begin{enumerate}[label=\textnormal{(\arabic*)}]
        \item a homotopically trivial proper curve in $U^r \backslash f^{-r}(\Upsilon^0 \cup \textnormal{CV})$, or
        \item a proper curve in a rectangle in $\mathcal{T}'$.
    \end{enumerate}
\end{lemma}

\subsection{Separating curves}
\label{sss:strips}

From now on, let us assume without loss of generality that $\mathcal{B}$ starts from and ends at $R^0$. Given a proper curve $\alpha$ in $U^r\backslash\Upsilon^r$, we will denote by $\mathcal{T}'[\alpha]$ the union of rectangles in $\mathcal{T}'$ that intersect $\alpha$. 

Let us fix vertical rays $\boldsymbol{a}_L$ and $\boldsymbol{a}_R$ in $U^0 \backslash \left(\Upsilon^0 \cup \text{CV}\right)$ where
\begin{itemize}
    \item[$\rhd$] $\boldsymbol{a}_L$ connects $\partial U^0$ to $L^0$ and $\boldsymbol{a}_R$ connects $\partial U^0$ to $R^0$;
    \item[$\rhd$] $\boldsymbol{a}_L$ is crossed by $\mathcal{B}$ exactly once and is disjoint from $\mathcal{T} \backslash \mathcal{B}$;
    \item[$\rhd$] $\boldsymbol{a}_R$ is disjoint from $\mathcal{T} $.
\end{itemize}
The first assumption states that $\boldsymbol{a}_L$ and $\boldsymbol{a}_R$ are vertical cuts of $U^0\backslash\left(\Upsilon^0\cup \text{CV}\right)$, whereas the other two assumptions state that the minimal intersection number relative to $\mathcal{T}$ is achieved. See Figure \ref{fig:wad01}. 

For $J \in \{L,R\}$, let $\alpha_{J,+}$ and $\alpha_{J,-}$ be the unique pair of lifts of $\boldsymbol{a}_J$ under $f^r$ that are attached to $\mathbf{J}'$ and are closest to $\mathbf{\Upsilon}'\backslash \mathbf{J}'$. Such lifts exist because $f^r: J^r \to J^0$ is a branched covering of degree at least $2$.

To estimate the width of $\mathcal{F}^r$, we will identify rectangles in $\mathcal{T}'$ that cross a number of proper curves in $U^r$ separating $\mathbf{L}'$ and $\mathbf{R}'$. These curves are constructed with the aid of $\boldsymbol{a}_L$ and $\boldsymbol{a}_R$ as follows. Refer to Figure \ref{fig:wad02} for a schematic picture.

\begin{lemma}[Middle curve]
\label{lem:middle-curve}
    There exist a proper curve $\boldsymbol{b}$ in $U^0$ and a proper curve $\beta$ in $U^r$ with the following properties.
    \begin{enumerate}[label=\textnormal{(\arabic*)}]
        \item $\boldsymbol{b}$ disjoint from $\mathbf{\Upsilon} \cup \textnormal{CV} \cup \boldsymbol{a}_L \cup \boldsymbol{a}_R$ and separates $\mathbf{L}$ from $\mathbf{R}$.
        \item $\mathcal{B}$ crosses $\boldsymbol{b}$ twice, $\mathcal{A}$ crosses $\boldsymbol{b}$ once, and $\mathcal{P}$ is disjoint from $\boldsymbol{b}$.
        \item $\beta$ is a lift of $\boldsymbol{b}$ under $f^r$ that separates $\mathbf{L}'$ from $\mathbf{R}'$.
        \item Every rectangle in $\mathcal{T}'$ crosses $\beta$ at most once, and $W(\mathcal{T}'[\beta]) = A+2B$.
    \end{enumerate}
\end{lemma}

\begin{proof}
    The existence of $\boldsymbol{b}$ satisfying (1) and (2) is clear. (See Figures \ref{fig:wad00} and \ref{fig:wad01}.) Let $Q$ denote the connected component of $U^0 \backslash \boldsymbol{b}$ containing $\mathbf{L}$. The unique lift $Q'$ of $Q$ under $f^r$ which contains $\mathbf{L}'$ must be disjoint from $\mathbf{R}'$. Since $f^r$ is a proper map on $U^r$, there exists a unique connected component $\beta$ of $\partial Q' \backslash \partial U^r$ that is a lift of $\boldsymbol{b}$ and separates $\mathbf{L}'$ and $\mathbf{R}'$.

    To prove (4), it suffices to show that every rectangle $\mathcal{R}$ in $\mathcal{B}$ admits exactly two distinct lifts in $\mathcal{T}'[\beta]$, and each of them crosses $\beta$ exactly once. If otherwise, then there would exist a unique lift $\mathcal{R}'$ of $\mathcal{R}$ in $\mathcal{T}'[\beta]$ which crosses $\beta$ exactly twice. In this case, $\mathcal{R}'$ would be crossing both $\alpha_{L,-}$ and $\alpha_{L,+}$, hence $\mathcal{R}$ would be crossing $\boldsymbol{a}_L$ twice, which is impossible.
\end{proof}

For $j \in \N$, let us consider the asymmetric width
\[
Z^j := W(\mathcal{A}^j) + 2 W(\mathcal{B}^j)
\]
on $U^j \backslash \Upsilon^j$. (One may compare with the notion of \emph{asymmetric modulus} in \cite{Lyu97}.)

\begin{corollary}[Monotonicity]
\label{cor:monotonicity}
    There exists a constant $C=C(m,\lambda)>0$ such that
    \[
    Z^r -C \leq Z^0.
    \]
\end{corollary}

\begin{proof}
    Consider $\mathcal{F}^r_{\text{sub}}$ from Lemma \ref{domination-property} and split them into $\mathcal{A}^r_{\text{sub}} \cup \mathcal{B}^r_{\text{sub}}$ according to the topological type. Observe that $\mathcal{A}^r_{\text{sub}}$ crosses $\beta$ once, whereas $\mathcal{B}^r_{\text{sub}}$ crosses $\beta$ twice. By Lemma \ref{lem:middle-curve} (4), $\mathcal{A}^r_{\text{sub}}$ admits a restriction $\mathcal{A}^r_{\text{res}}$ properly contained in $\mathcal{T}'[\beta]$, whereas $\mathcal{B}^r_{\text{sub}}$ admits two disjoint restrictions $\mathcal{B}^r_{\text{res},1}$ and $\mathcal{B}^r_{\text{res},2}$ that are properly contained in $\mathcal{T}'[\beta]$. Then,
    \begin{align*}
        Z^r - C 
        &\leq W(\mathcal{A}^r_{\text{res}}) + 2\left[W(\mathcal{B}^r_{\text{res},1}) \oplus W(\mathcal{B}^r_{\text{res},2}) \right] \\
        & \leq W(\mathcal{A}^r_{\text{res}}) + \frac{1}{2}\left[W(\mathcal{B}^r_{\text{res},1}) + W(\mathcal{B}^r_{\text{res},2})\right] \leq  W\left(\mathcal{T}'[\beta]\right) \leq A + 2B.
    \end{align*}
    At last, apply Lemma \ref{lem:correction} and we are done.
\end{proof}

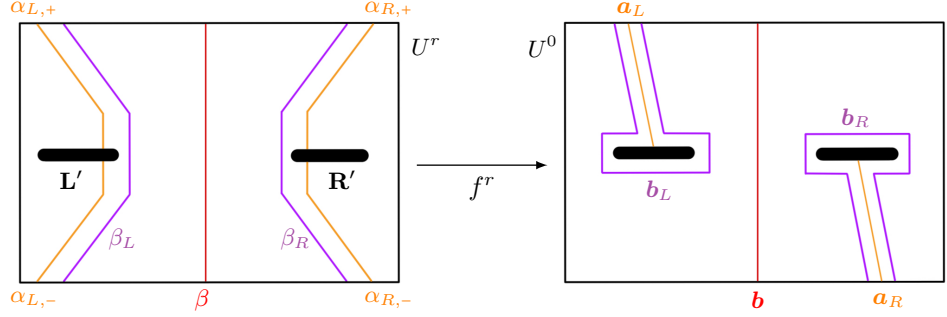
\begin{figure}
    \centering
    \begin{tikzpicture}
    \node[anchor=south west,inner sep=0] (image) at (0,0) {\includegraphics[width=0.98\linewidth]{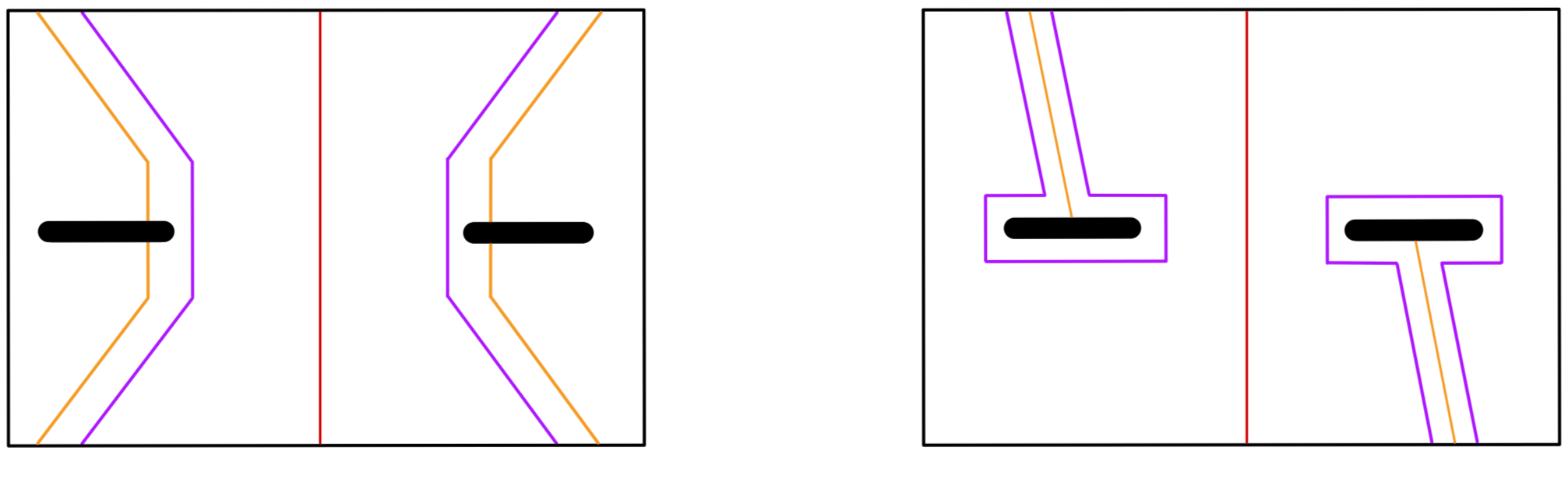}};
    \begin{scope}[
        x={(image.south east)},
        y={(image.north west)}
    ]
        \node [orange, font=\bfseries] at (0.66,1.025) {\small $\boldsymbol{a}_L$};
        \node [violet!70!white, font=\bfseries] at (0.69,0.415) {\small $\boldsymbol{b}_L$};
        \node [red, font=\bfseries] at (0.795,0.045) {\small $\boldsymbol{b}$};
        \node [violet!70!white, font=\bfseries] at (0.90,0.66) {\small $\boldsymbol{b}_R$};
        \node [orange, font=\bfseries] at (0.935,0.045) {\small $\boldsymbol{a}_R$};
        
        \node [orange, font=\bfseries] at (0.02,1.025) {\small $\alpha_{L,+}$};
        \node [orange, font=\bfseries] at (0.02,0.045) {\small $\alpha_{L,-}$};
        \node [violet!70!white, font=\bfseries] at (0.115,0.25) {\small $\beta_L$};
        \node [red, font=\bfseries] at (0.20,0.045) {\small $\beta$};
        \node [violet!70!white, font=\bfseries] at (0.30,0.25) {\small $\beta_R$};
        \node [orange, font=\bfseries] at (0.40,1.025) {\small $\alpha_{R,+}$};
        \node [orange, font=\bfseries] at (0.40,0.045) {\small $\alpha_{R,-}$};
        
        \draw[line width=0.5pt,-latex] (0.43,0.50) -- (0.57,0.50);
        \node [black, font=\bfseries] at (0.50,0.42) {$f^r$};

        \node [black, font=\bfseries] at (0.06,0.45) {\small $\mathbf{L}'$};
        \node [black, font=\bfseries] at (0.35,0.45) {\small $\mathbf{R}'$};        
        \node [black, font=\bfseries] at (0.44,0.90) {\small $U^r$};
        \node [black, font=\bfseries] at (0.565,0.90) {\small $U^0$};
    \end{scope}
\end{tikzpicture}
    
    \caption{A schematic diagram of the construction of separating proper curves $\beta$, $\beta_L$, and $\beta_R$.}
    \label{fig:wad02}
\end{figure}

Our next goal is to upgrade monotonicity to a strict loss. To do so, let us introduce two other separating curves $\boldsymbol{b}_L$ and $\boldsymbol{b}_R$. 

\begin{lemma}[Left and right curves]
\label{lem:left-and-right}
    There exist proper curves $\boldsymbol{b}_L$ and $\boldsymbol{b}_R$ in $U^0$ and proper curves $\beta_L$ and $\beta_R$ in $U^r$ with the following properties. For each $J \in \{L,R\}$, 
    \begin{enumerate}[label=\textnormal{(\arabic*)}]
        \item $\boldsymbol{b}_J$ is disjoint from $\mathbf{\Upsilon} \cup \textnormal{CV} \cup \boldsymbol{a}_L \cup \boldsymbol{a}_R \cup \boldsymbol{b}$ and separates $\mathbf{J}$ and $\boldsymbol{b}$;
        \item $\mathcal{B} \cup \mathcal{D}_J$ crosses $\boldsymbol{b}_J$ twice, $\mathcal{A}$ crosses $\boldsymbol{b}_J$ once, and $\mathcal{D} \backslash \mathcal{D}_J \cup \mathcal{E} $ is disjoint from $\boldsymbol{b}_J$;
        \item $\beta_J$ is a lift of $\boldsymbol{b}_J$ that separates $\mathbf{L}'$ and $\beta$ and is close to $\mathbf{J}' \cup \alpha_{J,+} \cup \alpha_{J,-}$;
        \item $W\left(\mathcal{T}'[\beta_J]\right) = A+2B+2D_J$.
    \end{enumerate}
    Moreover, the strip $\Pi \subset U^r$ cut out by $\beta_L$ and $\beta_R$ contains a piece $I$ of length $\asymp l_n$.
\end{lemma}

\begin{proof}
    For $J \in \{L,R\}$, pick an extremely small $\varepsilon>0$ such that the $\varepsilon$-neighborhood $O_J$ of $\mathbf{J}\cup \boldsymbol{a}_J$ is disjoint from $\text{CV} \backslash \mathbf{J}$. Let us set $\boldsymbol{b}_J := \partial O_J \cap U^0$, then items (1)--(3) immediately follow. Item (4) follows in a similar manner as the proof of Lemma \ref{lem:middle-curve}. Moreover, the existence of $I \subset \Pi$ follows from the property that $\mathbf{L}$ and $\mathbf{R}$ have combinatorial distance $\asymp l_n$.
\end{proof}

\subsection{Non-persistence induces width loss}
\label{ss:new-width}
We say that a rectangle in $\mathcal{T}'$ is \emph{persistent} if it crosses both $\beta_L$ and $\beta_R$, i.e. it belongs in 
\[
\mathcal{T}'_{\text{per}} := \mathcal{T}'[\beta_L] \cap \mathcal{T}'[\beta_R].
\]
Denote the total widths of persistent and non-persistent rectangles in $\mathcal{T}'[\beta]$ by
\[
    Z_{\textnormal{per}} := W(\mathcal{T}'_{\textnormal{per}}) \quad \text{and} \quad Z_{\textnormal{non}} := A+2B-Z_{\textnormal{per}}
\]
respectively. In this subsection, we prove the following non-dynamical result.

\begin{proposition}[Key estimate]
    \label{prop:change-in-width}
    There exists some constant $C=C(m,\lambda)>0$ such that
    \[
        Z^r -C \leq Z_{\textnormal{per}} + Z_{\textnormal{non}} \oplus 2(Z_{\textnormal{non}}+D).
    \]
\end{proposition}

The idea is captured in Figure \ref{fig:wad03}. Most leaves of $\mathcal{F}^r$ travel through either $\mathcal{T}'_{\text{per}}$ (the left part of the figure) or $\mathcal{T}'[\beta] \backslash \mathcal{T}'_{\text{per}}$ (the middle and the right parts). The former case gives the term $Z_{\text{per}}$. In the latter case, they must also travel through $(\mathcal{T}'[\beta_L] \cup \mathcal{T}'[\beta_R] )\backslash\mathcal{T}'_{\text{per}}$, which has total width $2(Z_{\text{non}}+D)$, and thus the series law can be applied to generate the harmonic sum. In Section \ref{ss:proof-of-loss}, we will show from this inequality that $Z^r$ shrinks provided that $Z_{\textnormal{per}}$ and $D$ are small relative to $K$. 

\begin{proof}

Consider the sublaminations $\mathcal{A}^r_{\text{sub}} \subset \mathcal{A}^r$ and $\mathcal{B}^r_{\text{sub}} \subset \mathcal{B}^r$ from Lemma \ref{domination-property}. For some $C=C(m,\lambda)>0$,
\[
W(\mathcal{A}^r)-C \leq W\left(\mathcal{A}^r_{\text{sub}}\right) \quad \text{and} \quad W(\mathcal{B}^r)-C \leq W\left(\mathcal{B}^r_{\text{sub}}\right),
\]
and every leaf of $\mathcal{A}^r_{\text{sub}} \cup \mathcal{B}^r_{\text{sub}}$ travels through rectangles in $\mathcal{T}'$. 

Let us assume that $\mathcal{B}^r$ is attached to $R^r$; if otherwise, $\mathcal{B}^r_{\text{sub}}$ would be empty because no rectangles in $\mathcal{T}'$ can cross $\alpha_{R,+} \cup \alpha_{R,-}$. Let us define two disjoint restrictions $\mathcal{B}^r_1$ and $\mathcal{B}^r_2$ of $\mathcal{B}^r_{\text{sub}}$ as follows. Denote by $Q$ the connected component of $U^r\backslash(\alpha_{L,+}\cup\alpha_{L,-})$ containing $L^r$. For $\gamma \in \mathcal{B}^r_{\text{sub}}$, let us fix a parametrization $\gamma:(0,1)\to U^r$ travelling around $L^r$ in an anticlockwise manner. Consider the set $T_\gamma$ of times $t \in (0,1)$ such that $\gamma(t)$ is in $f^{-r}(\Upsilon^0)\cap Q$. Note that $T_\gamma$ is non-empty because no rectangle in $\mathcal{T}'$ crosses both $\alpha_{L,+}$ and $\alpha_{L,-}$ simultaneously. Let $t_{\gamma,1}:=\min T_{\gamma}$ and $t_{\gamma,2}:=\max T_{\gamma}$. Then, we define restrictions
\[
    \mathcal{B}^r_1 := \left\{ \gamma|_{(0,t_{\gamma,1})}\: : \: \gamma \in \mathcal{B}^r_{\text{sub}} \right\} \quad \text{and} \quad \mathcal{B}^r_2 := \left\{ \gamma|_{(t_{\gamma,2},1)}\: : \: \gamma \in \mathcal{B}^r_{\text{sub}} \right\}.
\]

Let us consider the lamination 
\[
    \mathcal{G}:=\mathcal{A}^r_{\text{sub}} \cup \mathcal{B}^r_1 \cup \mathcal{B}^r_2
\]
Every leaf of $\mathcal{G}$ crosses $\beta_L$, ends at $R^r$ (thus crosses $\beta_R$ too), and travels through rectangles in $\mathcal{T}'$. Moreover, there is some $C=C(m,\lambda)>0$ such that
\begin{equation}
\label{eqn:Z-and-G}
    Z^r - C \leq W(\mathcal{A}^r_{\text{sub}}) + 2 \left[ W(\mathcal{B}^r_1) \oplus W(\mathcal{B}^r_2) \right] \leq W(\mathcal{G}).
\end{equation}

For $\gamma \in \mathcal{G}$, let $\gamma_0$, $\gamma_L$, and $\gamma_R$ be the connected components of $\gamma\backslash f^{-r}\left(\Upsilon^0\right)$ that are crossing $\beta$, $\beta_L$, and $\beta_R$ respectively. Let us split $\mathcal{G}$ into a disjoint union of three sublaminations $\mathcal{G}_{\text{per}} \cup \mathcal{G}_- \cup \mathcal{G}_+$ defined as follows. For $\gamma \in \mathcal{G}$,
    \begin{itemize}
        \item[$\rhd$] $\gamma \in \mathcal{G}_{\text{per}}$ if $\gamma_0$ crosses both $\beta_L$ and $\beta_R$; 
        \item[$\rhd$] $\gamma \in \mathcal{G}_-$ if $\gamma_0$ crosses $\beta_L$ but not $\beta_R$;
        \item[$\rhd$] $\gamma \in \mathcal{G}_+$ if $\gamma_0$ does not cross $\beta_L$.
    \end{itemize}
For $\bullet \in \{+,-\}$ and $\text{x} \in \{0,L,R\}$, let us denote $\mathcal{G}_{\bullet,\text{x}}:= \{\gamma_{\text{x}} \: : \: \gamma \in \mathcal{G}_\bullet\}$. By design, $\mathcal{G}_{-,0}$ and $\mathcal{G}_{-,R}$ are disjoint, and $\mathcal{G}_{+,0}$ and $\mathcal{G}_{+,L}$ are disjoint. See Figure \ref{fig:wad03}. Then,
\begin{align}
W(\mathcal{G}) &\leq W(\mathcal{G}_{\text{per}}) + W(\mathcal{G}_{-,0}) \oplus W(\mathcal{G}_{-,R}) + W(\mathcal{G}_{+,0}) \oplus W(\mathcal{G}_{+,L}) \nonumber \\
&\leq W(\mathcal{G}_{\text{per}}) + W(\mathcal{G}_{-,0} \cup \mathcal{G}_{+,0}) \oplus \left[ W(\mathcal{G}_{-,R}) + W(\mathcal{G}_{+,L}) \right]. \label{eqn:G-into-three}
\end{align}
Since $\mathcal{G}_{\text{per}}$ travels through $\mathcal{T}'_{\text{per}}$ and $\mathcal{G}_{-,0} \cup \mathcal{G}_{+,0}$ travels through $\mathcal{T}'[\beta] \backslash \mathcal{T}'_{\text{per}}$, then 
\begin{equation}
\label{eqn:per-and-non}
W(\mathcal{G}_{\text{per}}) \leq Z_{\text{per}} \quad \text{and} \quad W(\mathcal{G}_{-,0} \cup \mathcal{G}_{+,0}) \leq Z_{\text{non}}.
\end{equation}
Since $\mathcal{G}_{-,R}$ travels through $\mathcal{T}'[\beta_R]\backslash \mathcal{T}'_{\text{per}}$ and $\mathcal{G}_{+,L}$ travels through $\mathcal{T}'[\beta_L]\backslash \mathcal{T}'_{\text{per}}$, then by Lemma \ref{lem:left-and-right},
\begin{equation}
\label{eqn:plus-minus-L-R}
W(\mathcal{G}_{-,R}) \leq Z_{\text{non}} + 2 D_R \quad \text{and} \quad W(\mathcal{G}_{+,L}) \leq Z_{\text{non}} + 2 D_L.
\end{equation}
Hence, combining (\ref{eqn:Z-and-G}), (\ref{eqn:G-into-three}), (\ref{eqn:per-and-non}), and (\ref{eqn:plus-minus-L-R}) gives us the desired inequality.
\end{proof}

\begin{figure}
    \centering
    \begin{tikzpicture}
    \node[anchor=south west,inner sep=0] (image) at (0,0) {\includegraphics[width=\linewidth]{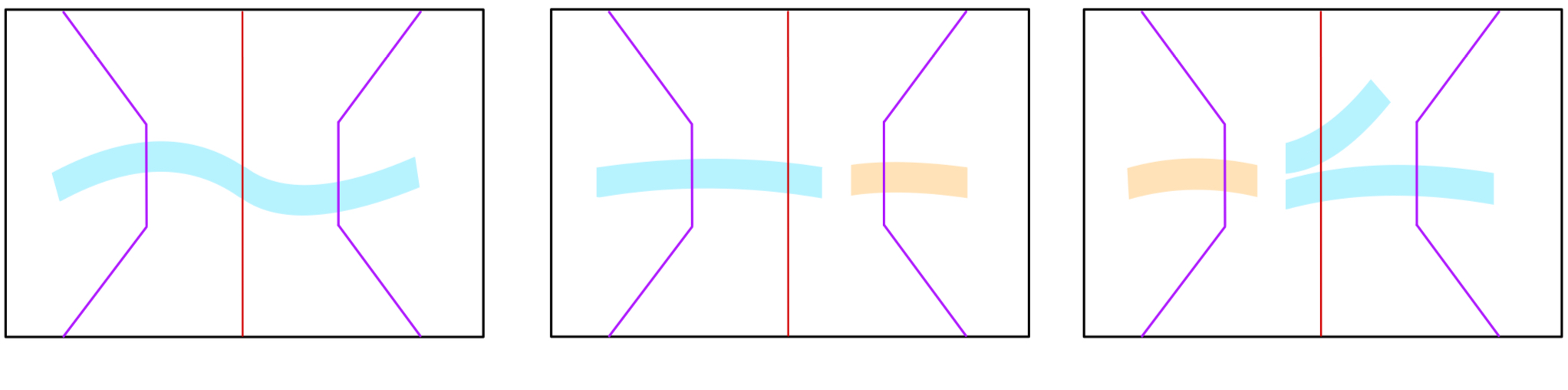}};
    \begin{scope}[
        x={(image.south east)},
        y={(image.north west)}
    ]
        \node [blue!40!cyan, font=\bfseries] at (0.04,0.46) {\footnotesize $\mathcal{G}_{\text{per}}$};
        \node [blue!40!cyan, font=\bfseries] at (0.385,0.455) {\footnotesize $\mathcal{G}_{-,0}$};
        \node [orange, font=\bfseries] at (0.60,0.44) {\footnotesize$\mathcal{G}_{-,R}$};
        \node [blue!40!cyan, font=\bfseries] at (0.875,0.42) {\footnotesize $\mathcal{G}_{+,0}$};
        \node [orange, font=\bfseries] at (0.725,0.45) {\footnotesize $\mathcal{G}_{+,L}$};
        
        \node [violet!70!white, font=\bfseries] at (0.05,0.04) {\small $\beta_L$};
        \node [red, font=\bfseries] at (0.155,0.04) {\small $\beta$};
        \node [violet!70!white, font=\bfseries] at (0.275,0.04) {\small $\beta_R$};
        
        \node [violet!70!white, font=\bfseries] at (0.39,0.04) {\small $\beta_L$};
        \node [red, font=\bfseries] at (0.5,0.04) {\small $\beta$};
        \node [violet!70!white, font=\bfseries] at (0.62,0.04) {\small $\beta_R$};
        
        \node [violet!70!white, font=\bfseries] at (0.725,0.04) {\small $\beta_L$};
        \node [red, font=\bfseries] at (0.84,0.04) {\small $\beta$};
        \node [violet!70!white, font=\bfseries] at (0.96,0.04) {\small $\beta_R$};
    \end{scope}
\end{tikzpicture}
    
    \caption{The lamination $\mathcal{G}=\mathcal{G}_{\text{per}}\cup\mathcal{G}_-\cup\mathcal{G}_+$ has width at least $Z^r$. $\mathcal{G}_{\text{per}}$ crosses both $\beta_L$ and $\beta_R$. In contrast, $\mathcal{G}_-$ overflows $\mathcal{G}_{-,0}$ and $\mathcal{G}_{-,R}$, whereas $\mathcal{G}_+$ overflows $\mathcal{G}_{+,0}$ and $\mathcal{G}_{+,L}$ .}
    \label{fig:wad03}
\end{figure}
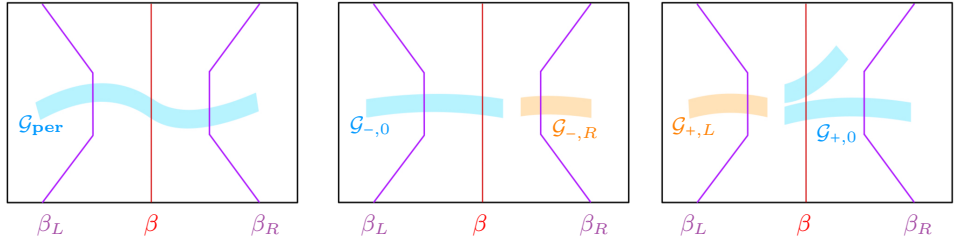

\subsection{Persistence amplifies degeneration}
\label{ss:persistence}

Let us consider the strip $\Pi \subset U^r$ from Lemma \ref{lem:left-and-right}, and a proper lamination $\mathcal{L}_{\text{per}}$ in $\Pi$ that is a restriction of the canonical lamination of $\mathcal{T}'_{per}$. Clearly, $\mathcal{L}_{\text{per}}$ connects $\beta_L$ and $\beta_R$, and its width satisfies
\[
    W(\mathcal{L}_{\text{per}}) \geq Z_{\textnormal{per}}.
\]

Let us denote by $\text{CP}=\text{CP}(f^r)$ the set of critical points of $f^r$.

\begin{lemma}
    \label{homotopy}
    All leaves of $\mathcal{L}_{\textnormal{per}}$ are properly homotopic to each other in $\Pi \backslash \textnormal{CP}$.
\end{lemma}

\begin{proof}
    Pick any two distinct leaves $\gamma_1$ and $\gamma_2$ of $\mathcal{L}_{\text{per}}$. Then, $\gamma_1 \cup \gamma_2 \cup \beta_L \cup \beta_R$ must enclose a disk $O'$ contained in $\Pi$. Denote by $O$ the disk enclosed by $f^r(\gamma_1) \cup f^r(\gamma_2) \cup \boldsymbol{b}_L \cup \boldsymbol{b}_R$. By the maximum principle, $f^r: O' \to O$ is a proper holomorphic map, and by the argument principle, $f^r|_{O'}$ must be univalent. In particular, $O'$ contains no critical points of $f^r$.
\end{proof}

\begin{lemma}[Persistence $\longrightarrow$ $\tau$-degeneration]
\label{greater-deg}
    For any $M > 1$, there exist constants $m_0=m_0(M) \in \N$ and $\mathbf{K}_2=\mathbf{K}_2(M,\lambda)>0$ such that if
    \begin{equation}
    \label{eqn:per-assumption}
    m \geq m_0, \quad K \geq \mathbf{K}_2, \quad \text{and} \quad Z_\textnormal{per} \geq 0.1 K,
    \end{equation}      
    then there is a level $n+m$ combinatorial piece $J$ of width $W_\tau(J) \geq M K$.
\end{lemma}

\begin{proof}
    Let us set $t:= q_{n+m-2}$ and $s:= r-t$. Assume that (\ref{eqn:per-assumption}) holds, and so
    \begin{equation}
    \label{eqn:L-per-assumption}
        W(\mathcal{L}_{\text{per}}) \geq 0.1 K.
    \end{equation}
    
    For every critical point $c$ of $f$, the backward orbit $\{(f|_\Hq)^{-i}(c) \}_{i=0,\ldots, t-1}$ partitions $\Hq$ into pieces of length between $l_{n+m-2}$ and $l_{n+m-4}$. By lifting this tiling by $f^s$, observe that $\text{CP}$ partitions $f^{-s}(\Hq)$ into preimages of pieces of length at most $l_{n+m-4}$.
    
    Before we proceed, we will first sketch the idea behind our construction. The horizontal lift $\mathcal{L}_{\text{per}}$ of the persistent lamination must cross through a large number of \emph{fences}, which are connected subsets of $f^{-s}(\Hq)$ separating $\beta_L$ and $\beta_R$ in $\Pi$, as shown in Figure \ref{fig:fences-01}. These fences can be chosen such that their images under $f^s$ have alternating configuration shown in Figure \ref{fig:fences-02}. As these fences are tiled by $\text{CP}$, then by Lemma \ref{homotopy}, $\mathcal{L}_{\text{per}}$ must intersect a common tile $G_i$ from each fence $\#_i$. We then apply the series law to obtain a large $\tau$-degeneration.

\begin{figure}
        \centering
        \begin{tikzpicture}
    \node[anchor=south west,inner sep=0] (image) at (0,0) {\includegraphics[width=1\linewidth]{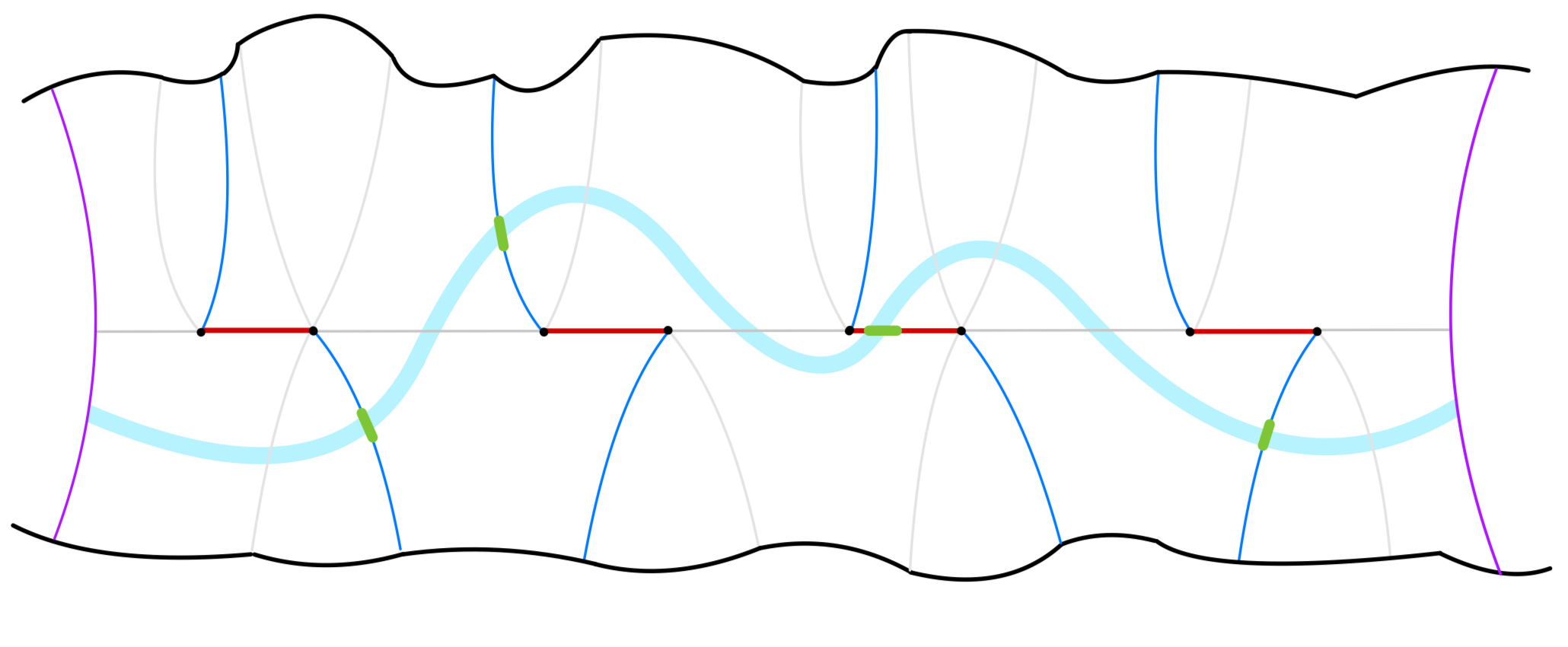}};
    \begin{scope}[
        x={(image.south east)},
        y={(image.north west)}
    ]
        \node [black, font=\bfseries] at (0.50,0.09) {$\Pi$};
        \node [blue!40!cyan, font=\bfseries] at (0.12,0.27) {\footnotesize{$\mathcal{L}_{\text{per}}$}};
        \node [violet!70!white, font=\bfseries] at (0.04,0.50) {$\beta_L$};
        \node [violet!70!white, font=\bfseries] at (0.95,0.50) {$\beta_R$};
        \node [red, font=\bfseries] at (0.17,0.44) {\footnotesize{$P_1$}};
        \node [red, font=\bfseries] at (0.38,0.44) {\footnotesize{$P_2$}};
        \node [red, font=\bfseries] at (0.59,0.44) {\footnotesize{$P_3$}};
        \node [red, font=\bfseries] at (0.795,0.44) {\footnotesize{$P_4$}};
        \node [blue, font=\bfseries] at (0.165,0.75) {\footnotesize{$\gamma_1^\infty$}};
        \node [blue, font=\bfseries] at (0.30,0.75) {\footnotesize{$\gamma_2^\infty$}};
        \node [blue, font=\bfseries] at (0.577,0.75) {\footnotesize{$\gamma_3^\infty$}};
        \node [blue, font=\bfseries] at (0.718,0.75) {\footnotesize{$\gamma_4^\infty$}};
        \node [blue, font=\bfseries] at (0.267,0.22) {\footnotesize{$\gamma_1^0$}};
        \node [blue, font=\bfseries] at (0.395,0.22) {\footnotesize{$\gamma_2^0$}};
        \node [blue, font=\bfseries] at (0.648,0.22) {\footnotesize{$\gamma_3^0$}};
        \node [blue, font=\bfseries] at (0.777,0.20) {\footnotesize{$\gamma_4^0$}};
        \node [green!70!black, font=\bfseries] at (0.26,0.34) {\footnotesize{$G_1$}};
        \node [green!70!black, font=\bfseries] at (0.345,0.62) {\footnotesize{$G_2$}};
        \node [green!70!black, font=\bfseries] at (0.565,0.535) {\footnotesize{$G_3$}};
        \node [green!70!black, font=\bfseries] at (0.825,0.28) {\footnotesize{$G_4$}};
    \end{scope}
\end{tikzpicture}

        \caption{The lamination $\mathcal{L}_{\text{per}}$ crosses \emph{fences} $\#_j = \gamma_j^\infty \cup P_j \cup \gamma_j^0$ through \emph{gates} $G_j$ in consecutive order. }
        \label{fig:fences-01}
\end{figure}
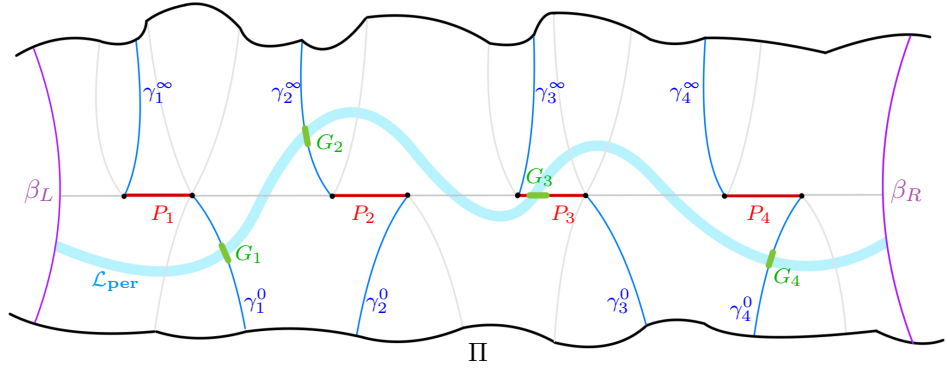
    
    Now, let us delve into the details. By Lemma \ref{lem:left-and-right}, there exists a piece $I$ in $\Pi$ of length $\asymp l_n$. Recall the three distinct cases \ref{case:herman-curve}, \ref{case:herman-ring}, and \ref{case:boundary-ring} introduced in \S\ref{ss:setup}.
    \vspace{0.1in}
    
    \noindent \textbf{Case \ref{case:herman-curve} or \ref{case:herman-ring}:} 
    Assuming $m$ is large enough (depending on $N$), there is a sequence 
    \[
        x_1^\infty, \, x_1^0, \, x_2^\infty, \, x_2^0, \, \ldots, \, x_{2N}^\infty, \, x_{2N}^0
    \]
    of critical points of $f^s$, written in consecutive order, with the following properties.
    \begin{itemize}
        \item[(i)] All the $x_i^\infty$'s and $x_i^0$'s are located on $I$, with $x_1^\infty$ being the closest to $\beta_L$ and $x_{2N}^0$ being the closest to $\beta_R$ combinatorially.
        \item[(ii)] Every $x_i^\infty$ (resp. $x_i^0$) is the root of an outer (resp. inner) bubble $\mathbf{B}_i^\infty$ (resp. $\mathbf{B}_i^0$) of generation at most $s$.
        \item[(iii)] The pieces $P_i:= [x_i^\infty, x_i^0]$ have length at least $l_{n+m-4}$ and are of distance at least $\frac{\tau-1}{2} l_{n+m-4}$ away from each other.
    \end{itemize}
    
    For every odd (resp. even) $i$ and $\bullet \in \{0,\infty\}$, the critical value $f^s(x_i^\bullet)$ partitions\footnote{In Case \ref{case:herman-ring}, we partition using the radial segment in $\Hq$ containing $f^s(x_i^\bullet)$.} $\overline{U^t} \cap \Hq$ into two pieces, one of which, which we will denote by $J_i^\bullet$, intersects $L^{t}$ (resp. $R^{t}$). Denote by $\gamma^\bullet_i$ the lift of $J_i^\bullet$ under $f^s$ that lies within the bubble $\mathbf{B}_i^\bullet$. By (ii), each $\gamma^\bullet_i$ intersects $\Hq$ precisely at the critical point $x^\bullet_i$. Define our \emph{fences} as
    \[
    \#_i := \gamma^\infty_i \cup P_i \cup \gamma^0_i.
    \]
    By (i) and (iii), they satisfy the following properties. (See Figures \ref{fig:fences-01} and \ref{fig:fences-02}.)
    \begin{enumerate}
        \item[(iv)] The $\#_i$'s are pairwise disjoint connected subsets of $\overline{\Pi} \cap f^{-s}(\Hq)$ which separate $\beta_L$ and $\beta_R$.
        \item[(v)] For each $l\in\{1,\ldots,N\}$, the images $f^s(\#_{2l-1})$ and $f^s(\#_{2l})$ are disjoint pieces in $\Hq$ that are at least $\frac{\tau-1}{2} l_{n+m-4}$ apart from each other.
    \end{enumerate}
    
    By property (iv), $\mathcal{L}_{\text{per}}$ crosses each fence in consecutive order, namely $\#_i$ first before $\#_{i+1}$. As $\text{CP}$ induces tiling on fences, Lemma \ref{homotopy} implies the existence of connected compact subsets $G_i \subset \#_i$ (the \emph{gates} of the fence) where the images $J_i:= f^s(G_i)$ are level $n+m-4$ combinatorial pieces and $\mathcal{L}_{\text{per}}$ crosses the $G_i$'s in consecutive order. 

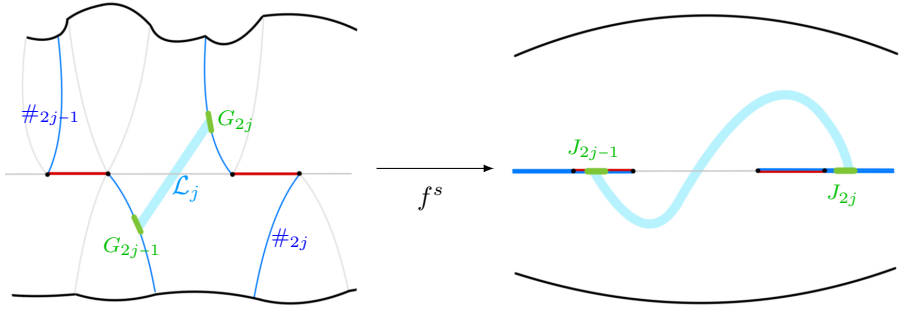
\begin{figure}
        \centering

\begin{tikzpicture}
    \node[anchor=south west,inner sep=0] (image) at (0,0) {\includegraphics[width=0.95\linewidth]{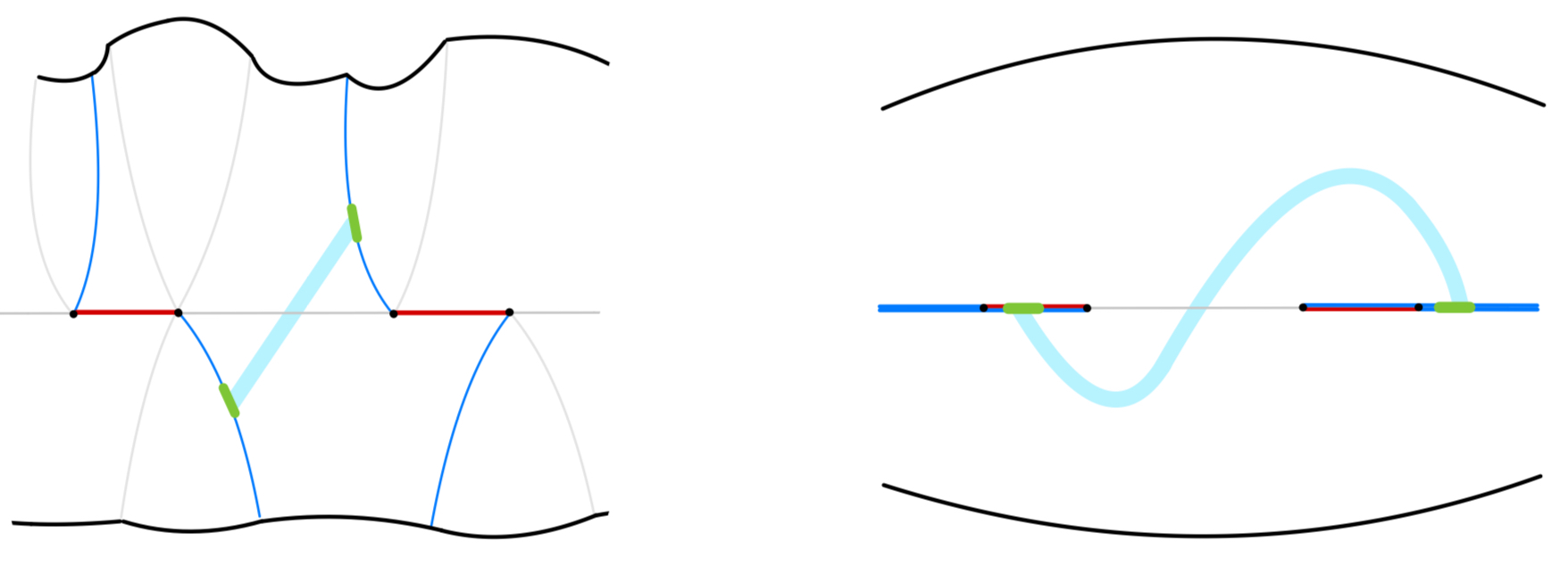}};
    \begin{scope}[
        x={(image.south east)},
        y={(image.north west)}
    ]
        \draw [black, -latex] (0.41,0.48)--(0.54,0.48);
        \node [black, font=\bfseries] at (0.47,0.39) {$f^s$};
        \node [blue, font=\bfseries] at (0.05,0.64) {\footnotesize{$\#_{2j-1}$}};
        \node [blue, font=\bfseries] at (0.315,0.27) {\footnotesize{$\#_{2j}$}};
        \node [green!75!black, font=\bfseries] at (0.65,0.535) {\footnotesize{$J_{2j-1}$}};
        \node [green!75!black, font=\bfseries] at (0.925,0.40) {\footnotesize{$J_{2j}$}};
        \node [green!75!black, font=\bfseries] at (0.14,0.245) {\footnotesize{$G_{2j-1}$}};
        \node [green!75!black, font=\bfseries] at (0.255,0.62) {\footnotesize{$G_{2j}$}};
        \node [blue!40!cyan, font=\bfseries] at (0.20,0.42) {$\mathcal{L}_j$};
    \end{scope}
\end{tikzpicture}

        \caption{The fences $\#_{2j-1}$ and $\#_{2j}$ are constructed such that their images under $f^s$ have $\tau$-separation.}
        \label{fig:fences-02}
\end{figure}

    Therefore, there exist pairwise disjoint laminations $\mathcal{L}_1,\ldots, \mathcal{L}_N$ such that each $\mathcal{L}_j$ is a restriction of $\mathcal{L}_{\text{per}}$ that connects $G_{2j-1}$ and $G_{2j}$. Let $k$ be such that $\mathcal{L}_k$ is the widest among all the $\mathcal{L}_j$'s. By property (v), since each $J_i$ lies within $f^s(\#_i)$, then $f^s(\mathcal{L}_k)$ overflows $\mathcal{F}_\tau(J_{2k})$. By Propositions \ref{series law} and \ref{harmonicsum}, and by (\ref{eqn:L-per-assumption}),
    \[
    W_\tau(J_{2k}) \geq W\left(f^s(\mathcal{L}_k)\right) \geq N \cdot W\left(f^s\left(\mathcal{L}_{\text{per}}\right)\right) = N \cdot W\left(\mathcal{L}_{\text{per}}\right) \geq 0.1 N K.
    \]
    Consider the constant $C>1$ from Proposition \ref{bounded-type}. There exists a level $n+m$ combinatorial subpiece $J$ of $J_{2k}$ with width $W_\tau(J) \geq 0.1 C^{-4} N K$. Finally, set $N = \lceil 10 C^4 M \rceil$ and we are done.
    \vspace{0.1in}
    
    \noindent \textbf{Case \ref{case:boundary-ring}:} The proof is similar to the previous case, but the construction of fences needs a small adjustment. Following Remark \ref{adjustment-in-case-c}, we assume that the $U^j$'s are disjoint from the connected component of the complement of the Herman ring $\He$ containing $0$. In particular, $U^{r}$ does not contain any inner bubbles. We will instead take the bubbles $\mathbf{B}_i^\infty$ and $\mathbf{B}_i^0$ described in (ii) to both be outer bubbles. Although the corresponding fences $\#_i$ will no longer separate $\beta_L$ and $\beta_R$, we claim that most of $\mathcal{L}_{\text{per}}$ still cross every fence in consecutive order. 
    
    Indeed, the set of leaves in $\mathcal{L}_{\text{per}}$ that are disjoint from some fence $\#_i$ overflows the family $\mathcal{L}'_i$ of curves in $\He \cap U^{r}$ that skip $P_i$, i.e. they all connect two disjoint intervals in $\Hq$ are adjacent to $P_i$ and are at most $\lambda l_n$ in length. By uniformizing $\He$ and applying Proposition \ref{log-rule}, the width of $\mathcal{L}'_i$ is at most some constant depending on $\lambda$ and $N$. Therefore, for sufficiently large $\mathbf{K}_2=\mathbf{K}_2(\lambda,N)>0$, we can assume that the width of the sublamination $\mathcal{L}''$ consisting of leaves in $\mathcal{L}_{\text{per}}$ that cross the fences in consecutive order is at least half of $\mathcal{L}_{\text{per}}$. The same remaining argument holds for $\mathcal{L}''$, and at the last moment we take $N=\lceil 20 C^4 M \rceil$ instead.
\end{proof}

\subsection{Proof of Proposition \ref{pos-entropy}}
\label{ss:proof-of-loss}

The results in the preceding subsections can be summarized as follows.

\begin{lemma}[Degeneration vs. loss of $Z^j$]
\label{lem:deg-or-loss-of-width}
    Given any $M \geq 1$, there exist constants $m=m(M) \in \N$, $\nu=\nu(M) \in (0,1)$, and $\mathbf{K}_1 = \mathbf{K}_1(M,\lambda)>0$ such that if $W(\mathcal{F}^j) \geq \mathbf{K}_1$ for some $j \in \N$, then either
    \begin{enumerate}[label=\textnormal{(\arabic*)}]
        \item there is a level $n+m$ combinatorial piece $J$ of width $W_\tau(J) \geq M\cdot W(\mathcal{F}^j)$, or
        \item $Z^{j+q_{n+m}} \leq \nu Z^j$.
    \end{enumerate}
\end{lemma}

\begin{proof}
    Let $K=W(\mathcal{F}^0)$ and $r=q_{n+m}$. Following Sections \S\ref{ss:decomposition}--\ref{ss:persistence}, assume $j=0$ without loss of generality and denote by $C=C(m,\lambda)$ any positive constant depending only on $m$ and $\lambda$. By Lemmas \ref{lem:DE-degeneration} and \ref{greater-deg}, for sufficiently high integers $m$ and $\kappa$ depending on $M$, either item (1) holds, or
    \begin{equation}
    \label{eqn:per-degeneration-01}
        D \leq \kappa K \qquad \text{and} \qquad Z_{\text{per}} \leq 0.1 K.
    \end{equation}
    We will show that the latter assertion implies (2). By Proposition \ref{prop:change-in-width},
    \begin{align*}
        Z^r -C & \leq Z_{\textnormal{per}} + (Z^0-Z_{\textnormal{per}}) \oplus (2+ 2\kappa)Z^0.
    \end{align*}
    Set $\nu' := 0.1+0.9 \oplus (2+2\kappa)$; clearly, $0<\nu'<1$. By (\ref{eqn:per-degeneration-01}), the inequality simplifies to
    \begin{align*}
        Z^{r} -C & \leq \nu'Z^0.
    \end{align*}
    Set $\nu= (1+\nu')/2$ and assume $\mathbf{K}_1 \geq 2C/(1-\nu')$. Then, $Z^r \leq \nu Z^0$.
\end{proof}

At last, we are ready to prove the main result of this section. We will apply Lemma \ref{lem:deg-or-loss-of-width} many times until the shrinking factor is as low as we want.

\begin{proof}[Proof of Proposition \ref{pos-entropy}]
Fix $\Delta>1$ and $\delta \in (0,1)$. We will be applying Lemma \ref{lem:deg-or-loss-of-width} using the constant $M=\delta^{-1}\Delta$. Consider the constants $m$, $\nu$, and $\mathbf{K_1}$ from the lemma. Set $r:=q_{n+m}$ and $\mathbf{K}:=\delta^{-1} \mathbf{K}_1$, and let us assume that $W\left(\mathcal{F}^0\right) = K \geq \mathbf{K}$. Let us pick $\mathbf{t} \in \N$ such that $\nu^{\mathbf{t}}\leq \delta/2$. Our goal is to prove that either
\begin{enumerate}
    \item[(a.)] there is a level $n+m$ combinatorial piece $J$ of $\tau$-width at least $\Delta K$,
\end{enumerate}
or there is some $t$ between $1$ and $\mathbf{t}$ such that
\begin{enumerate}
    \item[(b.$t$)] $W\left(\mathcal{F}^{rt}\right) \leq \delta K$.
\end{enumerate}
The proof below involves another related assertion, which is
\begin{enumerate}
    \item[\textnormal{(c.$t$})] $Z^{rt} \leq \nu^t Z^0$.
\end{enumerate}

\begin{claim}
    If (c.$t$) holds, then either (a.), or (b.$t$), or (c.$t+1$) holds.
\end{claim}

\begin{proof}
    Suppose (c.$t$) holds and (b.$t$) fails. By the lemma, either there is a level $n+m$ combinatorial piece $J$ of $\tau$-width at least $\delta^{-1} \Delta \cdot W(\mathcal{F}^{rt})$, or $Z^{r(t+1)} \leq \nu Z^{rt}$. If the former assertion holds, since (b.$t$) does not hold, then
    \[
        W_\tau(J) \geq \delta^{-1} \Delta \cdot W\left(\mathcal{F}^{rt}\right) \geq \Delta K.
    \]
    If the latter assertion holds instead, then by (c.$t$), $Z^{r(t+1)} \leq \nu^{t+1} Z^0$.
\end{proof}

Trivially, (c.0) holds. As we apply the claim above for $t=0,1,\ldots,\mathbf{t}-1$, we conclude that either (a.) holds, or (b.$t$) holds for some $t$ between $1$ and $\mathbf{t}-1$, or (c.$\mathbf{t}$) holds. The latter case implies (b.$\mathbf{t}$) because
\[
    W\left(\mathcal{F}^{r\mathbf{t}}\right) \leq Z^{r\mathbf{t}} \leq \nu^{\mathbf{t}} Z^0 \leq \frac{\delta}{2} Z^0 \leq \delta K.
\]
Therefore, either (a.) holds or (b.$t$) holds for some $t \leq \mathbf{t}$.
\end{proof}


\section{A priori bounds}
\label{sec:a-priori-bounds}
We are now prepared to prove the first main theorem of the paper. The results in Sections \S\ref{sec:bubble-wave-argument}--\ref{sec:loss-of-horizontal-width} are compiled together to obtain the following theorem.

\begin{theorem}[Amplification Theorem]
\label{amplification}
    There is an absolute constant $\tau > 1$ and some constants $\mathbf{K} >1$, $m \in \N$, and $N \in \mathbb{N}$ depending only on $d_0$, $d_\infty$, and $\beta(\theta)$ such that if 
    \begin{center}
        there is a $[K,\tau]$-wide combinatorial piece $I \subset \Hq$ of level $n\geq N$
    \end{center}
    where $K \geq \mathbf{K}$ then
    \begin{center}
        there is a $[2K,\tau]$-wide combinatorial piece $J \subset \Hq$ of level $n'\geq N$
    \end{center}
    where $|n'-n| \leq m$.
\end{theorem}

The motivation behind the Amplification Theorem comes from D. Dudko and Lyubich's motto in \cite{DL22}: 
\[
\text{\emph{``If life is bad now, it will be worse tomorrow.``}}
\]
This is in the same spirit as Kahn's general strategy in his proof of a priori bounds for infinitely renormalizable quadratic polynomials of bounded primitive combinatorics in \cite{K06}.

\begin{proof}
    Set $\tau:=10$. Fix a large constant $\lambda \gg \tau$, and set $N := \thres_\lambda$ and $m := \step_\lambda + \step$, where $\step_\lambda$ and $\step$ are constants from Theorems \ref{demotion} and \ref{promotion} respectively. We will take $\mathbf{K}$ to be sufficiently high so that all the arguments below hold. 
    
    Suppose $I$ is a $[K,\tau]$-wide level $\geq N$ combinatorial piece in $\Hq$, where $K\geq \mathbf{K}$. By Theorem \ref{promotion}, either there is a $[2K,\tau]$-wide combinatorial piece $J$ or there is a $[\chi K, \lambda]$-wide combinatorial piece $L$. In the latter case, apply Proposition \ref{spreading-lambda} for the value $\Xi =  2/\chi$ such that either there is a $[2K, \tau]$-wide combinatorial piece $J$ or there is an almost tiling $\mathcal{I}''$ consisting of $[\xi \chi K, \tau]$-wide pieces. If the latter holds, apply Theorem \ref{demotion} to $\mathcal{I}''$ to obtain a $[\Pi_\lambda \xi \chi K, \tau]$-wide combinatorial piece $J \subset \Hq$. Refer to Figure \ref{fig:implication-diagram} for an illustration. Finally, we choose the constant $\lambda$ such that $\Pi_\lambda \xi \chi \geq 2$. Then, $J$ is the piece we are looking for.
\end{proof}

\begin{figure}
    \centering
    \begin{tikzpicture}[node distance = 2.5cm, auto]
    \tikzset{edge/.style = {->,> = latex'}}
    \node [special] (1) {$W_\tau(I) \geq K$ \\for some $I$};
    \node [process, below of=1, node distance=2cm] (4) {Thm \ref{promotion}};
    \node [block, right of=4, node distance=4cm] (2) {$W_\lambda(I')\geq \chi K$ for some $I'$};
    \node [process, below of=2, node distance=2cm] (5) {Prop \ref{spreading-lambda}};
    \node [block, right of=5, node distance=4cm] (3) {$W_\lambda(I'')\geq \xi\chi K$ for all $I''$ in some $\mathcal{I}''$};
    \node [process, below of=3, node distance=2cm] (6) {Thm \ref{demotion}};
    \node [special, below of=5, node distance=3cm] (goal) {$W_\tau(J)\geq 2K$ for some $J$};
    \draw[edge, thick] (1) to (4);
    \draw[edge, thick] (4) to (2);
    \draw[edge, thick] (4) to (goal);
    \draw[edge, thick] (2) to (5);
    \draw[edge, thick] (5) to (3);
    \draw[edge, thick] (5) to (goal);
    \draw[edge, thick] (3) to (6);
    \draw[edge, thick] (6) to (goal);
\end{tikzpicture}
    \caption{Implication diagram illustrating the amplification process.}
    \label{fig:implication-diagram}
\end{figure}
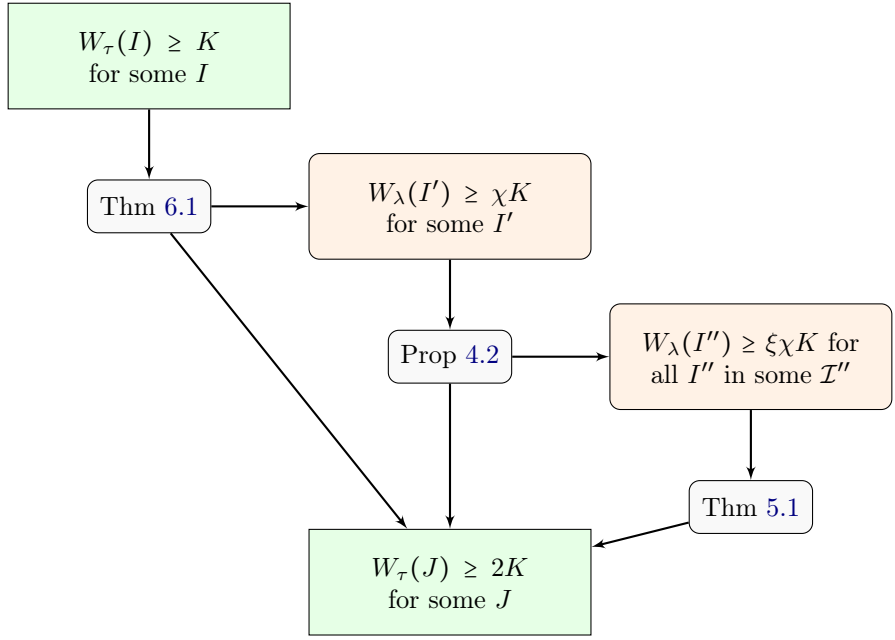

Below is a direct consequence of the Amplification Theorem in Case \ref{case:herman-curve}.

\begin{corollary}
\label{APB-deg}
    The Herman quasicircle $\Hq$ of every rational map $f \in \HQspace_{d_0, d_\infty,\theta}$ has dilatation depending only on $d_0$, $d_\infty$ and $\beta(\theta)$.
\end{corollary}

\begin{proof}
    Let $\tau$, $\mathbf{K}$ and $N$ be constants from the previous theorem, and let $C$ be the constant from Proposition \ref{bounded-type}. We claim that every interval $I \subset \Hq$ with $|I|\leq l_N$ has bounded $\tau$-width: $W_\tau(I) \leq C \mathbf{K}$. By Proposition \ref{main-prop} (2), this is enough to prove the corollary.
    
    Suppose for a contradiction that there is an interval $I$ of length $|I| \leq l_N$ that is $[CK,\tau]$-wide for some $K \geq \mathbf{K}$. Then, there is a $[K,\tau]$-wide combinatorial interval $I_0 \subset I$ of some level $\geq N$. By the Amplification Theorem, we obtain an infinite sequence of level $\geq N$ combinatorial intervals $\{I_j\}_{j=0,1,2,\ldots}$ such that each $I_j$ is $[2^j K, \tau]$-wide. This contradicts Proposition \ref{main-prop} (1).
\end{proof}

Proposition \ref{main-prop} (1) is not applicable in Case \ref{case:herman-ring}. In particular, $\tau$-degeneration can always be found amongst pieces of level $\gg \log(\modu(\He)^{-1})$. To prove a priori bounds for Herman rings $\He$ in $\HRspace_{d_0, d_\infty,\theta}$, we will switch between pieces of $\overline{\He}$ and intervals of a boundary component of $\He$ of sufficiently deep level depending on $\modu(\He)$.

\begin{proof}[Proof of Theorem \ref{main-theorem-01}]
    Let $f \in \HRspace_{d_0, d_\infty,\theta}$. Let $\He$ the Herman ring of $f$ and denote by $\mu>0$ the conformal modulus of $\He$. By Corollary \ref{old-bounds}, it is sufficient to prove the theorem when $\mu < \mu_0$ for some fixed $0<\mu_0<1$.
    
    Let $Y^0$ and $Y^\infty$ be the connected components of $\RS \backslash \overline{\He}$ containing $0$ and $\infty$ respectively. Denote the boundary components of $\He$ by 
    \[
    H^0 := \partial Y^0, \quad \text{and} \quad H^\infty := \partial Y^\infty.
    \]
    Let $\tau$, $m$, $\mathbf{K}$, and $N$ be constants from Theorem \ref{amplification}, and let $C$ be the constant from Proposition \ref{bounded-type}. It is sufficient to show that every interval $I$ in $H^0 \cup H^\infty$ of length $\leq l_N$ must have width $W_\tau(I) \leq C \mathbf{K}$. 
    
    Let $M \in \N$ be such that 
    \begin{equation}
        \label{ineq:level-modulus}
    l_{M+1} \leq \mu < l_{M}.
    \end{equation}
    Pick the threshold $\mu_0$ to be small enough such that $M > N + 2m$. All the combinatorial intervals and pieces considered below will be of level $\geq N$, and similar to the shallow-deep treatment in Sections \S\ref{sec:trading}--\ref{sec:amplifying-tau-degeneration}, they will be distinguished into two:
    \begin{description}
        \item[Herman scale] $N \leq n < M$,
        \item[Siegel scale] $n \geq M$.
    \end{description}
    Note that these scales coincide with the ones introduced in \S\ref{ss:outline} and \S\ref{sss:modulus-Siegel-scale}.
        
    \begin{lemma}
        If there is a $[K,\tau]$-wide combinatorial interval $I^\bullet \subset H^\bullet$ at the Siegel scale for some $\bullet \in \{0,\infty\}$ and $K \geq \mathbf{K}$, then there is a $[2K,\tau]$-wide combinatorial interval $J^\bullet \subset H^\bullet$ of level at least $N$.
    \end{lemma}
    
    \begin{proof} 
    By applying Theorem \ref{amplification} in Case \ref{case:boundary-ring}, we can obtain from $I^\bullet$ a $[2 K, \tau]$-wide combinatorial interval $J^\bullet \subset H^\bullet$ of level $\geq N_1 - m > N$.
    \end{proof}
    
    To amplify degeneration about intervals at the Herman scale, we will thicken them to pieces of $\overline{\He}$, amplify via Theorem \ref{amplification} in Case \ref{case:herman-ring}, and convert pieces to intervals to obtain more degenerate intervals.
    
    \begin{lemma}
    \label{non-abyssal}
        If there is a $[K,\tau]$-wide combinatorial interval $I^\bullet \subset H^\bullet$ at the Herman scale for some $\bullet \in \{0,\infty\}$ and $K \geq \mathbf{K}$, then there is a $[2 K, \tau]$-wide combinatorial interval $J^\dagger \subset H^\dagger$ at the Siegel scale for some $\dagger \in \{0,\infty\}$.
    \end{lemma} 
    
    \begin{proof}
    Let $I \subset \overline{\He}$ be the combinatorial piece such that $I \cap H^\bullet = I^\bullet$. The piece $I$ is also at the Herman scale and $[K,\tau]$-wide in $\overline{\He}$. By inductively applying the Amplification Theorem, we obtain an infinite sequence of combinatorial pieces $J_1, J_2,\ldots$ where each $J_i$ is $[2^i K, \tau]$-wide.
    
    Let $t \geq 2$ be a fixed integer that is to be determined later. By compactness, it is impossible for every piece in $\{J_{it}\}_{i \geq 1}$ to be at the Herman scale. Let $j \geq 1$ be the smallest integer such that $J_{jt}$ is at the Siegel scale. The piece $J:=J_{jt}$ has $\tau$-width $W_\tau(J) \geq 2^{jt} K \geq 2^t K$. Note that the level $n_1$ of $J$ must satisfy 
    \begin{equation}
        \label{ineq:M-n1}
        M \leq n_1 < M+tm.
    \end{equation}
        
    Denote the horizontal sides of $J$ by $P^0:= J \cap H^0$ and $P^\infty := J \cap H^\infty$. For each $\dagger \in \{0,\infty\}$, we denote by $Q^\dagger$ the union of $P^\dagger$ and both of its neighboring combinatorial intervals of level $n_1+1$, and by $R^\dagger$ the union of $(\tau P^\dagger)^c$ and both of its neighboring combinatorial intervals of level $n_1+1$.
    
    \begin{claim1}
        The width of curves in $\mathcal{F}_\tau (J)$ that cross through (intersect both horizontal sides of) any component of $\overline{\tau J \backslash J}$ is at most some absolute constant $C_1>0$. 
    \end{claim1}
    
    \begin{proof}
        Indeed, suppose $A$ is one of the two components of $\overline{\tau J \backslash J}$. As a conformal rectangle, the width of $A$ is equal to $|A|/\mu$. Note that
        \[
        \frac{|A|}{\mu} = \frac{\tau-1}{2}\cdot \frac{l_{n_1}}{\mu} \leq \frac{\tau-1}{2}\cdot \frac{l_{n_1}}{l_{M+1}} \leq \frac{C}{2}(\tau-1),
        \]
    where the first inequality follows from (\ref{ineq:level-modulus}) and the second follows from (\ref{ineq:M-n1}). As there are two possible $A$'s to consider, the claim follows by taking $C_1=C\cdot(\tau-1)$.
    \end{proof}
    
    \begin{claim2}
        The width of curves in $\mathcal{F}_\tau (J)$ that do not restrict to curves joining $Q^\sharp$ and $R^\flat$ for some $\sharp, \flat \in \{0,\infty\}$ is at most some constant $C_2(t)>0$. 
    \end{claim2} 
    
    \begin{proof}
        If a curve in $\mathcal{F}_\tau(J)$ does not have a subcurve joining some $Q^\sharp$ and $R^\flat$, then it must have a subcurve that is proper in $A$ and connects the vertical sides of $A$, where $A$ is one of the four level $n_1+1$ combinatorial pieces of $\overline{\He}$ next to $J$ or $(\tau J)^c$. The width $w_A$ of proper curves in $A$ connecting the vertical sides of $A$ satisfies
    \[
        w_A = \frac{\mu}{l_{n_1+1}} \leq \frac{l_M}{l_{n_1+1}} \leq C^{tm+1},
    \]
    where the first inequality follows from (\ref{ineq:level-modulus}) and the second is from (\ref{ineq:M-n1}). As there are four possible $A$'s to consider, our claim follows from taking $C_2=4C^{tm+1}$.
    \end{proof}
    
    From both claims above, there is some $\dagger \in \{ 0, \infty \}$ such that the width $W\left(Q^\dagger, R^\dagger\right)$ of curves joining $Q^\dagger$ and $R^\dagger$ satisfies
    \[
    W\left(Q^\dagger, R^\dagger\right) \geq \frac{W_\tau(J)-C_1-C_2(t)}{2} \geq 2^{t-1}K - \frac{C_1+C_2(t)}{2}.
    \]
    By replacing $\mathbf{K}$ with a higher constant depending on $t$ if necessary, we have
    \[
    W\left(Q^\dagger, R^\dagger\right) \geq 2^{t-2}K.
    \]
    There is an absolute constant $s \in \N$ such that for any combinatorial subinterval $J^\dagger$ of $Q^\dagger$ of level $n_1+s$, the piece $(\tau J^\dagger)^c$ contains $R^\dagger$. Therefore, there is a level $n_1 + s$ combinatorial subinterval $J^\dagger \subset Q^\dagger$ such that
    \begin{align*}
    W_\tau(J^\dagger) \geq \frac{|J^\dagger|}{|Q^\dagger|} \cdot W\left(Q^\dagger, R^\dagger\right) \succ 2^t K.
    \end{align*}
    Finally, we can pick $t$ to be sufficiently high such that $J^\dagger$ is $[2K,\tau]$-wide.
    \end{proof}

    Suppose for a contradiction that on one of the boundary components, say $H^\infty$, there exists an interval $I^\infty$ of length $\leq l_N$ and $\tau$-width at least $CK$ where $K\geq \mathbf{K}$. Then, $I^\infty$ admits a $[K,\tau]$-wide combinatorial subinterval $I_0^\infty \subset I^\infty$ of level $\geq N$. The two lemmas above imply that there is an increasing sequence of positive integers $\{i_j\}_{j\in \mathbb{N}}$ and $[2^{i_j}K, \tau]$-wide combinatorial intervals $I^{\bullet}_{i_j} \subset H^{\bullet}$ for all $j \in \mathbb{N}$ for some common $\bullet \in \{ 0, \infty \}$. This would contradict Proposition \ref{main-prop} (1) and thus conclude the proof of Theorem \ref{main-theorem-01}.
\end{proof}


\section{Construction of Herman curves} 
\label{sec:construction-of-herman-curves}

Endow the space $\rat_d$ of all degree $d = d_0 + d_\infty -1$ rational maps equipped with the topology of uniform convergence on compact sets. In this section, we will obtain Herman quasicircles in $\HQspace_{d_0, d_\infty,\theta}$ as limits of degenerating Herman rings in $\HRspace_{d_0, d_\infty,\theta}$. Towards the end, we show that such Herman quasicircles can be prescribed with arbitrary combinatorics.

Throughout this section, we will denote by $\D(x,r)$ the Euclidean disk centered at a point $z$ with radius $r$, and by $\mathbb{A}(r,R)$ the round annulus $\{r < |z| < R\}$ of inner and outer radii $r$ and $R$. For brevity, we will also encode the data $(d_0,d_\infty,\beta(\theta))$ with the symbol $\clubsuit$.

\subsection{Precompactness}
\label{ss:precompactness}
Given rational maps $f$ and $g$, we write $f \sim g$ to denote that $f$ and $g$ are conformally conjugate. Note that a M\"obius transformation preserves the space $\HRspace_{d_0,d_\infty,\theta}$ by conjugation if and only if it is a linear map $z \mapsto \lambda z$. One consequence of a priori bounds is the following theorem.

\begin{theorem}
\label{precompactness}
    Denote by $\He_f$ the Herman ring of a rational map $f$ whenever there is a unique one. For any $\mu>0$ and $N \in \N$, the quotient space 
    \[
    \left\{ f \in \rat_{d} \: | \: f \in \HRspace_{d_0,d_\infty,\theta} \text{ where } \beta(\theta)\leq N \text{ and } \modu(\He_f) < \mu \right\}/_\sim
    \]
    is precompact.
\end{theorem}

The following lemma will serve as a key ingredient in the proof.

\begin{lemma}[Bounded shape about $0$ and $\infty$]
    \label{bounded-shape-lemma}
    Let $f \in \HRspace_{d_0,d_\infty,\theta}$. The union of the inner (resp. outer) boundary component of the Herman ring of $f$ and all the inner (resp. outer) bubbles of generation $1$ is contained in some round annulus $\mathbb{A}(\varepsilon r,r)$ where $0<\varepsilon<1$ depends only on $\clubsuit$.
\end{lemma}

    \begin{proof}
        We will prove the lemma for the inner boundary and inner bubbles. The treatment for the outer case is analogous. Denote by $Y^0$ the connected component of the complement of the Herman ring containing $0$, and by $Y^0_1$ the component of $f^{-1}(Y^0)$ that is contained in $Y^0$. Let $H^0:= \partial Y^0$ and $H^0_1:=\partial Y^0_1$. By conjugating with a linear map, we can assume that the maximum Euclidean distance between $0$ and a point on $H^0_1$ is $1$. It is sufficient to find a lower bound $\varepsilon$ on $\dist(0,H^0_1)$.
        
        Denote by $\zeta$ a point on $H^0$ that is closest to $0$, by $I_\zeta$ the level $2$ combinatorial interval on $H^0$ centered at $\zeta$, and by $I'_\zeta:= (f|_{H^0})^{-1}(I_\zeta)$ the lift of $I_\zeta$ inside $H^0$. Let $\kappa$ be the harmonic measure of $I_\zeta$ in $Y^0$ about $0$. As $f: Y^0_1 \to Y^0$ is a degree $d_0$ covering map that is branched only at $0$, the harmonic measure of $I'_\zeta$ on $Y^0_{1}$ about $0$ is $\kappa/d_0$. Since $Y^0_{1} \subset Y^0$ and $H^0 \subset H^0_1$, then the harmonic measure of $I'_\zeta$ in $Y^0$ about $0$ is at least $\kappa/d_0$. Note that since $l_2 < \max\{\theta,1-\theta\}$, the intervals $I'_\zeta$ and $I_\zeta$ must be disjoint. As such, $\kappa$ must be bounded above by
        \begin{equation}
        \label{eqn:harm-meas-01}
            \kappa < \left( 1+ \frac{1}{d_0}\right)^{-1} < 1.
        \end{equation}
        
        By assumption, the Euclidean diameter of $H^0$ is greater than $1$. Since the conjugacy $\phi: H^0 \to \T$ between $f|_{H^0}$ and $R_\theta|_\T$ is a $K(\clubsuit)$-quasisymmetry, every connected component of $I_\zeta \backslash \{\zeta\}$ has diameter greater than some $L_1=L_1(\clubsuit)>0$. As $H^0$ is a quasicircle, there is also some small $L_2=L_2(\clubsuit)>0$ such that $H^0 \backslash I_\zeta$ is disjoint from the disk $\D(\zeta,L_2)$. Together with (\ref{eqn:harm-meas-01}), this implies that $\zeta$ cannot be arbitrarily close to $0$, that is, $\dist(0,H^0) > \varepsilon'$ for some $\varepsilon'=\varepsilon'(\clubsuit)>0$.

        The outer boundary of every inner bubble of generation $1$ is contained in $H^0_1$, and its harmonic measure in $Y^0_1$ about $0$ is simply the constant $1/d_0$. Using a similar argument, we conclude that every inner bubble of generation $1$ is of distance at least some constant $\varepsilon(\clubsuit)>0$ away from $0$.
    \end{proof}

\begin{lemma}
\label{normalized-conjugacy}
    Let $f\in \HRspace_{d_0,d_\infty,\theta}$. There is a $K$-quasiconformal map $\phi: \RS \to \RS$ such that the following properties hold.
    \begin{enumerate}[label=\textnormal{(\arabic*)}]
        \item $\phi$ maps the Herman ring $\He$ of $f$ to some annulus $\mathbb{A} = \mathbb{A}(r, re^{2\pi\modu(\He)})$;
        \item $\phi$ is conformal in $\He$;
        \item $\phi|_{\overline{\He}}$ is a conjugacy between $f|_{\overline{\He}}$ and rigid rotation $R_\theta|_{\overline{\mathbb{A}}}$; 
        \item $\phi$ fixes $0$ and $\infty$;
        \item $K$ depends only on $\clubsuit$.
    \end{enumerate}
\end{lemma}
\begin{proof}
    By Theorem \ref{main-theorem-01}, it is immediate that there is a map $\phi: \overline{\He} \to \overline{\mathbb{A}}$ satisfying (1)-(3) that restricts to a $K'(\clubsuit)$-quasisymmetric map from $\He$ to $\partial A$. By Lemma \ref{bounded-shape-lemma}, the control of $\partial\He$ relative to $0$ and $\infty$ allows us to extend $\phi$ to a global quasiconformal map satisfying (4) and (5).
\end{proof}

\begin{proof}[Proof of Theorem \ref{precompactness}]
    Let $f\in \HRspace_{d_0,d_\infty,\theta}$ be a rational map such that its Herman ring $\He$ has modulus $\modu(\He) < \mu$. Denote by $H^0$ and $H^\infty $ the inner and outer boundary components of $\He$. By conjugating $f$ with a linear map, assume that the maximum Euclidean distance between $0$ and a point in $H^0$ is $1$.
    
    From Definition \ref{main-definition}, the rational map $f$ must be of the form
    \[
    f(z) = \lambda z^{d_0} \frac{(z-z_1) \ldots (z-z_{d_\infty -1})}{(z-p_1)\ldots(z-p_{d_0 -1})},
    \]
    where $\mathcal{Z} := \{z_1, \ldots, z_{d_\infty -1} \}$ and $\mathcal{P} := \{p_1, \ldots, p_{d_0 -1} \}$ are the sets of zeros and poles of $f$ respectively. To prove precompactness, it is sufficient to show that there exists some $\varepsilon=\varepsilon(\clubsuit,\mu)>0$ such that 
    \begin{enumerate}[label=\textnormal{(\roman*)}]
        \item $\mathcal{Z} \cup \mathcal{P} \subset \mathbb{A}(\varepsilon ,\varepsilon^{-1})$,
        \item $\dist(\mathcal{Z},\mathcal{P}) > \varepsilon$, and
        \item $\varepsilon < |\lambda| < \varepsilon^{-1}$.
    \end{enumerate}
    
    From our choice of normalization, the outer boundary $H^\infty$ must contain some point $w$ such that $|w| \leq e^{2\pi \mu}$. Indeed, if otherwise, $\He$ would contain the annulus $\{1\leq |z|\leq e^{2\pi \mu}\}$ which would contradict the assumption that $\modu(\He) < \mu$. As a consequence of Lemma \ref{bounded-shape-lemma}, there is some $\varepsilon_1 = \varepsilon_1(\clubsuit,\mu)>0$ such that
    \begin{equation}
    \label{eqn:bdd-sph-dist}
    f^{-1}(\overline{\He}) \subset \mathbb{A} (\varepsilon_1,\varepsilon_1^{-1}).
    \end{equation}
    Since the zeros and poles are enclosed by bubbles of generation $1$, we obtain (i).
    
    Next, (ii) follows directly from the claim below. 
    
    \begin{claim}
        There is some $\varepsilon_2=\varepsilon_2(\clubsuit,\mu)>0$ such that $\dist\left(\overline{\He}, \mathcal{Z} \cup \mathcal{P}\right) > \varepsilon_2$.
    \end{claim} 
    
    \begin{proof}
        Let us pick a pole $p \in \mathcal{P}$. The treatment for zeros is analogous. Recall the notation $H^0_1$, $Y^0$ and $Y^0_1$ used in the proof of Lemma \ref{bounded-shape-lemma}. Let $c \in H^0$ be the critical point that is the root of the inner bubble $B$ of generation $1$ that encloses $p$. Let $\phi$ be the $K(\clubsuit)$-quasiconformal map from Lemma \ref{normalized-conjugacy}. We can normalize $\phi$ such that it maps the inner boundary $H^0$ to the unit circle and the critical value $f(c)$ to $1$.
        
        Let $\gamma$ be the straight segment $[0,1]$ and let $D'$ be the closure of the left half plane minus $\D$. By construction, the annulus $A':= \RS \backslash (D' \cup \gamma)$ has modulus equal to some universal constant $\kappa>0$. Let $A$ (resp. $D$) be the unique lift of $A'$ (resp. $D'$) under $\phi \circ f$ that intersects the bubble $B$. See Figure \ref{fig:last-claim}.
        
        \begin{figure}
        \centering
        \begin{tikzpicture}
    \node[anchor=south west,inner sep=0] (image) at (0,0) {\includegraphics[width=1\linewidth]{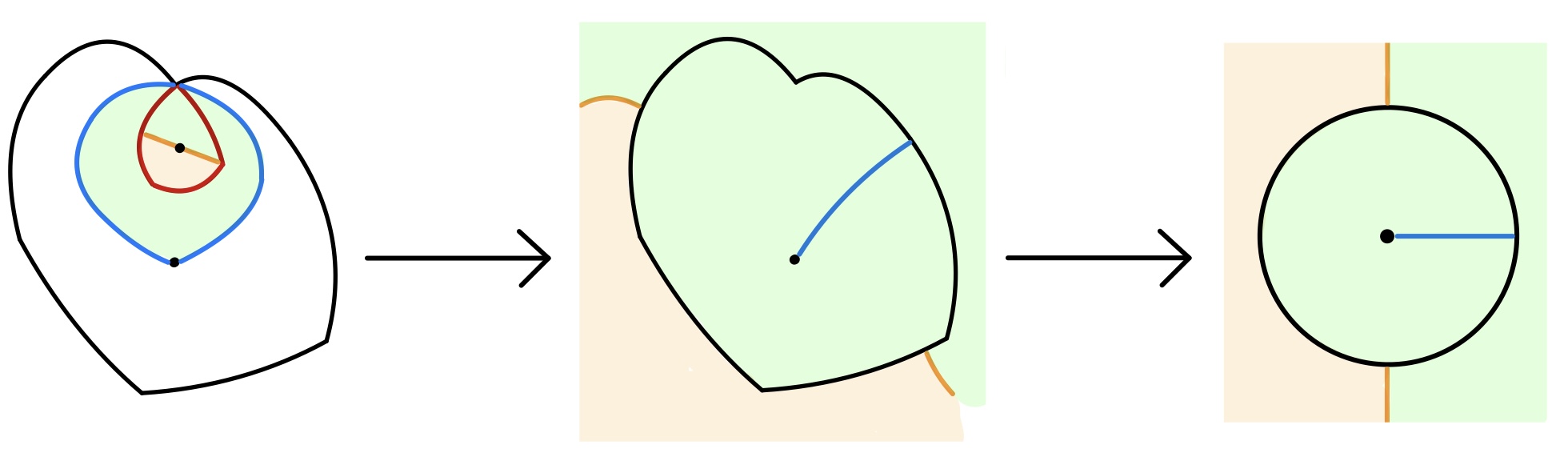}};
    \begin{scope}[
        x={(image.south east)},
        y={(image.north west)}
    ]
        \node [black, font=\bfseries] at (0.03,0.89) {\footnotesize{$H^0$}};
        \node [yellow!30!red, font=\bfseries] at (0.11,0.635) {\footnotesize{$D$}};
        \node [green!60!black, font=\bfseries] at (0.075,0.60) {\footnotesize{$A$}};
        \node [blue!70!white, font=\bfseries] at (0.12,0.87) {\footnotesize{$c$}};
        \node [black, font=\bfseries] at (0.11,0.715) {\footnotesize{$p$}};
        \node [black, font=\bfseries] at (0.11,0.39) {\footnotesize{$0$}};
        
        \node [black, font=\bfseries] at (0.29,0.37) {$f$};
        
        \node [black, font=\bfseries] at (0.42,0.89) {\footnotesize{$H^0$}};
        \node [yellow!30!red, font=\bfseries] at (0.41,0.16) {\footnotesize{$f(D)$}};
        \node [green!60!black, font=\bfseries] at (0.48,0.70) {\footnotesize{$f(A)$}};
        \node [black, font=\bfseries] at (0.48,0.39) {\footnotesize{$0$}};
        \node [blue!70!white, font=\bfseries] at (0.61,0.73) {\footnotesize{$f(c)$}};
        
        \node [black, font=\bfseries] at (0.70,0.37) {$\phi$};
        
        \node [yellow!30!red, font=\bfseries] at (0.82,0.18) {\footnotesize{$D'$}};
        \node [green!60!black, font=\bfseries] at (0.90,0.62) {\footnotesize{$A'$}};
        \node [blue!70!white, font=\bfseries] at (0.93,0.44) {\footnotesize{$\gamma$}};
        \node [blue!70!white, font=\bfseries] at (0.98,0.48) {\footnotesize{$1$}};
        \node [black, font=\bfseries] at (0.88,0.43) {\footnotesize{$0$}};
    \end{scope}
\end{tikzpicture}
        \caption{Construction of the annulus $A$ surrounding $D$.}
        \label{fig:last-claim}
\end{figure}
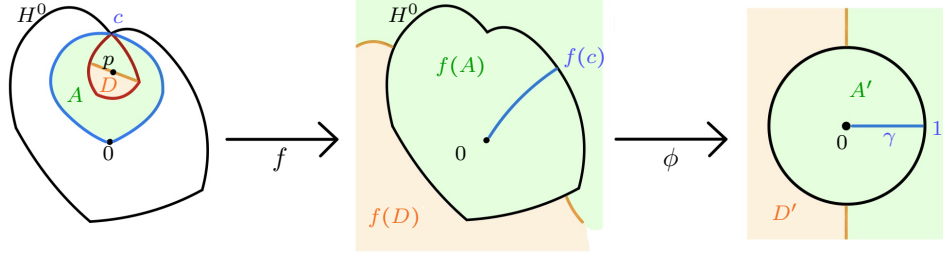
        
        Since $\phi$ maps $\left(Y^0,0\right)$ to $(\D,0)$, the harmonic measure of $f(D) \cap H^0$ in $Y^0$ about $0$ is at least some $\delta(\clubsuit)>0$. Therefore, the harmonic measure of $D \cap H_1^0$ in $Y^0_1$ about $0$ is at least $\delta/d_0$. Combined with (\ref{eqn:bdd-sph-dist}), the diameter of $D$ must be bounded above by some $\delta'(\clubsuit,\mu)>0$. Since $\modu(A) \geq \kappa/K$, we can apply Teichm\"uller estimates (cf. \cite[\S3]{A06}) and conclude that the distance between the two boundary components of $A$ is at least some constant $\varepsilon_2(\clubsuit,\mu)>0$. Finally, as $A$ separates the pole $p$ from $\overline{\He}$, $\dist\left(\overline{\He},p\right) > \varepsilon_2$.
    \end{proof}
    
    The claim and (\ref{eqn:bdd-sph-dist}) imply that every $w \in \mathcal{Z}\cup \mathcal{P}$ satisfies $\varepsilon_2 \leq |1-w| \leq 1+\varepsilon_1^{-1}$. Moreover, as $f(1)$ lies on the inner boundary $H^0$, then $\varepsilon_1 \leq |f(1)| \leq \varepsilon^{-1}_1$. These two observations imply (iii), and we are done.
\end{proof}

\begin{remark}
\label{generalization}
    With similar proof, we can show the compactness of the moduli space of rational maps in $\mathcal{X}_{d_0,d_\infty,\theta}$ as well as the moduli space of degree $d$ polynomials having a bounded type Siegel disk whose boundary contains all free critical points.
\end{remark}

\subsection{Degenerating Herman rings} 
\label{ss:degenerating-herman-rings}
Consider the limit space 
\[
\HRlimitspace_{d_0,d_\infty,\theta} := \overline{\HRspace_{d_0,d_\infty,\theta}} \backslash \HRspace_{d_0,d_\infty,\theta} \: \subset \: \rat_{d_0+d_\infty-1}.
\]
Theorem \ref{precompactness} implies that $\HRlimitspace_{d_0,d_\infty,\theta}/_\sim$ is compact. Another consequence of a priori bounds is that we are finally able to establish a formal relation between the two spaces $\HRspace_{d_0,d_\infty,\theta}$ and $\HQspace_{d_0,d_\infty,\theta}$.

\begin{corollary}
\label{limiting}
     $\HRlimitspace_{d_0,d_\infty,\theta}$ is contained in $\HQspace_{d_0,d_\infty,\theta}$.
\end{corollary}

\begin{proof}
    Suppose $f_n \to f$ for some sequence of rational maps $f_n \in \mathcal{H}_{d_0,d_\infty,\theta}$. We will show that the limit $f$ must lie in $\mathcal{H}_{d_0,d_\infty,\theta} \cup \HQspace_{d_0,d_\infty,\theta}$. 
    
    Due to uniform convergence, both $0$ and $\infty$ remain superattracting fixed points for $f$ of local degrees $d_0$ and $d_\infty$ respectively. In particular, the Julia sets $J(f_n)$ must all be contained in $\mathbb{A}(\varepsilon, \varepsilon^{-1})$ for some $0< \varepsilon<1$ independent of $n$, and that the moduli $\mu_n$ of the Herman rings $\He_n$ of $f_n$ are bounded above by $\frac{1}{\pi} \log \frac{1}{\varepsilon}$. Moreover, for sufficiently high $n$, $f_n$ has a free critical point $c_n$ of the same local degree independent of $n$ located on the inner boundary of $\He_n$, and $c_n \to c$ where $\varepsilon \leq |c| \leq \varepsilon^{-1}$.
    
    By Lemma \ref{normalized-conjugacy}, every $f_n$ admits a $K(\clubsuit)$-quasiconformal map $\phi_n : \RS \to \RS$ that is conformal in $\He_n$, fixes $0$ and $\infty$, maps $c_n$ to $1$, and restricts to a conjugacy between $f_n|_{\overline{\He_n}}$ and the rigid rotation $R_\theta$ on the closed annulus $A_n := \overline{\mathbb{A} (1, e^{2 \pi \mu_n})}$. By the compactness of normalized $K$-quasiconformal maps, $\phi_n$ has a subsequence converging to a $K$-quasiconformal map $\phi$ which fixes $0$ and $\infty$ and maps $c$ to $1$.
    
    By passing to a further subsequence, suppose $\mu_n \to \mu$ for some limit $\mu \geq 0$. As $n\to \infty$, $A_n$ converges in the Hausdorff topology to the closed annulus $A := \overline{\mathbb{A} (1,e^{2\pi\mu})}$ on which we have the conjugacy:
    \[
    R_\theta = \lim_{n\to\infty} \phi_{n} f_{n} \phi_{n}^{-1} = \phi f \phi^{-1}.
    \]
    Moreover, $\overline{\He_{n}}$ converges to $\Hq := \phi^{-1}(A)$. Since all free critical points of $f_n$ lie on $\partial \He_n$, then all free critical points of $f$ also lie on $\partial\Hq$. In particular, if $\mu =0$, then $\Hq$ must be a Herman quasicircle and thus $f \in \HQspace_{d_0, d_\infty,\theta}$. Else, $\Hq$ is the closure of a Herman ring of $f$ of modulus $\mu >0$ and thus $f \in \HRspace_{d_0, d_\infty,\theta}$.
\end{proof}

Corollary \ref{main-corollary} then follows from the corollary above. In the proof, notice that $\Hq$ is independent of any choice of convergent subsequence taken. In particular, we have simultaneously shown:

\begin{corollary}
\label{continuity}
    For $f \in \overline{\mathcal{H}_{d_0,d_\infty,\theta}}$, let $\Hq_f$ denote either the closure of the Herman ring of $f$ or the Herman quasicircle of $f$. Then, $f \mapsto \Hq_f$ is continuous in the Hausdorff topology.
\end{corollary}

Recall that the combinatorics of $\Hq_f$ can be encoded by elements of the space $\mathcal{C}_{d_0,d_\infty}$. (See Definition \ref{comb-definition}.) At last, we prove a stronger version of Theorem \ref{main-theorem-02}.

\begin{theorem}
    $\HRlimitspace_{d_0,d_\infty,\theta} \to \mathcal{C}_{d_0,d_\infty}, f \mapsto \text{comb}(f)$ is a continuous surjection.
\end{theorem}

\begin{proof}
    Continuity of comb$(\cdot)$ follows directly from Corollary \ref{continuity}. 
    
    Pick any arbitrary combinatorial data $\mathcal{C} \in \mathcal{C}_{d_0,d_\infty}$. By Theorem \ref{wang}, there is a rational map $f_1 \in \HRspace_{d_0,d_\infty,\theta}$ with a Herman ring $\He_1$ of modulus $1$ with $\text{comb}(f)=\mathcal{C}$. By deforming the complex structure of $\He_1$ (see \cite[\S6.1]{BF14}), we obtain a real analytic family of rational maps $\{f_t \}_{0<t \leq 1}$ in $\HRspace_{d_0, d_\infty,\theta}$ where each $f_t$ has a Herman ring $\He_t$ of modulus $t$ with the same combinatorics $\mathcal{C}$.
    
    From Theorem \ref{precompactness}, by appropriately normalizing $f_t$, there is a sequence $\{t_n\}_{n\in\N}$ in $(0,1)$ such that as $n\to \infty$, the modulus $t_n$ converges to $0$ and $f_{t_n}$ converges to a rational map $f$ of degree $d_0+d_\infty-1$. Clearly, $f$ cannot lie in $\HRspace_{d_0,d_\infty,\theta}$ because otherwise it would contradict the continuity of the moduli of Herman rings guaranteed in Corollary \ref{continuity}. Therefore, by Corollary \ref{limiting}, $f$ must lie on $\HRlimitspace_{d_0,d_\infty,\theta}$ and it has the same combinatorics $\mathcal{C}$.
\end{proof}


\section{Unicritical Herman curves}
\label{sec:unicritical-herman-curves}

Fix $d_0,d_\infty \geq 2$. For every bounded type irrational number $\theta \in (0,1)$, Theorem \ref{main-theorem-02} guarantees the existence of a rational map in $\HQspace_{d_0,d_\infty,\theta}$ possessing a unique free critical point which, by conjugating with a linear map, can be assumed to be at $z=1$. Elementary computation yields the following explicit formula.

\begin{proposition}
\label{unicritical-formula}
    Suppose that $F_c \in \rat_d$ has critical points at $0$, $\infty$, and $1$ with local degrees $d_0$, $d_\infty$, and $d:= d_0+d_\infty-1$ respectively, and that $F_c(0)=0$, $F_c(\infty)=\infty$, and $F_c(1)=c \in \C^*$. Then,
    \[
    F_c(z) := - c \: \cfrac{ \displaystyle\sum_{j=d_0}^{d} \binom{d}{j} \cdot (-z)^j}{ \displaystyle\sum_{j=0}^{d_0-1} \binom{d}{j} \cdot (-z)^j}.
    \]
\end{proposition}

\begin{proof}
    The rational map $F_1(z) := c^{-1}F_c(z)$ has superattracting fixed points at $0$, $\infty$, and $1$ with local degrees $d_0$, $d_\infty$, and $d$ respectively. From the behaviour at $0$ and $\infty$, the map $F_1$ is of the form $z^{d_0} \frac{p(z)}{q(z)}$ where $p$ is a degree $d_\infty-1$ polynomial and $q$ is a degree $d_0-1$ polynomial. Let us present $F_1$ as
    \[
    F_1(z) = - \: \frac{ (-z)^{d} + \sum_{j={d_0}}^{d-1} a_j (-z)^{j} }{ \sum_{j=0}^{d_0-1} a_j (-z)^{j} }
    \]
    for some coefficients $a_0,a_1,\ldots,a_{d-1}$. The map $g(z) := 1-F_1(-z)$ is of the form
    \[
    g(z) = \frac{ z^d + \sum_{j=0}^{d-1} a_j z^j }{ \sum_{j=0}^{d_0-1} a_j z^j}.
    \]
    From the behaviour of $F_1$ at $1$, $z=-1$ must be a zero of $g$ of order $d$. Thus, the numerator of $g$ must be divisible by $(z+1)^{d} = \sum_{j=0}^{d} \binom{d}{j} z^j$. This implies that $a_j = \binom{d}{j}$ for every $j$, and we are done.
\end{proof}

Denote by $c(\theta)$ the parameter such that $f_\theta:= F_{c(\theta)}$ lies in $\HQspace_{d_0,d_\infty,\theta}$. For every bounded type $\theta$, the uniqueness of the parameter $c(\theta)$ will follow from the sequel \cite{Lim23}, in which we justify combinatorial rigidity. Specific examples of $f_\theta$ can be found in Example \ref{eg:general-blaschke} when $d_0=d_\infty$, and in Figure \ref{fig:asymmetric-herman-quasicircle} when $d_0=2$, $d_\infty=4$, and $\theta$ is the golden mean.

We will conclude this paper with a discussion of Renormalization Theory for unicritical Herman quasicircles in two different viewpoints. These will be further explored in the near future.

First, we will consider the first return maps of $f_\theta$ near its critical point. Let us denote by $\Hq$ the Herman curve of $f_\theta$. For any $j \in \Z$, let $c_j := (f_\theta|_\Hq)^j(1)$. For brevity, we again write $\clubsuit=\left(d_0,d_\infty,\beta(\theta)\right)$. We say that a pointed domain $(U,x)$ has \emph{bounded shape} if the inner and outer radii of $U$ about $x \in U$ are comparable.

\begin{theorem}[First return maps]
    For every $n \in \N$, there exists a degree $d$ branched covering map
    \begin{equation}
        \label{1st-return}
        f_\theta^{q_n} : (U_n,c_0) \to (V_n,c_{q_n})
    \end{equation}
    between two pointed disks with bounded shape and diameter $ \asymp |c_0-c_{q_n}|$. All bounds depend only on $\clubsuit$.
\end{theorem}

\begin{proof}
    Let $\phi: (\RS, \Hq) \to (\RS, \T)$ be a $K(\clubsuit)$-quasiconformal map that conjugates $f|_\Hq$ and $R_\theta|_\T$. In this proof, any implicit constants involved depend on $\clubsuit$.
    
    Consider nested open intervals $I_n := (c_{2q_n},c_{-q_n}) \subset \tilde{I}_n := (c_{q_n-q_{n-1}},c_{q_{n-1}})$ in $\Hq$. Let $V_n$ (resp. $\tilde{V}_n$) denote the unique disk such that its image under $\phi$ is the round disk intersecting $\T$ orthogonally on the interval $\phi(I_n)$ (resp. $\phi(\tilde{I}_n)$). By construction, $V_n$ has bounded shape about $c_{q_n}$ with diameter $\asymp |c_0-c_{q_n}|$. Moreover, $V_n$ is contained in $\tilde{V}_n$ and the annulus $\tilde{V}_n \backslash \overline{V_n}$ has modulus $\asymp 1$.
    
    Denote by $U_n$ (resp. $\tilde{U}_n$) the lift of $V_n$ (resp. $\tilde{V}_n$) under $f_\theta^{q_n}$ containing $c_0$. The interval $\tilde{I}_n$ is precisely the maximal interval such that the only critical value of $f_\theta^{q_n}$ on $\tilde{I}_n$ is $c_{q_n}$. As such, $f_\theta^{q_n}: \tilde{U}_n \to \tilde{V}_n$ is a degree $d$ covering map branched only at $c_0$. By Koebe distortion theorem, since $\modu\left(\tilde{V}_n \backslash \overline{V_n}\right) \asymp 1$, the pointed disk $(U_n,c_0)$ also has bounded shape. Since $U_n \cap \Hq = (c_{q_n}, c_{-2q_n})$ has diameter comparable to $I_n$, then $U_n$ must have diameter $\asymp |c_0 - c_{q_n}|$.
\end{proof}

\begin{figure}
    \centering
\begin{tikzpicture}
    \node[anchor=south west,inner sep=0] (image) at (0,0) {\includegraphics[width=0.9\linewidth]{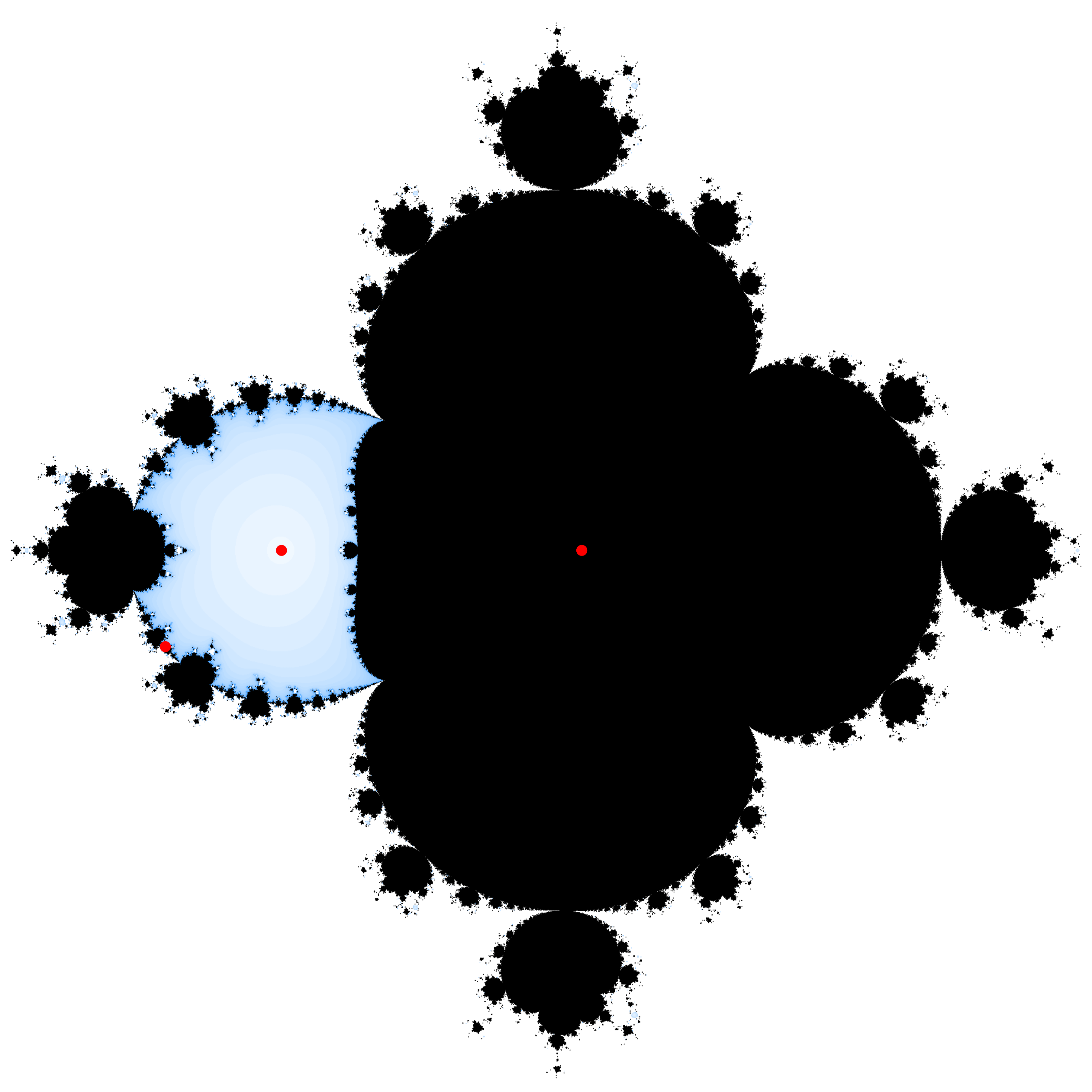}};
    \begin{scope}[
        x={(image.south east)},
        y={(image.north west)}
    ]
    
    \node [red, font=\bfseries] at (0.257,0.475) {\footnotesize{$0$}};
    \node [red, font=\bfseries] at (0.53,0.475) {\footnotesize{$1$}};
    \draw [red] (0.151,0.4074) -- (0.08,0.36);
    \node [red, font=\bfseries] at (0.08,0.34) {\footnotesize{$c_*$}};
    \end{scope}
\end{tikzpicture}
    \caption{The parameter space $\{F_c\}_{c \in \C^*}$ for $(d_0,d_\infty)=(2,4)$ is shown above. The region $\{c \in \C^* \: | \: F_c^n(1)\to \infty\}$ is colored white, the region $\{c \in \C^* \: | \: F_c^n(1)\to 0\}$ is colored blue, and the non-escaping locus $\mathcal{M}_{2,4}$ is shown in black.}
    \label{fig:parspace1}
\end{figure}
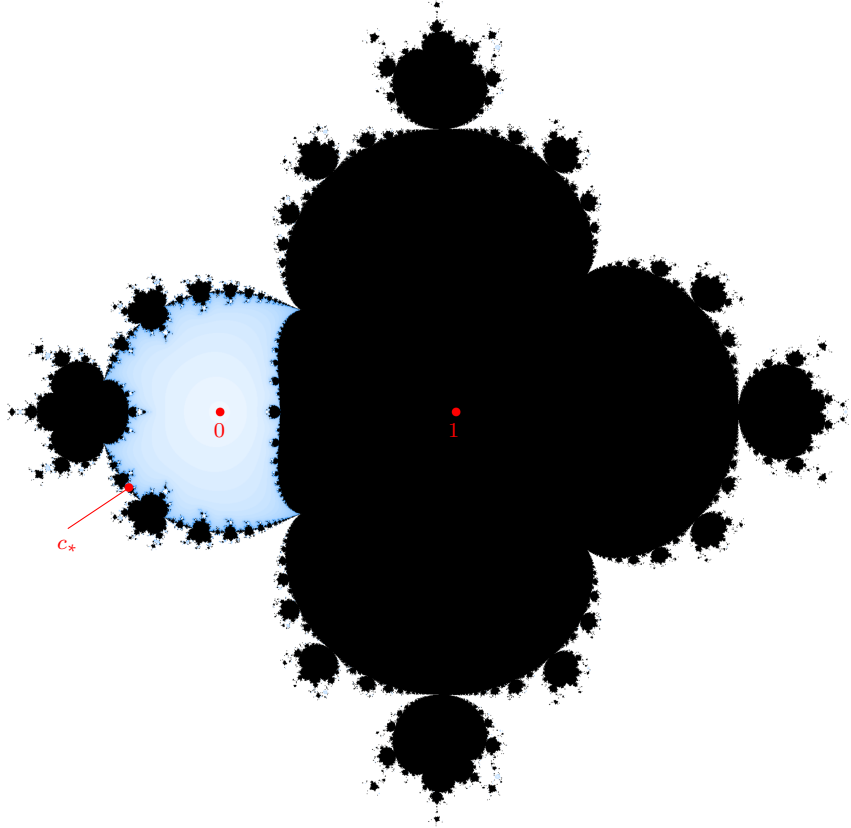

We shall equip the space of branched coverings between pointed domains with the Carath\'eodory topology in the sense of \cite[\S5]{McM94}. The theorem above has the following immediate implication.

\begin{corollary}[Rescaled limits of (\ref{1st-return})]
\label{rescaled-limits}
    Consider the sequence of affine maps $A_n(z)=(z-c_0)/(c_{q_n}-c_0)$. The sequence of rescaled first return maps
    \[
    A_n \circ f_{\theta}^{q_n} \circ A_n^{-1}: \left(A_n(U_n),0\right) \to \left(A_n(V_n), 1\right)
    \]
    is precompact in the Carath\'eodory topology. Moreover, all limit maps are degree $d$ covering maps branched only at $0$.
\end{corollary}

In light of Corollary \ref{rescaled-limits}, we conduct a more rigorous study of the first return maps to prove various scaling properties for unicritical Herman quasicircles in \cite{Lim23}. These include universality and self-similarity about the critical point, similar to critical circle maps and quadratic Siegel disks. See \cite{McM98, dF99, dFdM99}.

\begin{figure}
    \centering
    \includegraphics[width=0.78\linewidth]{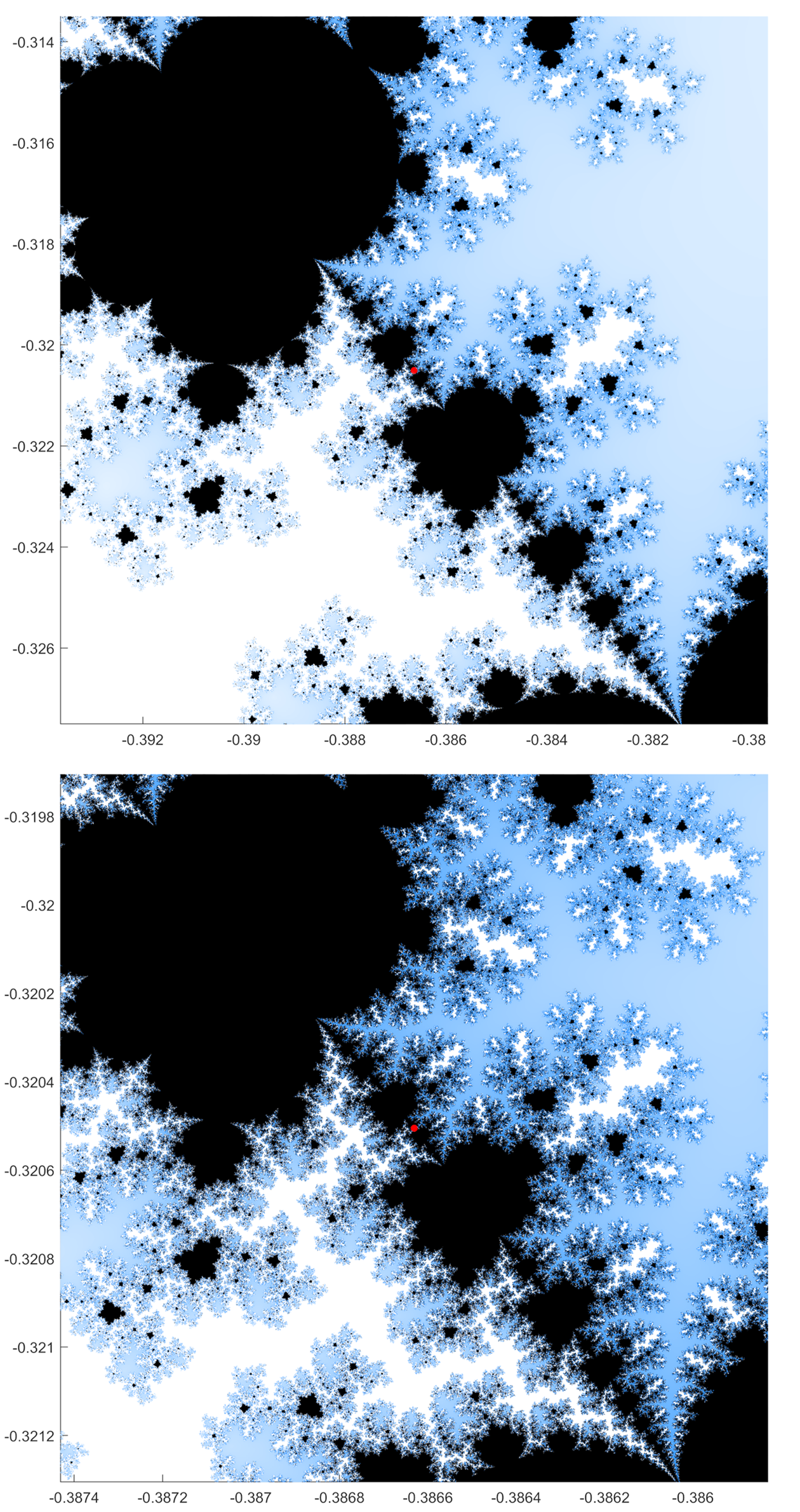}

    \caption{Magnifications of Figure \ref{fig:parspace1} by different scales about the parameter $c_* \approx -0.386631-0.320505i$ that is marked in red. The Julia set of $F_{c_*}$ is displayed in Figure \ref{fig:asymmetric-herman-quasicircle}.}
    \label{fig:parspace2}
\end{figure}

Secondly, let us consider the parameter space $\{F_c\}_{c \in \C^*}$ in Proposition \ref{unicritical-formula}. The non-escaping locus 
\[
\mathcal{M}_{d_0,d_\infty}=\left\{c \in \C^* \: | \: F_c^n(1) \not\to 0 \text{ and } F_c^n(1) \not\to \infty \right\}
\]
is displayed in Figure \ref{fig:parspace1}. Let $c_* := c(\theta_*)$ where $\theta_*$ is the golden mean irrational number. Displayed in Figure \ref{fig:parspace2} are the blow ups of the parameter space near $c_*$, which provide preliminary evidence of asymptotic self-similarity of $\mathcal{M}_{d_0,d_\infty}$ near $c_*$. As we compare with \cite{L99, DLS}, we expect that this phenomenon is a consequence of hyperbolicity of some appropriate renormalization operator. 

\begin{conjecture}
\label{conj:self-similarity}
    For every stationary type irrational number $\theta = [0;N,N,N,\ldots]$, the non-escaping locus $\mathcal{M}_{d_0,d_\infty}$ is asymptotically self-similar at $c(\theta)$. There is a hyperbolic renormalization operator associated to it.
\end{conjecture}

In a follow up work \cite{Lim24}, we construct a hyperbolic renormalization operator associated to unicritical Herman curves. However, Conjecture \ref{conj:self-similarity} remains open.

\appendix


\section{Near-degenerate regime}
\label{sec:near-degenerate-regime}

\subsection{Extremal width}
\label{ss:extremal-width}
Given a family $\mathcal{G}$ of curves on a Riemann surface $S$, we denote by $W(\mathcal{G})$ the extremal width of $\mathcal{G}$. We list without proof a number of fundamental results on extremal width. (See \cite{A06} and the appendix in \cite{KL05} for details.)

\begin{proposition}[Parallel Law]
    \label{parallel law}
    For any two curve families $\mathcal{G}_1$ and $\mathcal{G}_2$,
    \[
    W(\mathcal{G}_1 \cup \mathcal{G}_2) \leq W(\mathcal{G}_1) + W(\mathcal{G}_2).
    \]
    Equality is achieved when $\mathcal{G}_1$ and $\mathcal{G}_2$ have disjoint support.
\end{proposition}

We say that a curve family $\mathcal{G}$ \emph{overflows} another curve family $\mathcal{H}$, denoted by $\mathcal{H} < \mathcal{G}$, if every curve in $\mathcal{G}$ contains a curve in $\mathcal{H}$ (curves in $\mathcal{G}$ are longer and fewer). We also say that $\mathcal{H}$ is a \emph{restriction} of $\mathcal{G}$ if $\mathcal{G}$ overflows $\mathcal{H}$ but not any proper subfamily of $\mathcal{H}$ (curves in $\mathcal{G}$ are longer, but not more nor fewer).

Denote by $x \oplus y$ the harmonic sum $(x^{-1}+y^{-1})^{-1}$.

\begin{proposition}[Series Law]
    \label{series law}
    Suppose a curve family $\mathcal{G}$ overflows two disjoint curve families $\mathcal{G}_1$ and $\mathcal{G}_2$. Then,
    \[
    W(\mathcal{G}) \leq W(\mathcal{G}_1) \oplus W(\mathcal{G}_2).
    \]
\end{proposition}

The following proposition allows us to convert harmonic sums into friendlier expressions. 

\begin{proposition}
    \label{harmonicsum}
    For any positive numbers $a_1, \ldots, a_n$,
    \[ 
    \bigoplus_{i=1}^n a_i \leq \min \left\{ a_1, \ldots , a_n, \frac{1}{n}\max_i a_i, \frac{1}{n^2} \sum_{i=1}^n a_i \right\}.
    \]
\end{proposition}

Extremal width is invariant under conformal maps. More generally, we have the following transformation rule.

\begin{proposition}
    \label{transformation law}
    Let $f: U\to V$ be a holomorphic map between two Riemann surfaces and $\mathcal{G}$ be a family of curves in $U$. Then,
    \[
    W(f(\mathcal{G})) \leq W(\mathcal{G}).
    \]
    If $f$ is at most $d$ to $1$, then
    \[
    W(\mathcal{G}) \leq d \cdot W(f(\mathcal{G})).
    \]
\end{proposition}

A \emph{(conformal) rectangle} $P$ on a surface $S$ is the image of a continuous map $\phi: [0,m]\times[0,1] \to \overline{S}$ that restricts to a conformal embedding in the interior. The vertical sides of a rectangle $P$ are $\phi(\{0\}\times [0,1])$ and $\phi(\{m\} \times [0,1])$, and the horizontal sides of $P$ are $\phi([0,m]\times \{0\})$ and $\phi([0,1]\times\{m\})$. A curve in $P$ is called \emph{vertical} if it connects the two horizontal sides of $P$. The \emph{vertical foliation} of $P$ is defined to be the collection of curves
\[
\mathcal{F}(P) := \{ \phi\left( \{t\} \times (0,1) \right) \: | \: t \in (0,m)\}.
\]
The \emph{width} of $P$ is 
\[
W(P) := W(\mathcal{F}(P))= m.
\]
We say that $P$ \emph{crosses} a curve $\gamma$ if every vertical curve in $P$ intersects $\gamma$.

\begin{proposition}[Non-Crossing Principle]
\label{non-crossing-principle}
    If a pair of rectangles $P_1$ and $P_2$ on $S$ has width $W(P_1), W(P_2)> 1$, then they never cross, i.e., there exist disjoint leaves $\gamma_1 \in \mathcal{F}(P_1)$ and $\gamma_2 \in \mathcal{F}(P_2)$.
\end{proposition}

Suppose a rectangle $P$ of width $m$ has width greater than $8$. The \emph{buffers} of $P$ are subrectangles of $P$ of the form $\phi( [0,\varepsilon)\times (0,1)$ or $\phi( (m-\varepsilon,m] \times (0,1))$ for some $\varepsilon \leq 4$. A direct consequence of the non-crossing principle is the following proposition.

\begin{proposition}[{\cite[Lemma 2.14]{KL05}}]
    Every pair of rectangles $P_1$ and $P_2$ of width greater than $8$ admits subrectangles $P_1^{new}$ and $P_2^{new}$ obtained by removing some buffers such that $P_1^{new}$ and $P_2^{new}$ have disjoint vertical sides.
\end{proposition}

When $S$ has boundary, we say that a curve $\gamma : (0,1) \to S$ is \emph{proper} if it has well-defined endpoints $\gamma(0)$ and $\gamma(1)$ contained in $\partial S$. For any disjoint subsets $I$ and $J$ of $\partial S$, we denote by $\mathcal{F}_S (I,J)$ and $W_S(I,J)$ the family of proper curves in $S$ that connect $I$ and $J$, and its width respectively. When $S$ is a Jordan disk, the width $W_S (I,J)$ can be estimated as follows.

\begin{proposition}[Log-Rule, {\cite[Lemma 2.5]{DL22}}]
\label{log-rule}
    Suppose $S$ is a Jordan disk and suppose its boundary $\partial S$ is partitioned into four intervals $I_1, I_2, I_3, I_4$, labelled cyclically. Denote by $|I_i|$ the harmonic measure of $I_i$ in $S$ about a point $x \in S$.
    \begin{enumerate}[label=\textnormal{(\arabic*)}]
        \item If $\min(|I_1|, |I_3|) \geq \min(|I_2|, |I_4|)$, then
        \[
        W_S(I_1,I_3) \asymp \log \frac{\min \{|I_1|, |I_3|\}}{\min\{|I_2|, |I_4|\}} + 1;
        \]
        \item Otherwise,
        \[
        W_S(I_1,I_3) \asymp \left( \log \frac{\min\{|I_2|, |I_4|\}}{\min\{|I_1|, |I_3|\}} + 1 \right)^{-1}.
        \]
    \end{enumerate}
\end{proposition}

Given a compact subset $I$ of $S$, denote by $W(S,I)$ the extremal width of the family $\mathcal{F}(S,I)$ of proper curves in $S\backslash I$ connecting $I$ and $\partial S$. We will state important near-degenerate tools from \cite{KL05} in a manner best suited to our context.

\begin{lemma}[Quasi-Additivity Law]
\label{quasi-additivity-law}
    Suppose $S$ is a topological disk in $\mathbb{C}$ and $A_1, \ldots, A_n$ be pairwise disjoint non-empty compact connected subsets of $S$. Let
    \[
    X :=  W\left(S, \bigcup_{i=1}^n A_i\right), \quad
    Y := \sum_{i=1}^n W(S, A_i), \quad
    Z_i := W\left(S \backslash \bigcup_{j\neq i} A_j, A_i\right)  \text{ for } i=1,\ldots,n. 
    \]
    Then, there exists some $K =K(n) > 0$ such that
    \[
    Y \geq K \quad \Longrightarrow \quad \max\{X, Z_1, \ldots, Z_n\} \geq \frac{Y}{\sqrt{2n}}.
    \]
\end{lemma}

\begin{lemma}[Covering Lemma]
\label{covering-lemma}
    Let $\Lambda \Subset \Lambda' \subset U$ and $B \Subset B' \subset V$ be two nests of simply connected domains and $f: (U, \Lambda', \Lambda) \to (V,B', B)$ be a branched covering map with degrees $\text{deg}(f: \Lambda' \to B') = d$ and $\text{deg}(f: U\to V) = D$. For all $\kappa >1$, there is some $\mathbf{K} = \mathbf{K}(\kappa, D)>0$ such that if $W(U,\Lambda) = K > \mathbf{K}$, then either
    $$
    W(B', B) > \kappa K, \qquad \text{or} \qquad W(V, B) > (2\kappa d^2)^{-1} K.
    $$
\end{lemma}
\vspace{0.05in}

\subsection{Canonical lamination}
\label{ss:canonical-lamination}

Consider an open hyperbolic Riemann surface $S$ with a finite number of boundary components. We allow the presence of finitely many punctures, which are separate from the ideal boundary $\partial S$. We will survey the fundamental properties of the \emph{canonical lamination} $\mathcal{F}_{\textnormal{can}}(S)$ of $S$ following Kahn's work \cite{K06}. The canonical lamination captures the near-degeneracy of $S$ induced by components of $\partial S$ that are very close to one another. Let us first sketch the construction.

Let $\pi: \D \to S$ be the universal cover of $S$. Since $\partial S$ is non-empty, the limit set $\Lambda \subset \partial \D$ of $\pi_1(S)$ is a Cantor set. For every component $\tilde{I} \subset$ of $\D \backslash \Lambda$, $\pi$ extends continuously to a universal covering $\tilde{I} \to I$ for some component $I$ of $\partial S$. Two proper curves $\gamma_0$ and $\gamma_1$ are \emph{properly homotopic} in $S$ if there is a homotopy $\gamma_{t}$, $t\in [0,1]$ between $\gamma_0$ and $\gamma_1$ such that each $\gamma_t$ is also a proper curve in $S$. An \emph{arc} in $S$ is a proper homotopy class of proper curves in $S$.

Consider a non-trivial arc $\alpha$ in $S$ connecting two (not necessarily distinct) components $I$ and $J$ of $\partial S$. Let $\tilde{\alpha}$ be a lift of $\alpha$ under $\pi$; it connects $\tilde{I}$ and $\tilde{J}$, which are some lifts of $I$ and $J$ respectively. Let us identify $\D$ with the structure of a conformal rectangle with horizontal sides $\tilde{I}$ and $\tilde{J}$. Kahn observed that removing buffers of width $1$ gives us a subrectangle that can be pushed forward by $\pi$ to a new conformal rectangle $\mathcal{R}_{\textnormal{can}}(S;\alpha)$ with horizontal sides contained in $I$ and $J$.

The \emph{canonical arc diagram} $\mathcal{A}_{\textnormal{can}}(S)$ is the set of non-trivial arcs $\alpha$ in $S$ such that the canonical rectangle $\mathcal{R}_{\textnormal{can}}(S;\alpha)$ is non-empty. The removal of buffers in the construction ensures that these rectangles are pairwise disjoint. The cardinality of $\mathcal{A}_{\textnormal{can}}(S)$ is at most a constant depending only on the Euler characteristic of $S$. 

We define the \emph{thick-thin decomposition} and the \emph{canonical lamination} of $S$ by
\[
    \textnormal{TTD}(S) := \bigcup_{\alpha \in \mathcal{A}_{\textnormal{can}}(S)} \mathcal{R}_{\textnormal{can}}(S;\alpha) \quad \text{and} \quad
    \mathcal{F}_{\textnormal{can}}(S) := \bigcup_{\alpha \in \mathcal{A}_{\textnormal{can}}(S)} \mathcal{F}_{\textnormal{can}}(S;\alpha),
\]
respectively, where $\mathcal{F}_{\textnormal{can}}(S;\alpha)$ is the vertical foliation of the canonical rectangle $\mathcal{R}_{\textnormal{can}}(S;\alpha)$. Every leaf of $\mathcal{F}_{\textnormal{can}}(S;\alpha)$ is represented by $\alpha \in \mathcal{A}_{\textnormal{can}}(S)$. If a proper arc $\alpha$ is not in $\mathcal{A}_{\textnormal{can}}(S)$, we set $\mathcal{F}_{\textnormal{can}}(S;\alpha)$ to be the empty lamination.

Below, we list without proof a number of fundamental properties of the canonical lamination. Firstly, it is maximal in the following sense.

\begin{proposition}[{\cite[Lemma 3.2]{K06}}]
    \label{maximality}
    For any proper family $\mathcal{F}$ of curves in $S$ represented by a single arc $\alpha$,
    \[
    W(\mathcal{F}) -2 \leq W\left(\mathcal{F}_{\textnormal{can}}(S;\alpha)\right).
    \]
    In other words, up to an additive constant, curves in $\mathcal{F}$ are vertical curves inside of the rectangle $\mathcal{R}_{\textnormal{can}}(S;\alpha)$.
\end{proposition}

Consider two hyperbolic Riemann surfaces $U$ and $V$ with boundary. The fact that the thick-thin decomposition is defined via the universal cover yields the following property.

\begin{proposition}[{\cite[Lemma 3.3]{K06}}]
    \label{naturality}
    For any holomorphic covering map $f: U \to V$ of finite degree, 
    \[
         \textnormal{TTD}(U) = f^*\textnormal{TTD}(V) \quad \text{and} \quad \mathcal{F}_{\textnormal{can}}(U) = f^*\mathcal{F}_{\textnormal{can}}(V) .
    \]
\end{proposition}

When $U \subset V$, the restriction of $\mathcal{F}_{\textnormal{can}}(V)$ onto $U$ results in a proper lamination in $U$. By Proposition \ref{maximality}, the width of this restriction will be bounded above by the canonical lamination of $U$ after some buffers are removed. This can be formulated more precisely as follows.
    
\begin{proposition}[{\cite[Lemma 3.10]{K06}}]
    \label{domination}
    When $U \subset V$, there exists a sublamination $\mathcal{L} \subset \mathcal{F}_{\textnormal{can}}(V)$ such that
    \[
        W\left(\mathcal{F}_{\textnormal{can}}(V)\right) - C \leq W(\mathcal{L})
    \]
    for some constant $C>0$ depending only on the Euler characteristic of $U$ with the following property. For every leaf $\gamma$ of $\mathcal{L}$, every component of $\gamma \cap U$ is either
    \begin{enumerate}[label=\textnormal{(\arabic*)}]
        \item a homotopically trivial proper curve in $U$, or
        \item a vertical curve in $\mathcal{R}_{\textnormal{can}}(U;\alpha)$ for some $\alpha \in \mathcal{A}_{\textnormal{can}}(U)$.
    \end{enumerate}
\end{proposition}

In application, the Riemann surface $S$ we consider in \S\ref{sec:amplifying-tau-degeneration}--\ref{sec:loss-of-horizontal-width} is of the form $U\backslash K$ where $U \subset \C$ is a disk and $K$ is a non-empty compact subset of $U$. We say that a proper curve in $U \backslash K$ is \emph{horizontal} if both of its endpoints are on $K$, and \emph{vertical} if it connects a point on $K$ and a point on $\partial U$. We define the \emph{canonical horizontal} (resp. \emph{vertical}) \emph{lamination} $\mathcal{F}^h_{\textnormal{can}}(U, K)$ (resp. $\mathcal{F}^v_{\textnormal{can}}(U, K)$) on $U\backslash K$ to be the lamination consisting of all horizontal (resp. vertical) leaves of $\mathcal{F}_{\textnormal{can}}(U \backslash K)$. Similarly, we define the \emph{horizontal} and \emph{vertical thick-thin decomposition} $\textnormal{TTD}^h(U,K)$ and $\textnormal{TTD}^v(U,K)$ of $U\backslash K$ respectively.

\bibliographystyle{alpha}

\bibliography{bibliography}

\begin{thebibliography}{CDKvS22}

\bibitem[ABC04]{ABC04}
Artur Avila, Xavier Buff, and Arnaud Ch\'eritat.
\newblock {Siegel} disks with smooth boundaries.
\newblock {\em Acta Math.}, 193(1):1--30, 2004.

\bibitem[Ahl06]{A06}
Lars Ahlfors.
\newblock {\em Lectures on Quasiconformal Mappings}, volume~38.
\newblock American Mathematical Society, 2nd edition, 2006.

\bibitem[AKLS09]{AKLS}
Artur Avila, Jeremy Kahn, Mikhail Lyubich, and Weixiao Shen.
\newblock Combinatorial rigidity for unicritical polynomials.
\newblock {\em Ann. Math.}, 170(2):783--797, 2009.

\bibitem[ALS11]{ALS}
Artur Avila, Mikhail Lyubich, and Weixiao Shen.
\newblock Parapuzzle of the {Multibrot} set and typical dynamics of unimodal maps.
\newblock {\em J. Eur. Math. Soc.}, 13(1):27--56, 2011.

\bibitem[BBM]{BBM}
Araceli Bonifant, Xavier Buff, and John Milnor.
\newblock Antipode preserving cubic maps: tongues and the ring locus.
\newblock Manuscript in preparation.

\bibitem[BBM18]{BBM18}
Araceli Bonifant, Xavier Buff, and John Milnor.
\newblock Antipode preserving cubic maps: The fjord theorem.
\newblock {\em Proc. Lond. Math. Soc.}, 116(3):670--728, 2018.

\bibitem[BF14]{BF14}
Bodil Branner and N\'uria Fagella.
\newblock {\em Quasiconformal Surgery in Holomorphic Dynamics}.
\newblock Cambridge University Press, 2014.

\bibitem[CDKvS22]{CDKvS}
Trevor Clark, Kostiantyn Drach, Oleg Kozlovski, and Sebastian van Strien.
\newblock The dynamics of complex box mappings.
\newblock {\em Arnold Math. J.}, 8(2):319--410, 2022.

\bibitem[dF99]{dF99}
Edson de~Faria.
\newblock Asymptotic rigidity of scaling ratios for critical circle mappings.
\newblock {\em Ergod. Theory Dyn. Syst.}, 19(4):995--1035, 1999.

\bibitem[dFdM99]{dFdM99}
Edson de~Faria and Welington de~Melo.
\newblock Rigidity of critical circle mappings {II}.
\newblock {\em J. Amer. Math. Soc.}, 13(2):343--370, 1999.

\bibitem[DH93]{DH93}
Adrien Douady and John Hubbard.
\newblock {A proof of Thurston's topological characterization of rational functions}.
\newblock {\em Acta Math.}, 171(2):263--297, 1993.

\bibitem[DL22]{DL22}
Dzmitry Dudko and Mikhail Lyubich.
\newblock Uniform a priori bounds for neutral renormalization, 2022.
\newblock \href{https://arxiv.org/abs/2210.09280}{arXiv.2210.09280}.

\bibitem[DL23]{DL23}
Dzmitry Dudko and Mikhail Lyubich.
\newblock {MLC at Feigenbaum points}, 2023.
\newblock \href{https://arxiv.org/abs/2309.02107}{arXiv.2309.02107}.

\bibitem[DLS20]{DLS}
Dzmitry Dudko, Mikhail Lyubich, and Nikita Selinger.
\newblock Pacman renormalization and self-similarity of the {Mandelbrot} set near {Siegel} parameters.
\newblock {\em J. Amer. Math. Soc.}, 33(3):653--733, 2020.

\bibitem[dMvS93]{dMvS93}
Welington de~Melo and Sebastian van Strien.
\newblock {\em One-Dimensional Dynamics}.
\newblock Springer-Verlag, 1993.

\bibitem[Dou87]{D87}
Adrien Douady.
\newblock Disques de {S}iegel et anneaux de {H}erman.
\newblock In {\em S\'eminaire Bourbaki : volume 1986/87, expos\'es 669-685}, number 152-153 in Ast\'erisque, pages 4, 151--172. Soci\'et\'e math\'ematique de France, 1987.

\bibitem[Ere20]{E20}
Alexander Eremenko.
\newblock Open problems session at the conference {On Geometric Complexity of Julia Sets II}.
\newblock \href{https://www.impan.pl/konferencje/bcc/2020/20-juliasets2/problem-session.pdf}{https://www.impan.pl/konferencje/bcc/2020/20-juliasets2/problem-session.pdf}, 2020.

\bibitem[Fat20]{F20}
Pierre Fatou.
\newblock Sur les \'equations fonctionnelles.
\newblock {\em Bull. Soc. Math. Fr.}, 48:208--314, 1920.

\bibitem[Ghy84]{G84}
{\'E}tienne Ghys.
\newblock Transformations holomorphes au voisinage d'une courbe de {Jordan}.
\newblock {\em C. R. Acad. Sci. Paris S{\'e}r. I Math.}, 298(16):385--388, 1984.

\bibitem[GM05]{GM05}
John Garnett and Donald Marshall.
\newblock {\em Harmonic Measure}.
\newblock New Mathematical Monographs. Cambridge University Press, 2005.

\bibitem[Her86]{H86}
Michael Herman.
\newblock Conjugaison quasi-sym\'etrique des homeomorphismes analytiques du cercels \`a des rotations.
\newblock Manuscript, 1986.

\bibitem[Kah06]{K06}
Jeremy Kahn.
\newblock A priori bounds for some infinitely renormalizable quadratics: I. bounded primitive combinatorics, 2006.
\newblock \href{https://arxiv.org/abs/math/0609045}{arXiv:math/0609045}.

\bibitem[KL05]{KL05}
Jeremy Kahn and Mikhail Lyubich.
\newblock The quasi-additivity law in conformal geometry.
\newblock {\em Ann. Math.}, 169(2):561--593, 2005.

\bibitem[KL08]{KL08}
Jeremy Kahn and Mikhail Lyubich.
\newblock A priori bounds for some infinitely renormalizable quadratics: {II}. decorations.
\newblock {\em Ann. Sci. \'Ec. Norm. Sup\'er.}, 41:57--84, 2008.

\bibitem[KL09a]{KL09b}
Jeremy Kahn and Mikhail Lyubich.
\newblock Local connectivity of {Julia} sets for unicritical polynomials.
\newblock {\em Ann. Math.}, 170(1):413--426, 2009.

\bibitem[KL09b]{KL09a}
Jeremy Kahn and Mikhail Lyubich.
\newblock A priori bounds for some infinitely renormalizable quadratics, {III}. molecules.
\newblock In {\em Complex Dynamics: Families and Friends}, chapter~5, pages 229--256. A K Peters, 2009.

\bibitem[KvS09]{KvS}
Oleg Kozlovski and Sebastian van Strien.
\newblock Local connectivity and quasi-conformal rigidity of non-renormalizable polynomials.
\newblock {\em Proc. London Math. Soc.}, 99(2):275--296, 2009.

\bibitem[Lim23]{Lim23}
Willie~Rush Lim.
\newblock {Rigidity of J-rotational rational maps and critical quasicircle maps}, 2023.
\newblock \href{https://arxiv.org/abs/2308.07217}{arXiv.2308.07217}.

\bibitem[Lim24a]{LimThesis}
Willie~Rush Lim.
\newblock {\em From {Herman} rings to {Herman} curves}.
\newblock PhD thesis, Stony Brook University, 2024.

\bibitem[Lim24b]{Lim24}
Willie~Rush Lim.
\newblock Hyperbolicity of renormalization of critical quasicircle maps, 2024.
\newblock \href{https://arxiv.org/abs/2405.09008}{arxiv.2405.09008}.

\bibitem[Lyu97]{Lyu97}
Mikhail Lyubich.
\newblock Dynamics of quadratic polynomials, {I-II}.
\newblock {\em Acta Math.}, 178(2):189--297, 1997.

\bibitem[Lyu99]{L99}
Mikhail Lyubich.
\newblock {Feigenbaum-Coullet-Tresser} universality and {Milnor}'s hairiness conjecture.
\newblock {\em Ann. Math.}, 149(2):319--420, 1999.

\bibitem[McM]{McM}
Curtis McMullen.
\newblock Simultaneous uniformization of {Blaschke} products.
\newblock Manuscript, \href{https://people.math.harvard.edu/~ctm/papers/home/text/papers/simunif/simunif.pdf}{people.math.harvard.edu/$\sim$ctm/papers/home/text/papers/simunif/simunif.pdf}.

\bibitem[McM94]{McM94}
Curtis McMullen.
\newblock {\em Complex Dynamics and Renormalization}.
\newblock Princeton University Press, 1994.

\bibitem[McM98]{McM98}
Curtis McMullen.
\newblock Self-similarity of {Siegel} disks and {Hausdorff} dimension of {Julia} sets.
\newblock {\em Acta Math.}, 180(2):247--292, 1998.

\bibitem[Mil06]{M06}
John Milnor.
\newblock {\em Dynamics in One Complex Variable}, volume 160 of {\em Annals of Mathematics Studies}.
\newblock Princeton University Press, 3rd edition, 2006.

\bibitem[Pet04]{Pe04}
Carsten~Lunde Petersen.
\newblock On holomorphic critical quasi-circle maps.
\newblock {\em Ergod. Theory Dyn. Syst.}, 24(5):1739--1751, 2004.

\bibitem[Shi87]{S87}
Mitsuhiro Shishikura.
\newblock On the quasiconformal surgery of rational functions.
\newblock {\em Ann. Sci. \'Ec. Norm. Sup\'er.}, 20(1):1--29, 1987.

\bibitem[{\'{S}}wi88]{S88}
Grzegorz {\'{S}}wi{\k{a}}tek.
\newblock On critical circle homeomorphisms.
\newblock {\em Bol. Soc. Bras. Mat.}, 29(2):329--351, 1988.

\bibitem[Thu86]{Th86}
William Thurston.
\newblock {Hyperbolic structures on 3-manifolds I: Deformation of acylindrical manifolds}.
\newblock {\em Ann. Math.}, 124(2):203--246, 1986.

\bibitem[Wan12]{W12}
Xiaoguang Wang.
\newblock A decomposition theorem for {Herman} maps.
\newblock {\em Adv. Math.}, 267:307--359, 2012.

\bibitem[Yan22]{Y22}
Fei Yang.
\newblock {Rational maps with smooth degenerate Herman rings}, 2022.
\newblock \href{https://arxiv.org/abs/2207.06770}{arXiv:2207.06770}.

\bibitem[Zak99]{Z99}
Saeed Zakeri.
\newblock Dynamics of cubic {Siegel} polynomials.
\newblock {\em Comm. Math. Phys.}, 206(1):185--233, 1999.

\bibitem[Zha08]{Z08}
Gaofei Zhang.
\newblock Dynamics of {Siegel} rational maps with prescribed combinatorics, 2008.
\newblock \href{https://arxiv.org/abs/0811.3043}{arXiv.0811.3043}.

\bibitem[Zha11]{Z11}
Gaofei Zhang.
\newblock All bounded type {Siegel} disks of rational maps are quasi-disks.
\newblock {\em Invent. Math.}, 185(2):421--466, 2011.

\end{thebibliography}

\end{document}